\RequirePackage{fix-cm}
\documentclass[smallextended]{svjour3}       

\usepackage[title]{appendix}
\usepackage{amsmath, amsthm, amssymb, amsfonts, mathtools, graphicx, environ}

\usepackage{mathrsfs}  
\usepackage{bm}        

\usepackage{hyperref}
\usepackage{algorithm}
\usepackage{algpseudocode}
\usepackage{indentfirst}
\usepackage{subcaption}
\graphicspath{{./images/}}
\usepackage{placeins}
\usepackage{enumitem}
\usepackage [ short , c2 ]{ optidef }
\usepackage{xargs}                      
\usepackage[pdftex,dvipsnames]{xcolor}  

\usepackage[colorinlistoftodos,prependcaption,textsize=tiny]{todonotes}
\newcommandx{\unsure}[2][1=]{\todo[linecolor=red,backgroundcolor=red!25,bordercolor=red,#1]{#2}}
\newcommandx{\change}[2][1=]{\todo[linecolor=blue,backgroundcolor=blue!25,bordercolor=blue,#1]{#2}}
\newcommandx{\info}[2][1=]{\todo[linecolor=OliveGreen,backgroundcolor=OliveGreen!25,bordercolor=OliveGreen,#1]{#2}}
\newcommandx{\improvement}[2][1=]{\todo[linecolor=Plum,backgroundcolor=Plum!25,bordercolor=Plum,#1]{#2}}
\newcommandx{\thiswillnotshow}[2][1=]{\todo[disable,#1]{#2}}
\DeclareMathOperator*{\argmax}{arg\,max}
\DeclareMathOperator*{\argmin}{arg\,min}

\newcommand{\exclude}[1]{}
\newcommand{\bc}{\mathbf{c}}
\newcommand{\bpi}{\bm{\pi}}
\newcommand{\bfxi}{\bm{\xi}}
\newcommand{\pr}{\mathbb{P}}
\newcommand{\vertex}{\mathcal{V}}
\newcommand{\commodities}{\mathcal{K}}
\newcommand{\commodity}{\ell}

\newcommand{\piselk}{\vertex_k^{sel}}
\newcommand{\kink}{k \in [K]}
\newcommand{\xvalue}{z}
\newcommand{\pisel}{\{\vertex_k^{sel}\}_{\kink}}
\newcommand{\vsel}[1]{\overline{\xvalue}(#1, \pisel)}
\newcommand{\qsel}[1]{\overline{Q}_k(#1, \piselk)}

\newcommand{\dfs}{\Pi}

\newcommand{\pilist}{\vertex_{DSP}}
\newcommand{\piperm}{\vertex_{\text{perm}}}
\newcommand{\pitrial}{\vertex_{\text{trial}}}
\newcommand{\piused}{\vertex_{\text{used}}}
\newcommand{\xlist}{X^{feas}}
\newcommand{\xlistopt}{X^{opt}}
\newcommand{\phaseonexlist}{\xlistopt}
\newcommand{\nodes}{\mathcal{N}}
\newcommand\numberthis{\addtocounter{equation}{1}\tag{\theequation}}
\newcommand{\R}{\mathbb{R}}

\newcommand{\thetakt}{\{\theta_k^t\}_{\kink}}
\newcommand{\thetaip}{\{\hat{\theta}_k\}_{\kink}}
\newcommand{\bestx}{x^{WS}}
\newcommand{\closex}{\overline{x}}
\newcommand{\prooffeasibleregion}{\Pi}
\newcommand{\picure}{\vertex_{cur}}
\newcommand{\someformat}[1]{\texttt{#1}}
\newcommand{\qapp}[1]{\overline{Q}_k(#1)}
\newcommand{\vapp}[1]{\overline{\xvalue}(#1)}

\begin{document}

\title{Accelerating Benders decomposition for solving
a sequence of sample average approximation replications\thanks{This research has been supported by ONR grant N00014-21-1-2574.}}
\titlerunning{Accelerating solution of a sequence of SAA replications}

\author{Harshit Kothari \and James R. Luedtke}


\institute{Harshit Kothari \at
Department of Industrial and Systems Engineering, Wisconsin Institute of Discovery, University
of Wisconsin-Madison, Madison, WI, USA
           \email{hkothari2@wisc.edu}
\and
James R. Luedtke \at
Department of Industrial and Systems Engineering, Wisconsin Institute of Discovery, University
of Wisconsin-Madison, Madison, WI, USA
           \email{jim.luedtke@wisc.edu}}

\maketitle
\begin{abstract}

Sample average approximation (SAA) is a technique for obtaining approximate solutions 
to stochastic programs that uses the average from a random sample to approximate the expected value 
that is being optimized.
Since the outcome from solving an SAA is random, statistical estimates on the optimal value of the true problem can be obtained by solving multiple SAA replications with independent samples.
We study techniques to accelerate the solution of this set of SAA replications, 
when solving them sequentially via Benders decomposition.
We investigate how to exploit similarities in the problem structure,
as the replications just differ in the realizations of the random samples.
Our extensive computational experiments 
provide empirical evidence that our techniques for using information from solving previous replications can significantly reduce the solution time of later replications.

\keywords{Benders decomposition \and Stochastic Programming \and Integer Programming}
\end{abstract}
\section{Introduction}
\label{sec:Introduction}

We study methods for solving two-stage mixed-integer 
stochastic programs with continuous recourse and randomness only in the right-hand side of the second-stage problem:
\begin{equation}
\label{eq:tsp}
 \min_{x \in X}c^\top x + \mathbb{E}_{\bfxi} [Q(x, \bfxi)] , 
\end{equation}
where \( Q(x, \bfxi) \) is the optimal value of the second-stage problem
and is defined by
\[ Q(x, \bfxi):= \min_y \{q^\top y : Wy = h(\bfxi) -T(\bfxi)x, y \ge 0 \} \label{eq:sp_def}
  \numberthis
.\]
Here \(x \in  X \subseteq \mathbb{Z}^q_+ \times \mathbb{R}^{n-q}_+ \) is the 
first-stage decision vector, \( y \in \mathbb{R}^{m}_+ \) is the vector of recourse decisions
and \( \bfxi \) is a random vector.

Unless \( \bfxi \) has a finite and small number of possible realizations, it is usually 
impossible to solve this general form of a stochastic program to optimality
because the expected value is hard to compute.
Indeed, in \cite{dyer2006computational}
the authors prove that two-stage linear stochastic programming 
problems are $\#P$-hard.
Sample average approximation (SAA) \cite{kleywegt2002sample} is an approach for 
obtaining approximate solutions to stochastic programs which approximates the expectation with an average over a finite set of scenarios sampled from the distribution of the random vector.
With \( K \) scenarios \( \left(  \xi_1, \ldots, \xi_K \right)\), the SAA problem is given by:
\begin{equation}
\hat{z}_K = \min_{x \in X}c^\top x + \sum_{k=1}^{K} p_k Q(x,\xi_k). \label{eq:saa} 
\end{equation}
\sloppypar{Let \( \hat{z}_K, \hat{x}_K \) be the optimal value and an optimal solution of problem \eqref{eq:saa}.
As these quantities are random, the multiple-replications procedure (MRP) \cite{mak1999monte,bayraksan2006assessing} has been proposed as a method
to determine a confidence interval on the optimal value of the problem. MRP calculates this confidence interval for a candidate solution by solving 
multiple SAA replications with different independent samples of \( \bfxi \).
When only a fixed set of realizations of the random data is available, \cite{lam2018bounding} propose a method to estimate solution quality
using bootstrap aggregating to generate multiple samples and solving the corresponding SAA replications.
Solving multiple SAA replications with different random samples can also be used to find higher-quality feasible solutions, e.g.,  \cite{sen2016mitigating}. 
Stochastic decomposition \cite{higle1991stochastic,wang2020statistical} and
stochastic approximation \cite{robbins1951stochastic} are attractive alternatives to SAA for {\it continuous} stochastic programs, but are
not applicable to mixed-integer stochastic programs.
With these motivations in mind, we focus on the problem of solving a set of SAA replications of the same underlying stochastic programming instance derived from different samples.}

Our objective in this work is to investigate how information obtained from solving one SAA replication can be used to speed up solution of other replications in order to minimize the total computation time to solve a set of SAA replications. If these SAA replications are solved on a sufficiently large cluster of computers, the wall clock time (as opposed to total computation time) could be minimized by simply solving all the replications in parallel. However, even in this setting, we argue that minimizing total computation time is an appropriate goal, as computing time is typically a limited resource, and  power usage grows with total computation time. In addition, when multiple machines are available, algorithms such as Benders decomposition can be implemented in parallel on these machines when solving a single replication. Thus, solving the SAA replications sequentially and using information from early replications to reduce the solution time of later replications can also lead to reduced wall clock time.

One approach to solving an SAA replication is to solve its deterministic equivalent form \cite{birge2011introduction}. However, this formulation can become too large to solve directly when the the number of scenarios \( (K) \) is large. 
Decomposition methods like  Benders decomposition \cite{Benders1962,van1969shaped} and dual decomposition \cite{caroe1999dual} address this by decomposing the problem and solving a sequence of smaller problems, coordinating the results, and repeating.
In this work, we focus on problems with continuous recourse and randomness appearing only in the right-hand side of the subproblem constraints.
This problem structure is seen in many applications such as
fleet planning \cite{mak1999monte}, telecommunications network design \cite{sen1994network} and
melt control \cite{dupavcova2005melt}. Benders decomposition is a leading technique for solving problems having this structure as it is able to exploit the convexity of the recourse function, and hence we study techniques for reusing information when using
Benders decomposition to solve a set of SAA replications. 

\begin{sloppypar}
Our assumptions on the problem structure imply that the 
dual feasible region 
of the Benders subproblem is fixed across all possible scenarios. Thus, our first proposal is to
reuse dual solutions from previous replications 
to generate Benders cuts for future SAA replications by storing dual solutions in a dual solution pool (DSP). Then, whenever we would normally solve a subproblem to generate a Benders cut, we first check the DSP to see if any dual solutions there define a violated cut, and if so, we add the cut and avoid solving the subproblem.
The idea of reusing stored dual solutions to generate Benders cuts
has been used in different contexts when solving a single stochastic program, such as in stochastic decomposition
\cite{higle1991stochastic,wang2020statistical} 
and Benders decomposition \cite{sakhavand2022subproblem,adulyasak2015benders}. 
This is the first time this technique 
has been used 
in the context of solving multiple different SAA replications.
We support this idea theoretically by estimating the number of solutions that need to be in the DSP to assure that a nearly most-violated cut can be found for a given first-stage solution.
\end{sloppypar}

While using the DSP can reduce time spent solving subproblems to generate Benders cuts, we make several additional contributions that reduce the computation time significantly beyond this. First, in preliminary computational studies
we observed that the DSP tends to grow excessively large as the number of replications increases, as more distinct dual extreme points are discovered across replications. This makes the process of checking the DSP for violated cuts time-consuming.
To address this,
we propose a method for curating the DSP by retaining
only some of the dual solutions in the pool. 
Second,
we propose two techniques for choosing  Benders cuts to include in the Benders main problem at the start of the algorithm.
We tested these methods on two-stage stochastic linear and integer programs on three test problems.
The combination of initialization techniques and DSP
methods reduced the total time taken
to solve these replications by half compared to using the DSP alone.

\begin{sloppypar}
Our work contributes to a growing body of literature investigating techniques for improving methods for solving a sequence of closely related instances of an optimization problem. 
The surveys \cite{bengio2021machine} and \cite{deza2023machine} 
provide an overview of recent research investigating the use of machine learning (ML) to 
learn better methods for solving instances from a family of related instances.
In stochastic programming, the authors in \cite{jia2021benders} train a support vector machines for the binary
classification of the usefulness of a Benders cut and observe
that their model allows for a reduction in the total solving
time for a variety of two-stage stochastic programming instances. 
In \cite{dumouchelle2022neur2sp}, the authors
propose to approximate the second-stage solution value with a feed-forward neural network.
Recent work by \cite{larsen2023fast} also
leverages ML to estimate
the scenario subproblem optimal values.
In \cite{mitrai2023learning}, the authors
develop an ML approach to accelerate
generalized Benders decomposition by estimating the optimal number of cuts 
that should be added to the main problem in the first iteration.
There has also been work on using ML 
to quickly compute primal solutions for stochastic programs 
\cite{bengio2020learning,nair2018learning}.
ML-enhanced Benders decomposition has been 
used to accelerate solution times across various domains, 
including power systems \cite{BOROZAN2024110985},
wireless resource allocation \cite{lee2020accelerating}, 
network design problems \cite{chan2023machine}, 
model predictive control \cite{mitrai2024computationally},
among others.
These studies generally begin by solving optimization problems
offline to collect data,
followed by training an ML model
to find algorithm parameters that speed up future solves.
Our setting is fundamentally different. We solve a sequence of SAA replications of the same stochastic program to obtain
statistical estimates (e.g., for confidence intervals). We do not assume we have solved any related instances prior to
beginning the solution of these SAA replications, and hence there is no opportunity to train an ML model based on
previous instance data to aid in solving these replications. Instead, information must be gathered and reused on the fly as each replication is solved sequentially.
\end{sloppypar}

This paper is organized as follows:
In Section \ref{sec:Benders}, we review the Benders decomposition method, for both two-stage stochastic linear and integer programs.
In Section \ref{sec:techniques},
we present our methods for accelerating Benders decomposition by reusing information from the solution of previous replications.
In Section \ref{sec:Computational study},
we present our results from computational experiments.

%
%

\section{Benders Decomposition}
\label{sec:Benders}

\begin{sloppypar}
In this section, we describe the Benders decomposition method \cite{Benders1962,van1969shaped} for solving the SAA \eqref{eq:saa}.
The dual of the subproblem \eqref{eq:sp_def} for scenario $\kink := \{1,\ldots,K\}$ is given by:
\begin{align}
   \max_{\pi} \{(h(\xi_k) - T(\xi_k)x)^\top\pi :W^\top\pi \le q \}. \label{eq:spdual} 
\end{align}
We denote the dual feasible region as \( \dfs = \{\pi: W^\top\pi \le q\} \), which is independent of the scenario $\kink$.\end{sloppypar}

We make the following assumptions about the stochastic program \eqref{eq:tsp}:
\begin{itemize}
\item The subproblem dual feasible region $\dfs$ is non-empty.
\item Relatively complete recourse:
For every feasible first-stage solution $x \in X$ and every $\bfxi$ in the support of the random variable $\bfxi$, 
there exists a feasible decision to subproblem \eqref{eq:sp_def}.
\end{itemize}
We make the relatively complete recourse assumption mainly to simplify exposition.
In Section \ref{ssub:relatively_complete_recourse},
we introduce extensions of our methods to handle the case where relatively complete recourse does not hold.

Under these assumptions both the primal 
and the dual of the subproblem \eqref{eq:sp_def} have an optimal solution 
and from strong duality, we conclude that their optimal values are equal.
Let \( \vertex \) denote the set of 
all the vertices of \( \dfs \).
Then, 
\begin{align}
  Q(x,\xi_k) &= \max_{\pi} \{(h(\xi_k) - T(\xi_k)x)^\top\pi : \pi \in \vertex \}.
  \label{eq:spvertex}
\end{align}
%
Benders decomposition is based on a reformulation of \eqref{eq:saa}, 
that introduces a new variable \( \theta_k \)
to represent the optimal value of subproblem \( k \) for each scenario  $\kink$. 
Using \eqref{eq:spvertex}, the reformulation is as follows:
\begin{mini}
  {\textrm{\scriptsize $x\in X, \theta$}}{c^\top x + \sum_{k=1}^{K} p_k \theta_k}{}{}
\addConstraint{Ax = b}{}{}
\addConstraint{\theta_k \ge (h(\xi_k) - T(\xi_k) x)^\top \pi,}{\quad}{\pi \in \vertex,  \quad \kink.}
\label{eq:reformulation}
\end{mini}
In the following subsections we discuss how this reformulation is used within  Benders decomposition to solve two-stage stochastic linear programs (LPs) (Section \ref{sub:Stochastic LPs}) and two-stage stochastic IPs with continuous recourse (Section \ref{sub:Stochastic IPs}).

\subsection{Stochastic LPs} 
\label{sub:Stochastic LPs}

The size of the reformulation \eqref{eq:reformulation} depends on the number of vertices of the dual subproblem \( (\vertex) \), which is usually too large to explicitly enumerate.
Thus, a delayed cut generation scheme 
is used to iteratively add these constraints. Specifically, Benders decomposition works with a ``main problem'' which has the form of \eqref{eq:reformulation} but at each iteration $t$ only includes a subset $\vertex_{k,t}$ of the constraints for each scenario $\kink$:
\begin{mini!}
  {\textrm{\scriptsize $x\in X$}}{c^\top x + \sum_{k=1}^{K} p_k \theta_k}{\label{form:MPt}}{\textrm{MP}^t=}
\addConstraint{Ax = b}{}{}
\addConstraint{\theta_k \ge (h(\xi_k) - T(\xi_k) x)^\top \pi_k,}{\quad}{\pi_k \in \vertex_{k, t}, \, \kink. \label{eq:benderscuts}}
\end{mini!} 
The constraints \eqref{eq:benderscuts} in the reformulation are called Benders cuts.

Algorithm \ref{alg:BD} outlines the Benders decomposition algorithm for LPs.
We start by solving the main problem with some initial cuts of the form \eqref{eq:benderscuts} included. For simplicity of exposition, we assume these initial cuts are sufficient to assure problem \eqref{form:MPt} is bounded and hence has an optimal solution, (\( x^t, \thetakt \)). 
Next, we solve subproblems \eqref{eq:sp_def}, with $\xi = \xi_k$, to check for violated cuts
and evaluate the objective for each subproblem $\kink$. 
If this objective value is greater than \( \theta_k^t \), 
we have identified a violated cut.
This leads to the generation of a Benders cut,
defined by the dual solution \( \pi_k^t \). 
After iterating through all the subproblems, and adding violated cuts to the main problem, 
we solve the updated main problem.
This iterative process continues until no further violated cuts are found,
leading to an optimal solution of the original problem.

At each iteration, the main problem objective provides a lower bound, \( L^t \),
to the problem because it is a relaxation to the reformulation \eqref{eq:reformulation}. 
We also obtain an upper bound $U^t$ in each iteration by solving the subproblems at every iteration.
The difference between $U^t$ and $L^t$ can be used as a 
convergence condition, terminating when this difference
falls below a tolerance denoted by \( \epsilon \).

\begin{algorithm} 
\caption{Benders decomposition algorithm.}
\begin{algorithmic}[1]
  \State Initialize \( t \coloneqq 0 \) 
  \State Initialize \( \vertex_{k, 0}\) for all \(\kink\) \Comment{Initialization step}
\Repeat
  \State \( \textrm{cutAdded} \gets \textrm{False} \)
  \State Solve \( MP^t \) (\ref{form:MPt}) and obtain (\(x^t, \thetakt\))
  \State \(L^t \gets c^T x^t + \frac{1}{K} \sum_{k=1}^K \theta_k^t\) \Comment{Calculate lower bound}
  \ForAll {\(\kink\)}
  \State Solve \eqref{eq:sp_def}, with $\xi = \xi_k, x=x^t$ and obtain $Q(x^t,\xi_k)$ and dual solution $\pi_k^t$
    \If{\(Q(x^t,\xi_k) > \theta_k^t\)} 
      \State \(\vertex_{k, t+1} \gets \vertex_{k, t} \cup \{ \pi_k^t \}\)  \Comment{Store dual solution to add cut in next iteration}
      \State \(\textrm{cutAdded} \gets \textrm{True}\)
    \EndIf
  \EndFor
  \State \(U^t \gets c^{\top} x^{t} + \frac{1}{K} \sum_{k=1}^K Q(x^t,\xi_k)\) \Comment{Calculate upper bound}
  \State \(t \gets t+1\)
\Until{\(\textrm{cutAdded} = \textrm{False}\) or \( U^t - L^t \le \epsilon \)} 
\end{algorithmic}
\label{alg:BD}
\end{algorithm}

This version of Benders decomposition 
in which we introduce auxiliary variables $\theta_k$ for every scenario $\kink$
is called the multi-cut version.
An alternate version is the single-cut version \cite{van1969shaped}
in which a single variable is used to represent the epigraph of the expected value of the subproblem objective taken over the full set of scenarios.
We focus on the multi-cut version in this paper, but 
in Section \ref{ssub:single-cut}
we discuss how the methods proposed here can be extended to 
the single-cut version of the algorithm.

\subsection{Stochastic IPs} 
\label{sub:Stochastic IPs}

When dealing with stochastic programs featuring integer variables in the first-stage,
solving the main problem can be computationally expensive because it is an integer program.
Hence, instead of iteratively solving the main problem and adding cuts,
stochastic IPs can alternatively be solved using the branch-and-cut method 
\cite{fortz2009improved,naoum2013interior}.

\begin{algorithm}
\caption{Branch-and-cut for IPs.}
\begin{algorithmic}[1]
  \State Initialize $\nodes \gets \{0\}$, $\bar{z} \gets +\infty$, $(x^*, \{\theta_k^*\}_{\kink}) \gets \emptyset$
    \State Initialize $LP_0$ to be the relaxation of main problem with initial cuts
    \While{$\nodes \neq \emptyset$}
        \State Choose a node $i \in \nodes$,  $\nodes \gets \nodes \setminus \{i\}$
        \State Solve $LP_i$. If feasible, obtain optimal solution $(\hat{x}, \thetaip)$ and optimal value $\hat{z}$
        \If{$LP_i$ feasible and $\hat{z} < \bar{z}$} 
        \If{$\hat{x} \in X$} \Comment{Search for violated Benders cuts}
            \State \( \textrm{cutAdded} \gets \textrm{False} \)
            \ForAll {\(\kink\)}
                \State 
                Solve \eqref{eq:sp_def}, with $\xi = \xi_k, x=\hat{x}$ and obtain $Q(\hat{x},\xi_k)$ and dual solution $\pi_k$
                \If{\(Q(\hat{x},\xi_k) > \hat{\theta}_k\)} 
                    \State Add Benders cut: \(\theta_k \ge (h(\xi_k) - T(\xi_k) x)^\top \pi_k \)
                    \State \(\textrm{cutAdded} \gets \textrm{True}\)
                \EndIf
            \EndFor
            \If{cutAdded = True}
                \State go to step 5
            \Else
            \State Update incumbent solution $x^* \gets \hat{x}$, and $\bar{z}$  \Comment{(\( \hat{x}, \hat{\theta} \)) is feasible}
            \EndIf
        \Else
            \State Partition the problem and update \( \nodes \) \Comment{Branching step}
        \EndIf
        \EndIf
    \EndWhile
    \State \Return $(x^*, \{\theta_k^*\}_{\kink})$
\end{algorithmic}
\label{alg:bnc}
\end{algorithm}

A standard branch-and-cut algorithm for solving two-stage stochastic integer programs
with continuous recourse is outlined in Algorithm \ref{alg:bnc}.
The algorithm begins by adding initial Benders cuts to the main problem.
Starting the algorithm with some initial cuts defined can help speed up convergence of the algorithm \cite{saharidis2011initialization,fortz2009improved}.
The LP relaxation of this main problem
is represented by the root node or \( LP_0 \).
This node is added to the list of candidate branch-and-bound nodes \(\nodes\).

In each iteration, the algorithm selects a node from \(\nodes\) 
and solves the corresponding LP relaxation,
a relaxation of the problem with added constraints from branching.
This LP generates a candidate solution \((\hat{x}, \thetaip) \).
If \(\hat{x}\) violates the integrality constraints (i.e., \(\hat{x} \notin X\)),
the algorithm creates two new subproblems by subdividing the feasible region (branching step).
These two new subproblems are then appended to \( \nodes \).

If \(\hat{x} \in X\), we solve subproblems \eqref{eq:sp_def} with $x=\hat{x}$ and for $\xi=\xi_k$ for each $k \in [K]$ to check if the solution (\( \hat{x}, \thetaip \))
is feasible to the original problem.
If any violated Benders cuts are found, they are added to the main problem,
and the relaxation at that node is solved again.
This process of dynamically adding cuts
can be implemented using a ``Lazy constraint callback" in MIP solvers.
If no violated cuts are identified, then the current solution is feasible.
Whenever we find a feasible solution with a better objective value than the current best solution,
we update the upper bound \(\bar{z}\) and the incumbent solution \(x^*\).
The algorithm terminates when there are no more nodes to explore in \(\nodes\).


In \cite{fortz2009improved}, the authors describe an implementation of Benders decomposition using the branch-and-cut technique.
They highlight the importance of initializing the algorithm 
with a good set of cuts included in the main problem to accelerate convergence.
To achieve this, they first solve the LP relaxation of the original problem using Benders decomposition
and retain all identified cuts in the initial main problem of the IP. 
We follow their approach
but only retain the cuts which are active at the optimal solution of the LP relaxation
to manage the size of the main problem.
%


\section{Techniques for Reusing Information to Accelerate Benders Decomposition}
\label{sec:techniques}

We now consider a setting where we wish to solve a sequence of SAA replications of the form \eqref{eq:saa}, each with an independently drawn set of scenarios. Suppose we wish to solve $M$ such replications, and each having $K$ scenarios, \( \left(  \xi^r_1, \ldots, \xi^r_K \right)\) for $r=1,\ldots,M$. The question we investigate is, when solving the SAA \eqref{eq:saa} associated with a replication $r > 1$ with Benders decomposition, how can we use information obtained from solving replications $1,\ldots,r-1$ to reduce the solution time? 

We first consider a simple warm-start strategy for stochastic IPs: use the optimal first-stage solution (say $\tilde{x}$) from a previous replication as an initial feasible solution in the branch-and-bound process.
When solving a new replication, we evaluate $\tilde{x}$ by solving the subproblems \eqref{eq:sp_def} at $(\tilde{x},\xi_k)$ for all scenarios $\kink$, and pass the resulting solution $(\tilde{x}, {Q(\tilde{x},\xi_k)}_{\kink})$ to the solver.
This can help the solution process by providing an upper bound for pruning nodes.

In the remainder of this section we propose other techniques for reusing information to accelerate Benders decomposition. To do so, we consider the
key computational tasks within Benders decomposition:
solving the main problem and 
solving subproblems to generate cuts.
We propose two techniques,
the use of a dual solution pool (Section \ref{sub:DSP}) 
and a curated version of this pool (Section \ref{sub:curateddsp})
to reduce time solving the subproblems.
We also propose initialization techniques in Section \ref{sub:init} 
aimed at accelerating convergence of the algorithm.

\subsection{Dual Solution Pool}
\label{sub:DSP}

As discussed in Section \ref{sec:Benders}, 
the extreme points of the dual of the subproblem \eqref{eq:sp_def} are used to generate Benders cuts.
Given our assumption that \( W \text{ and } q \) are fixed in the second-stage problem,
the dual feasible region of the subproblem is the same for every possible scenario, as shown in \eqref{eq:spdual}.
This implies that the dual solutions from previous replications can be used
to generate valid Benders cuts for the current replication
without solving subproblems.

To exploit this observation, 
as we discover dual solutions when solving an SAA replication, we store them
in a dual solution pool (DSP) denoted by \( \pilist \).
Then, when solving a new SAA replication, these dual solutions 
can be used to generate Benders cuts 
if they cut off the current primal solution, thereby potentially avoiding the need to solve the subproblem \eqref{eq:sp_def}.

We first describe how the DSP is used when solving two-stage stochastic LPs. Algorithm \ref{alg:BD} is modified by running Algorithm \ref{alg:dsp} before starting the loop of solving subproblems (line 7). 
After obtaining a main problem solution, \( x^t \),
we check the DSP, \( \pilist \), for each scenario $\kink$
to find if it contains any dual solutions that define a Benders cut violated by $x^t$. 
Specifically,
we solve the following problem for each scenario $k \in [K]$ (line 3 of Algorithm \ref{alg:dsp}):
\begin{equation}
  \qapp{x^t, \pilist} \coloneqq \max_{{\pi \in \pilist}} \pi^\top(h(\xi_k) - T(\xi_k) x^t).
  \label{eq:value_approx}
\end{equation}
That is, for each $k \in [K]$, \(   \qapp{x^t, \pilist} \) is defined as the objective value of subproblem \eqref{eq:spvertex} evaluated at \( x^t \),
  with the feasible region replaced by \( \pilist \).
  If problem \eqref{eq:value_approx} identifies a violated cut, i.e., \( \qapp{x^{t}, \pilist } > \theta^t_k \)
for some scenario $\kink$,
we add the identified cut for each such scenario 
and proceed with solving the updated main problem (lines 4-5 in Algorithm \ref{alg:BD}).
If no violated cut is found from the DSP for any scenario $k \in [K]$, then Algorithm \ref{alg:dsp} proceeds with solving the scenario subproblems (line 7). 
If any of these subproblems yields a violated Benders cut, the dual solution that defines the cut is not in the DSP (otherwise we would have found the violated cut when running Algorithm \ref{alg:dsp}).
Thus, every dual solution that defines a violated cut is saved, and at the end of the replication we add these the DSP for use in the solution of the following SAA replications.

We propose to use the DSP in two ways when solving two-stage stochastic IPs. First, as discussed in Section \ref{sub:Stochastic IPs}, the LP relaxation of the stochastic IP is solved  to obtain a set of initial cuts to include in the main problem before starting the branch-and-cut process. The DSP can be used exactly as described in the last paragraph for solving two-stage stochastic LPs to accelerate this process. 
Second, we can apply Algorithm \ref{alg:dsp} when an integer feasible solution $\hat{x} \in X$ is found in Algorithm \ref{alg:bnc} before solving subproblems (lines 9-15 of Algorithm \ref{alg:bnc}).
If a cut is found from the DSP for any scenario $k \in [K]$ it is added to the main problem and the LP relaxation at the node is re-solved (line 5). If no cuts are found in the DSP, then the subproblems are solved as usual.

\begin{algorithm} 
  \caption{Using DSP to look for a violated Benders cut.}
\begin{algorithmic}[1]
\State \textbf{Input:} Current main problem first-stage solution: \(x^t\), DSP: \(\pilist\)
\State $\textrm{cutAdded} \gets \textrm{False}$ 
\ForAll{$\kink$}
\State Evaluate $\qapp{x^t, \pilist}$ using \eqref{eq:value_approx} and let \( \overline{\pi}_k \) be a dual solution achieving the max
    \If{$\qapp{x^t, \pilist} > \theta^t_k$} 
    \State Add violated cut: \( \theta_k \ge {\overline{\pi}_k}^\top(h(\xi_k) - T(\xi_k) x^t) \)
        \State $\textrm{cutAdded} \gets \textrm{True}$
    \EndIf
\EndFor
\end{algorithmic}
\label{alg:dsp}
\end{algorithm}

We next turn to a theoretical investigation of the size of the DSP that is required for it to be expected to successfully find violated Benders cuts. 
Let $\prooffeasibleregion \subset \R^d$ be a polytope with vertex set $\vertex$ and define $z^*: \R^d \to \R$ as:
\begin{equation}
\label{eq:randprob}
    z^*(c) = \max_{\pi \in \prooffeasibleregion} c^\top \pi.
\end{equation}
In addition, define $\pi^*:\R^d \rightarrow \vertex$, where $\pi^*(c)$ is an optimal solution of \eqref{eq:randprob}
for each $c \in \R^d$ (ties can be broken arbitrarily).
Let $A_v, v \in \vertex$ be a partition of
$\mathbb{R}^d$ such that $A_v = \{ c \in \mathbb{R}^d : \pi^*(c) = v \}$. 
Assume we are solving a sequence of $N$ problems of the form \eqref{eq:randprob} with $c=\mathbf{c}^i$ for 
$i=1,\ldots, N$ where each $\mathbf{c}^i$ is a random vector drawn independently from the same distribution. 
Given the set of optimal solutions obtained from these problems (think of it as the DSP), we then obtain a new random
coefficient $\mathbf{\hat{c}}$ from the same distribution and wish to understand whether any solution in this set is optimal for this new coefficient.
The following lemma relates this probability to the geometry of the set $\prooffeasibleregion$. 

\begin{lemma}
\label{lem:bound_primal_exact}
 Let $\bc^1, \dots, \bc^N$ and $\hat{\bc}$ be independent random vectors identically distributed according to a
 distribution $\mathcal{D}$.
For each vertex $v \in \vertex$, let $q_v = \pr_\mathcal{D}(\bc \in A_v)$.
Define `$\mathrm{Failure}$' as the event that none of the solutions $\pi^*(\bc^i)$ for $i=1,\ldots,N$ is optimal for
$z^*(c)$.
Then,
\[
\pr( \mathrm{Failure})  = \sum_{v \in \vertex} q_v (1 - q_v)^N.
\]
\end{lemma}

\begin{proof}
We analyze the probability by conditioning on the realization of the target vector $\hat{\bc}$. Suppose
$\pi^*(\hat{\bc}) = v$, and hence $\hat{\bc} \in A_v$. The sample set fails to recover $v$ if and only if no sample $\bc^i$ falls into $A_v$.
For a single sample $\bc^i$, the probability of missing region $A_v$ is $1 - q_v$. Since the $N$ samples are independent, the probability that all of them miss $A_v$ is:
\[
\pr(\text{Failure} \mid \hat{\bc} \in A_v) = (1 - q_v)^N.
\]
To find the unconditional failure probability, we integrate over the possible
realizations of $\hat{\bc}$. Since the regions $A_v$ partition the space, we apply the law of total probability summing
over all $v \in \vertex$:
\begin{align*}
\pr(\text{Failure}) &= \sum_{v \in \vertex} \pr(\text{Failure} \mid \hat{\bc} \in A_v) \pr(\hat{\bc} \in A_v) \\
&= \sum_{v \in \vertex} (1 - q_v)^N \cdot q_v. \qedhere
\end{align*}
\end{proof}
If
all vertices are equally likely to be optimal ($q_v = 1/|\vertex|$), then the failure probability in Lemma
\ref{lem:bound_primal_exact} simplifies to 
\[ \sum_{v \in \vertex} \frac{1}{|\vertex|} \Bigl(1 - \frac{1}{|\vertex|}\Bigr)^N = \Bigl(1 - \frac{1}{|\vertex|}\Bigr)^N \leq
e^{-N/|\vertex|}.  \]
Thus, to assure a failure probability less than $\rho$ would require $N \geq |\vertex | \ln (1 / \rho)$. This
illustrates dependence on the geometry of $\prooffeasibleregion$ in terms of its number of vertices $|\vertex|$.
For a standard simplex in $\mathbb{R}^d$, this is a modest bound as $|\vertex|=d$. However, in general, $|\vertex|$ may be exponential in
$d$. On the other hand, the following corollary demonstrates that if a subset of vertices has high probability of being
optimal, then a sample size on the order of the size of that subset is sufficient to achieve low failure probability. 

\begin{corollary}
Let $\delta > 0$ and assume $S \subseteq \vertex$  satisfies
$ \sum_{v \in S} q_v \geq 1- \delta.$ 
Then, 
\[ \pr(\mathrm{Failure}) \leq \frac{|S|}{e N} + \delta. \]
\end{corollary}
\begin{proof}
Applying Lemma \ref{lem:bound_primal_exact} and the assumption yield
\begin{align}
\pr( \text{Failure}) &= \sum_{v \in \vertex} q_v (1 - q_v)^N \nonumber \\
&=\sum_{v \in S} q_v (1 - q_v)^N  + \sum_{v \in \vertex \setminus S}  q_v (1 - q_v)^N \nonumber \\
&\leq  \sum_{v \in S} q_v (1 - q_v)^N + \sum_{v \in \vertex \setminus S} q_v 
\leq  \sum_{v \in S} q_v (1 - q_v)^N + \delta \label{eq:im1} . 
\end{align}
Next observe that for $q \in [0, 1]$, $q(1-q)^N \leq qe^{-Nq}$ and the maximum of $f(q) = qe^{-Nq}$ over $q \geq 0$ is
obtained at $q=1/N$. Indeed, $f'(q) = e^{-Nq}(1-Nq)$ and hence the unique stationary point occurs at $q=1/N$. Since
$f''(q)=-Ne^{-Nq}+(1-Nq)(-Ne^{-Nq})=(-2N+N^2q)e^{-Nq}$, $f''(1/N) < 0$ and hence this is the maximum. Thus, $q(1-q)^N
\leq (1/N)e^{-1} = 1/(eN)$. Substituting into \eqref{eq:im1} yields
\[\pr( \text{Failure})  \leq \sum_{v \in S} \frac{1}{eN} + \delta = \frac{|S|}{eN} + \delta.  \qedhere  \]
\end{proof}

To relate this analysis to our use of DSP for searching for a violated Benders cut, recall that to search for a violated cut from the DSP, we solve the problem \eqref{eq:value_approx} in the hopes of finding a violated cut, rather than solving the true dual subproblem \eqref{eq:spdual}.
We thus consider $\pilist$ as the set of the solutions, $\{\pi^*(\bc^1), \pi^*(\bc^2), \dots, \pi^*(\bc^N)\}$, 
and the cost vector as $\hat{c} = h(\xi_k) - T(\xi_k) x^t$. If the solution $x^t$ has been used in a previous replication to generate Benders cuts, then $\hat{c}$ has the same distribution as the distribution of the cost vectors in the subproblems \eqref{eq:spdual} in that previous replication, and hence  $\pilist$ will contain dual optimal solutions from $K$ samples from the same distribution as $\hat{c}$.
Then Lemma \ref{lem:bound_primal_exact} provides a bound on the probability that the
best solution 
in the DSP (i.e., the dual solution that provides the largest right-hand side value of a Benders cut) 
is optimal (i.e., the maximum possible Benders cut violation). 
This implies that if a violated Benders cut exists, one would likely be found in the DSP. This analysis is not directly applicable if the primal solution $x^t$ was not seen in a previous replication. However, the DSP contains dual solutions derived from many other primal solutions, and hence may still be useful for generating Benders cuts, which we verify empirically in Section \ref{sec:Computational study}.

\subsection{Curated DSP}
\label{sub:curateddsp}
Our preliminary experiments indicated that the number of distinct dual extreme points discovered and stored in the DSP
tends to grow rapidly as we solve more SAA replications, even though the underlying dual polytope remains fixed.
This growth makes it increasingly time-consuming to search the DSP for violated cuts.
To address this, we propose to use a curated DSP in which we restrict the set of stored dual solutions to a more manageable size. This restricted set of stored dual solutions is denoted by $\picure$. We use the curated DSP exactly as the DSP is used as described in Section \ref{sub:DSP}. The only difference is that when searching for cuts based on past dual solutions, we search the set $\picure$ rather than the full DSP ($\pilist$).

Our approach for creating a curated DSP is detailed in Algorithm \ref{alg:curated_dsp}.
After each SAA replication,
the full DSP $\pilist$ is partitioned into three sets:
the permanent set ($\piperm$), the trial set ($\pitrial$), and the remaining solutions which we refer to as ``the bench''.
The permanent set includes dual solutions that have 
generated violated cuts in multiple past replications.
Once a solution is added to the permanent set,
it remains there for all subsequent replications.
The trial set consists of newly generated dual solutions discovered 
during the previous replication.
These newly discovered solutions are included in the curated DSP in the following replication.
If a solution in  the trial set successfully identifies a violated cut,
it is added to the permanent set in the next replication.
The curated DSP, $\picure$,
is the union of the permanent and trial sets. 
After finishing solving a replication, 
all dual solutions used to define violated cuts in the replication 
are added to a set called $\piused$.
After each replication, dual solutions are re-evaluated:
those that were already in the DSP and
defined a violated cut (i.e., they are in the set $\piused$) are added to the permanent set $\piperm$.
Any dual solution that was newly generated during the current replication, and was not already part of the full DSP, is added to the trial set $\pitrial$.
This allows new solutions to be tested and potentially included 
in future replications if they prove effective.
For the next replication,
the curated DSP is updated by combining the updated permanent and trial sets.

\begin{algorithm}
\caption{Curated dual solution pool.}
\begin{algorithmic}[1]
\State Initialize \(\pilist\) with dual solutions collected in the first replication
\State Initialize \(\piperm \gets \emptyset\), \(\pitrial \gets \emptyset\) 
\ForAll{replications \(r = 2, \dots, M\)} 
    \State \(\picure \gets \piperm \cup \pitrial\) 
    \State Solve replication using Benders decomposition
    \State Add dual solutions used to generate cuts to \(\piused\) 
    
    \For{\(\pi \in \piused\)} \Comment{Re-evaluation of duals}
  \If{\(\pi \in \pilist\)} 
            \State \(\piperm \gets \piperm \cup \{\pi\}\) 
        \ElsIf{\(\pi \notin \pilist\)}
            \State \(\pitrial \gets \pitrial \cup \{\pi\}\) \Comment{\(\pi\) is newly discovered}
            \State \(\pilist \gets \pilist \cup \{\pi\}\) \Comment{Update full DSP}
        \EndIf
    \EndFor
    
    \State Re-evaluate and update \(\piperm\) and \(\pitrial\) for the next replication
\EndFor
\end{algorithmic}
\label{alg:curated_dsp}
\end{algorithm}

\subsection{Initialization Techniques}
\label{sub:init}

The Benders decomposition method can be significantly accelerated if the main problem is initialized with a good set of Benders cuts.
For linear programs, adding the ``right'' initial cuts can theoretically lead to 
convergence in just one iteration.
 Initializing the main problem with Benders cuts has also been observed to be important to solve stochastic IPs\cite{saharidis2011initialization,fortz2009improved}.
Thus, we investigate techniques for determining a set of Benders cuts to add to the initial Benders main problem. We focus our discussion on initialization techniques for two-stage stochastic IPs, and then discuss adaptations of these ideas that we propose for two-stage stochastic LPs.

In our approach, we aim to leverage dual solutions 
collected from prior replications ($\pilist$) to generate initialization cuts 
for the current replication.
Using all the collected dual solutions to generate a cut for every scenario
would provide an initialization of the algorithm that provides the best possible bound given those dual solutions.
However, this would also lead to a large number of initial cuts in the main problem,
slowing down the solution of the main problem LP relaxations throughout the algorithm.
Thus,  we explore techniques 
for selecting, for each $k \in [K]$, a subset $ \piselk \subseteq \pilist$ of dual solutions from the DSP from which to add Benders cuts.
To guide this selection, we use the first-stage solutions encountered
during Benders decomposition in previous replications.
We denote the set of all feasible solutions found during the solution process in past replications as $\xlist$ (e.g., this includes integer first-stage solutions the solver finds via its internal heuristics and integer solutions discovered at nodes in the branch-and-bound search),
and denote the set of optimal solutions of previous replications as $\xlistopt$. We propose two methods for choosing the initial cuts to add: static initialization and adaptive initialization, described in the following two subsections.

 
\subsubsection{Static Initialization}
\label{ssub:Static Initialization}
The main idea behind this approach that a previous optimal solution has a high likelihood of 
being near-optimal for this replication \cite{kim2015guide}.
Therefore, we want to initialize the algorithm with cuts 
that maximize the objective value of the subproblem dual when evaluated on solutions in our set of previous optimal solutions \( \xlistopt \).
For
each previous optimal solution, \( \overline{x}  \in \xlistopt \),
and for each scenario $k \in [K]$, we identify a dual solution 
from the DSP that maximize the subproblem objective value, i.e., a dual solution that achieves the maximum in \eqref{eq:value_approx}.
If there is a tie between multiple dual solutions,
we randomly select one of them.
We could add the Benders cut corresponding to this dual solution for each $\overline{x} \in \xlistopt$ and each $k \in [K]$.
However, in our experiments,
we found that doing this for all previous optimal solutions yielded many cuts in the initialization that were not useful.
Thus, we only add cuts for the first two optimal solutions 
in $\xlistopt$, leading to at most two cuts per scenario.

\subsubsection{Adaptive Initialization} 
\label{ssub:Adaptive Initialization}

In this initialization technique,
 we use both the set of optimal solutions 
of previous replications, $\xlistopt$, and 
the full set of previously found feasible solutions,
$\xlist$, to identify initial Benders cuts to include in the main model. This technique proceeds in two phases. In phase one, 
we find the solution \( \bestx \) 
which has the 
lowest objective value among the solutions in $\xlistopt$ for the current SAA replication. 
In the second phase, we identify a set of cuts which ensure that 
the objective values of all other solutions in \( \xlist \)
are suboptimal compared to \( \bestx \) when evaluated on the selected set of cuts.
These cuts provide a strong initialization and
ensure that these solutions are not encountered 
later in Benders decomposition,
as, by design, their objective values in the model will be worse 
than the objective value of the initial solution we provide to the model.
The hope is that the set of solutions, $\xlist$,
serves as a representative approximation for the entire feasible 
region $X$, and constructing the cuts 
this way will lead to faster convergence.
In the limit,
if we had all the feasible primal solutions in $\xlist$, 
then this initialization would find the optimal solution. 


\paragraph{Phase One.}%
Let \( \xvalue(x) = c^\top x + \sum_{k \in [K]} p_k Q(x,\xi_k)\) denote the objective value of a first-stage solution \( x \in X \). 
In this phase, our goal is to find the solution 
$\bestx \in \xlistopt$ with the lowest true objective value,
i.e., it satisfies 
\[ \xvalue(\bestx)\le \xvalue(x) \quad \forall x \in \xlist. \]
While this could be accomplished by directly evaluating \( \xvalue(x) \) for all \( x \in \xlist \) this would be computationally expensive as it requires 
solving all scenario subproblems for each solution. 
Therefore, to find \( \bestx \),
we use an approximation of the true objective value, 
\[ \vapp{x, \pilist} = c^\top x + \sum_{k \in [K]} p_k \qapp{x, \pilist}, \]
where, for each $k \in [K]$, $\qapp{x, \pilist}$ is the approximate value of the subproblem 
for scenario $k$ as defined in \eqref{eq:value_approx}.
$\vapp{x, \pilist}$ represents the approximate objective value 
of \( x \),
with the subproblems being evaluated on the DSP (\( \pilist \))
instead of the full feasible region (\( \dfs \)).
This approximation always underestimates the true objective value,
i.e., \( \vapp{x, \pilist} \le \xvalue(x)  \) as 
$\qapp{x, \pilist} \le Q(x,\xi_k)$ for all $\kink$
because \( \pilist \subseteq \dfs \).


Algorithm \ref{alg:phase_one} outlines the process of finding \( \bestx \).
We start by calculating the approximate objective \( \vapp{x, \pilist} \) 
for each solution in $\xlistopt$ 
and then arrange them in ascending order of this approximate objective value.
We then check whether the solution with the smallest value of \(\vapp{x, \pilist} \),
 say \( \overline{x}  \),
 has the lowest true objective value,
by solving the scenario subproblems \eqref{eq:sp_def} with $x=\overline{x}$ and $\xi = \xi_k$ for each scenario $k \in [K]$ to find the true objective value of this solution. 
In this process, we may generate new dual solutions which 
are added to the DSP.
If any new dual solutions are found, this will 
increases the value of \( \vapp{\overline{x}, \pilist} \)
to \( \xvalue(\overline{x}) \), and
\( \overline{x} \) might no longer be the solution with the lowest value of 
\( \vapp{x, \pilist} \) among $\xlistopt$.
Therefore, we re-evaluate $\vapp{x, \pilist}$ for all solutions 
on the updated $\pilist$ to see if \(\overline{x} \)
remains the best candidate.
If so, the algorithm terminates;
otherwise, we select the new minimizer and repeat the process.
When the algorithm terminates it holds that
\[ z(\bestx) = \vapp{\bestx,\pilist} \leq \vapp{x,\pilist} \leq z(x) \quad \forall x \in \phaseonexlist \]
and hence 
 we have found the solution with the lowest true objective value in the set $\phaseonexlist$.
This algorithm is guaranteed to converge in at most $|\phaseonexlist|$ iterations because in each iteration we expand $\pilist$ such that $z(x) = \vapp{x,\pilist}$ for a new $x \in \phaseonexlist$.

\begin{algorithm} 
  \caption{Adaptive initialization - phase one.}
\begin{algorithmic}[1]
\State Evaluate \( \vapp{x, \pilist} \) for all \( x \in \phaseonexlist \)
\Repeat
    \State \( \bestx \gets \argmin \{ \vapp{x, \pilist} : x \in \phaseonexlist \} \)
    \State Solve subproblem \eqref{eq:sp_def}, with $x=\bestx, \xi =\xi_k$ for each $k \in [K]$
   \State Update \( \pilist \) with newly found dual solutions
\State Re-evaluate \( \vapp{x, \pilist} \) for all \( x \in \phaseonexlist \) 
    \Until{\( \argmin \{ \vapp{x, \pilist}: x \in \phaseonexlist \} = \bestx \)}
\State \Return \( \bestx \)
\end{algorithmic}
\label{alg:phase_one}
\end{algorithm}
\paragraph{Phase two.}%
The goal of phase two is to identify a (hopefully small)  
set \( \piselk \subseteq \pilist \) for each scenario $k \in [K]$ such that, when the objective value
of each solution in \( \xlist \) is evaluated using the Benders cuts defined by these solutions, the evaluation is higher than \( \xvalue(\bestx) \).
As a result, when the Benders algorithm proceeds, none of these solutions will be identified as a candidate solution that might be better than $\bestx$. In this process, since we are considering more solutions ($\xlist$) than we considered in phase one ($\xlistopt$), we may find a solution $x \in \xlist$ that has a better objective value than $\bestx$ identified in phase one, in which case we update $\bestx$.



Given a collection of sets of dual solutions \( \piselk \subseteq \pilist \) for $k \in [K]$, we define the lower bound approximation of the objective value of \( x \), 
$$ \vsel{x} := c^\top x + \sum_{k} p_k \qsel{x},$$
where, for each $k \in [K]$, \( \qsel{x} \) is the objective of the subproblem \( k \) 
evaluated using the set of dual solutions \( \piselk \): 
\[ \qsel{x} := \max_{{\pi \in \piselk}} \pi^\top(h(\xi_k) - T(\xi_k) x). 
\]
Note that for all scenarios, $\kink$, $\piselk \subseteq \pilist \subseteq \dfs$ and so
\begin{align*}
\qsel{x} \leq \qapp{x, \pilist} \leq Q(x,\xi_k)
\end{align*}
for all $x \in X$.
Thus,
\[ \vsel{x} \leq \vapp{x, \pilist} \leq \xvalue(x) \quad \forall x \in X. \]
Using this notation, we restate the primary goal of phase two which is to find sets of dual solutions \( \pisel \) that satisfy
\begin{equation}
  \xvalue(\bestx) \le \vsel{x}  \quad \forall x \in \xlist.
\label{eq:convergence_phase2}
\end{equation}
%
The pseudocode for phase two is presented in Algorithm \ref{alg:phase2_ip}.
To initialize \( \pisel \), we first add the dual solutions obtained 
by solving subproblems for \( \bestx  \).
This ensures that $\vsel{\bestx} = \xvalue(\bestx)$.
For IPs, we first solve the LP relaxation to add initialization cuts and thus we also include the dual solutions which defined active cuts 
at the optimal solution of the LP relaxation in the sets $\pisel$.
Next, we evaluate \( \vsel{x} \) for all \( x \in \xlist \).
If the solution with least value of \( \vsel{x} \), say \( \closex \),
has the same objective value as \( \xvalue(\bestx ) \), then we have converged having achieved our goal
\eqref{eq:convergence_phase2}.

However, if $\vsel{\closex}< \xvalue(\bestx)$ (line 7),
this implies that we need to add more cuts to $\pisel$ to increase $\vsel{\closex}$ above $\xvalue(\bestx)$.
To do so, we first calculate \(\vapp{\closex, \pilist} \).
If \(\vapp{\closex, \pilist} \geq \xvalue(\bestx)\), 
then we know there are cuts in the DSP which can be added to $\pisel$ to achieve the goal of
\begin{equation}
\label{eq:thegoal}
\vsel{\closex}\geq \xvalue(\bestx) .
\end{equation}
Indeed, this would be achieved by adding the dual solution from $\pilist$ that achieves the maximum in \eqref{eq:value_approx} to $\piselk$ for each $\kink$. However, we heuristically try to minimize the number of dual solutions that are added to achieve \eqref{eq:thegoal}.
We arrange the scenarios in decreasing order of values of 
\( \qapp{\closex, \pilist} - \qsel{\closex}\).
This quantity tells us how much \( \qapp{\closex, \piselk} \) 
will increase if we add the dual solution which achieves 
\( \qapp{\closex, \pilist} \) to \( \piselk \).
For each scenario $k$ in this order, we add a dual solution from $\pilist$ that achieves the maximum in \eqref{eq:value_approx} to $\piselk$, and stop as soon as we achieve \eqref{eq:thegoal}.

If \( \vapp{\overline{x}, \pilist} < \xvalue(\bestx)\) 
(line \ref{alg:else}),
this implies the dual solutions in $\pilist$ are not sufficient for adding to $\pisel$ to achieve \eqref{eq:thegoal}. Indeed,
$\overline{x}$ may even have a lower true objective than the current $\bestx$.
In this case, we solve the subproblems \eqref{eq:sp_def} for $\overline{x}$ and each scenario $k \in [K]$ to calculate
$\xvalue(\closex)$. 
For each $k \in [K]$, we add to $\piselk$ and $\pilist$, the optimal dual solution from the subproblem.
If $\xvalue(\closex) < \xvalue(\bestx)$,
we update the current best solution to be \( \closex \).
Once the algorithm converges,
the final \( \bestx \) is provided as an initial feasible solution when solving the SAA replication.

%

\begin{algorithm} 
  \caption{Adaptive initialization - phase two.}
\begin{algorithmic}[1]
\State \textbf{Input:} $\bestx$ from phase one
\State Initialize \( \piselk \gets \emptyset \) for all \( k \in [K] \)
\State Add $\pi$ which achieves the maximum in \eqref{eq:value_approx}  at \( \bestx \) to \( \piselk \) for each $k \in [K]$
\State Add active LP cuts to \( \pisel \) \label{alg:lpcuts}
\State $z^{WS} \gets  \vapp{\bestx, \pilist}$ 
\While{True}
\State Evaluate \( \vsel{x} \) for all \( x \in \xlist \)
\State \( \overline{x} \gets \argmin \{ \vsel{x} : x \in \xlist \} \)
    \If{\( \vsel{\overline{x}} < z^{WS} \)}
        \State Compute \( \vapp{\closex, \pilist} \)
        \If{\( \vapp{\closex, \pilist} \geq z^{WS} \)}
    \State Add enough duals from DSP to \( \pisel \) so that \( \vsel{\closex} \ge z^{WS} \)
        \Else \label{alg:else}
            \State Solve \eqref{eq:sp_def} for \( x=\overline{x} \) and $\xi = \xi_k$ for each $k \in [K]$ to compute  $\xvalue(\overline{x})$
            \State Add the optimal dual solution from \eqref{eq:sp_def}  to \( \piselk \) and $\pilist$ for each $k \in [K]$
            \If{\( \xvalue(\overline{x}) < \xvalue(\bestx) \)}
                \State Update \( \bestx \gets \overline{x}, z^{WS} \gets \xvalue(\overline{x}) \)
            \EndIf
        \EndIf
        \Else
        \State Break
    \EndIf
\EndWhile
\State \Return \(\bestx, \pisel \)
\end{algorithmic}
\label{alg:phase2_ip}
\end{algorithm}

Adaptive initialization as described is designed primarily for initializing Benders decomposition when solving stochastic IPs.
For stochastic LPs, in preliminary experiments we found
that too much time is spent doing this initialization process 
relative to the savings it yields in the eventual algorithm. 
Thus,  
for LPs, we make some changes to the adaptive initialization.
In phase one, we evaluate the previously collected optimal solutions (\( \xlistopt \)) on the DSP.
Let $\closex$ denote the solution with the lowest approximate 
objective value \( \vapp{x, \pilist} \). 
Rather than solving scenario subproblems to
verify whether $\closex$ has the lowest 
true objective value (line 4 of Algorithm \ref{alg:phase_one}),
we directly declare $\closex$ to be $\bestx$. Phase two begins with this solution and proceeds as in Algorithm \ref{alg:phase2_ip}, with line \ref{alg:lpcuts} skipped as it is not relevant for LPs. The next change is after the else condition on line \ref{alg:else}, which is run when $\vapp{\closex, \pilist} < z^{WS}$. In the LP case, we do not solve subproblems at this point. Instead, we update $\bestx$ to $\closex$, add the dual solution that achieves the maximum in \eqref{eq:value_approx}  at \( \bestx \) to \( \piselk \) for each $k \in [K]$, and then terminate the initialization.

\subsection{Extensions}
\label{sub:extensions} 

In this section, we describe how our ideas for accelerating Benders decomposition using information from previous replications can be adapted for solving problems without relatively complete recourse (Section \ref{ssub:relatively_complete_recourse}) and for the single-cut version of Benders decomposition (Section \ref{ssub:single-cut}).

\subsubsection{Relatively Complete Recourse}
\label{ssub:relatively_complete_recourse} 

In the absence of relatively complete recourse, the reformulation
\eqref{eq:reformulation} that is the basis of Benders decomposition  needs to be augmented with Benders feasibility cuts \cite{Benders1962}:
\( (h(\xi_k) - T(\xi_k) x)^\top r \le 0 \quad \forall r \in \mathcal{R}, k \in [K], \)
where $\mathcal{R}$ is the set of extreme rays of the dual feasible region $\dfs$.

The Benders decomposition algorithm is modified such that if the subproblem \eqref{eq:sp_def} is infeasible, then an extreme ray of the dual feasible region is identified and used to add a Benders feasibility cut to the main problem. 

Since the dual feasible region remains fixed according to our assumptions that $W$ and $c$ are fixed,
the set of rays also remains constant across replications.
Thus, as we seek to solve a sequence of SAA replications,
we can store the dual extreme rays that are identified in a DSP just as we do for Benders optimality cuts. 
When we obtain a main problem solution, we search the DSP for violated cuts (both optimality and feasibility cuts).
If no violated cut is found for any scenario,
we proceed by solving the subproblems to generate a feasibility or optimality cut.
Curating the DSP follows the same principles as we have discussed in 
Section \ref{sub:curateddsp}.

For the static and adaptive initialization method, we can initialize the Benders optimality cuts exactly as described previously. In adaptive initialization, if a scenario subproblem \eqref{eq:sp_def} is infeasible, then we identify an extreme ray of the dual feasible region and add the associated Benders feasibility cut as an initial cut. For adding Benders feasibility cuts, we would check the DSP to determine if any primal solutions from previous replications $\xlist$ are violated by any of the associated feasibility cuts, and add at least one such cut for each solution in $\xlist$ that violates one of these cuts.



\subsubsection{Single-Cut}
\label{ssub:single-cut} 
In single-cut implementation of Benders decomposition,
we introduce a variable $\Theta$
which represents the expected value of 
the subproblem objective.
The single-cut Benders decomposition algorithm is based on a reformulation that includes Benders cuts of the form
\begin{equation}
  \Theta \ge \sum_{k=1}^K p_k (h(\xi_k) - T(\xi_k) x)^\top \pi_k, 
  \label{eq:singlecut}
\end{equation}
where $ \pi_k \in \vertex$ for each $\kink$.

In standard Benders decomposition, given a main problem solution $\hat{x}$, a Benders cut of the form \eqref{eq:singlecut} is found by solving the scenario subproblem \eqref{eq:sp_def} for each scenario $\kink$, 
and then using the dual solution $\pi_k$ from subproblem $k$ for each $\kink$ to define the cut \eqref{eq:singlecut}.
The single-cut version uses fewer variables in the main problem and only adds one cut per iteration. Hence, the main problem typically is more compact and hence solves faster than in the multi-cut approach, but
it often requires more iterations to reach optimality.

Our proposal for using the DSP and the curated DSP directly adapts to the single-cut version. 
For every scenario $\kink$, at a solution $x^t$,
we find the dual solution with maximum $\qapp{x^t, \pilist}$ by solving \eqref{eq:value_approx}
and then aggregate them to generate a cut \eqref{eq:singlecut} using these dual solutions.
If it is violated by $x^t$, we add it to the main problem and continue with the algorithm. 
Otherwise, we solve all scenario subproblems and generate a cut using those dual solutions (and update the DSP with the newly identified dual solutions). 
We expect that this would speed up the Benders iterations because the full set of scenario subproblems do not need to be solved at every iteration. 

To adapt static initialization in this context,
for each previously collected optimal solution, 
we identify dual solutions from the DSP that 
maximize the subproblem objective value by solving \eqref{eq:value_approx} for each $k \in [K]$.
Aggregating these yields a cut for that primal solution, and  
repeating this procedure for all previously collected optimal solutions 
leads to a set of initial cuts that can be added 
to the main problem.

For adaptive initialization, phase one remains unchanged,
where the primary goal is to identify 
the best solution ($\bestx$) to provide as an initial feasible solution to the branch-and-cut algorithm. 
We follow the steps in Algorithm \ref{alg:phase_one} to do this. 
To adapt this method for single-cut, we also maintain $\pisel$ in phase one.
Now, whenever we evaluate any solution, say $\closex$,
on the DSP to estimate its value,
the dual solutions for each scenario $k \in [K]$ that correspond to $\qapp{\closex,
\pilist}$ are stored in $\piselk$.
Furthermore, any dual solutions that are found when solving 
the scenario subproblems are also incorporated into $\pisel$.
In phase two, we follow the same steps as outlined
in Algorithm \ref{alg:phase2_ip}.
This algorithm outputs both an updated $\pisel$ 
and the initial candidate solution, $\bestx$. 
In the multi-cut version, we would add a Benders cut to the main problem for every dual solution in $\piselk$ for each scenario $\kink$.
For single-cut, the main problem is first initialized with a cut of the form \eqref{eq:singlecut}, with the $\pi_k$ dual solutions defined according to the optimal dual solution of subproblem \eqref{eq:sp_def} with $x=\bestx$ and $\xi=\xi_k$.
To determine which additional cuts to initialize the main problem with,
we iterate through each primal solution  $\closex \in \xlist$. For each such $\closex$, we find the maximum value of the right-hand side of the currently added cuts on $\closex$. 
Let's call this value $\hat{\Theta}(\closex)$. 
If $c^\top \closex + \hat{\Theta}(\closex) $ exceeds $z(\bestx)$, we do nothing as the current cuts are sufficient to ensure that the objective value of $\closex$ in the model is higher than the objective value of $\bestx$ in the model.
Otherwise,  we find the dual solution which achieves $\qapp{\closex, \pisel}$
for each scenario $k \in [K]$ and use these to define a cut. 
This cut is added to the main problem, ensuring that, as a result of phase two, the updated $c^\top \closex + \hat{\Theta}(\closex) $ will be at least $z(\bestx)$.

\section{Computational Study}
\label{sec:Computational study}

This section presents a comprehensive computational study to evaluate the 
effectiveness of our proposed information reuse strategies. 

\subsection{Experimental Setup and Implementation Details}

We compare the following approaches for reusing information from previous solves of a replication:
\begin{itemize}
  \item \textbf{Baseline:} This approach represents the standard Benders decomposition algorithm. The only information reused from previous replications is that for IPs an initial feasible solution is provided based on the optimal solution of the most recent replication, as described at the beginning of Section \ref{sec:techniques}. No information is reused for LPs. 
  \item \textbf{DSP:} This approach stores the dual solutions collected in previous SAA replications and uses them to generate cuts as described in Section \ref{sub:DSP}. 
  \item \textbf{Curated DSP:} This approach refines the DSP by maintaining a smaller pool of dual solutions as described in Section \ref{sub:curateddsp}. 
  \item \textbf{Static init:} This approach extends the curated DSP approach by initializing the algorithm with cuts generated through static initialization as described in Section \ref{ssub:Static Initialization}. 
  \item \textbf{Adaptive init:} This approach extends the curated DSP approach by initializing the algorithm with cuts generated through adaptive initialization as  described in Section \ref{ssub:Adaptive Initialization}.
\end{itemize}

For the computational experiments,
we solve 26 replications of problem \eqref{eq:saa}, each with an independently drawn set of scenarios. 
The first replication is used for data collection and is identical for all compared methods. Thus, to compare the impact of different strategies for reusing information, 
all results presented in the following sections 
are based on the  25 SAA replications excluding the first one.
Every replication is given a time limit of one hour for each method.
To reduce the time required to run the experiments, we run less than the full 25 replications for the baseline method because it is significantly slower than the other methods. This method is only tested on the 2nd, 14th and 26th replications, and results reported are averaged over these three runs instead of the full 25 as in the other methods. Since all replications are independent and the only information used from previous replications is an initial feasible solution for IPs, and nothing for LPs, an average of a metric taken over this subset is expected to be a close approximation of the average over the full set of replications.

We implemented Benders decomposition in Python using
Gurobi 10.0.1 as the optimization solver for both LPs and IPs.
We build a main problem model and
a single subproblem model which is updated
 with the current primal solution ($\hat{x}$)
and scenario data ($\xi$) whenever
we need to look for a cut,
saving model building time and enabling warm-starting of the subproblems.

The DSP is implemented as follows: each dual solution $\pi$ is stored in an indexed array and assigned a unique integer index for efficient retrieval.
We use Python's hash function to determine if a dual solution obtained after solving a subproblem is already present in the DSP or not.
When searching the DSP for dual solutions that potentially generate a violated cut, we must efficiently evaluate problem \eqref{eq:value_approx} for each scenario $k \in [K]$: determine which dual solution from $\pilist$ provides the tightest Benders cut for a given first-stage solution $x^t$.
To avoid repeatedly computing $\pi^\top h(\xi_k)$ for all dual solutions $\pi \in \pilist$ and scenarios $k \in [K]$ for every primal solution $x^t$, we precompute and store these terms in a matrix of dimension $|\pilist| \times K$ at the beginning of each SAA replication using a single matrix multiplication operation.
When evaluating a candidate first-stage solution $x^t$, we compute $\pi^\top(-T(\xi_k) x^t)$ for all dual solutions simultaneously using vectorized matrix operations via the \texttt{linalg} library in NumPy, yielding a vector of length $|\pilist|$.
We then solve problem \eqref{eq:value_approx} by combining this vector with the precomputed $\pi^\top h(\xi_k)$ terms and finding the maximum for each scenario using \texttt{numba} to efficiently calculate the $\argmax$.
As discussed in the beginning of Section \ref{sec:techniques}, whenever we solve a new SAA replication (after the first),
we provide the solver with the optimal primal solution of the previous replication as a candidate solution.
To do this, we solve the scenario subproblems given the new scenario data and this candidate solution $\hat{x}$,
and then provide this solution (\( \hat{x}, \{Q(\hat{x},\xi_k)\}_{\kink}  \))
to the solver.
This initialization is done for all methods except for the adaptive initialization, 
in which the method generates its own candidate solution. We do this initialization even for the baseline method in order to better illustrate the impacts of the other techniques we propose for reusing information.

We terminate Benders decomposition when the optimality gap percentage is less than $10^{-4}$:
$(U^t - L^t)/L^t * 100 \%\le 10^{-4}.$
A candidate Benders cut of the form $\theta_k \ge \alpha_k - \beta_k x$
is considered violated by the current main problem solution
$(\hat{x}, \{\hat{\theta}_k\}_{\kink})$
if it satisfies
\[
\hat{\theta}_k - Q(\hat{x},\xi_k) \geq 10^{-5} \left\| \left( 1, \alpha_k, \beta_k \right) \right\|.
\]
This relative cut violation threshold ensures that the violation exceeds a small tolerance scaled by the norm of the cut’s coefficients.

For the branch-and-cut method (Section \ref{sub:Stochastic IPs}),
we implement the algorithm using a callback provided by the solver Gurobi. 
Whenever the algorithm finds an integer feasible solution, the callback is called 
to verify if this solution is feasible to the true problem.
In this callback, if we are using the DSP, we first 
check the DSP for violated cuts. If no violated
cuts are found in the DSP, 
we solve subproblems  to check if any cuts are violated. 
If DSP is not employed, then we directly solve subproblems 
to check for violated cuts.
Due to the presence of callbacks,
we set the \texttt{lazyconstraints} parameter to 1,
and that avoids reductions and transformations which are incompatible 
with lazy constraints.

For stochastic IPs, we first solve the LP relaxation of the problem via Benders decomposition.
Benders cuts that are active after solving the LP relaxation
are retained in the main problem and used as part of the formulation that is given to the solver when it starts the the branch-and-cut algorithm.
The initialization methods add cuts in addition to these cuts.
If we deploy any initialization method for the IP,
then the same method is also used to initialize 
the LP relaxation of the problem.
Initialization methods are always used in conjunction 
with curated DSP to check for violated cuts.
Only duals from curated DSP are considered to generate 
initialization cuts.

The experiments were run on two Intel Core i7 machines: 
an i7-9700 CPU at 3.00GHz and an i7-10700 CPU at 2.90GHz.

\subsection{Test Problems}

The study investigates the performance of all these methods
on three problem classes: stochastic capacitated facility location, stochastic network design, and stochastic unit commitment. We describe these problems at a high level below. Appendix \ref{sec:Appendix} provides the detailed formulation of each problem.
 All test instances have 400 scenarios,
unless mentioned otherwise.

\paragraph{Capacitated Facility Location Problem (CFLP)}

The CFLP has a set of facilities and a set of customers with uncertain demands. 
The objective of the problem is to minimize the total expected cost of building 
and operating facilities while ensuring that customer demand 
is met.
In the first-stage, we decide which facilities to open.
Every facility has a setup cost and also a capacity.
In the second-stage, we decide how to allocate goods 
from open facilities to satisfy customer demand as much as possible.
Unmet demands are  penalized, and thus this model has relatively complete recourse.

Our instances use the data from \cite{cornuejols1991comparison} 
and the extension of these to create stochastic programming instances from 
\cite{dumouchelle2022neur2sp}. 
They create 
a stochastic variant by first generating the first-stage costs and
capacities, followed by generating scenarios 
by sampling $K$ demand vectors
using the distributions defined in
\cite{cornuejols1991comparison}.
In Tables \ref{tab:cflp_ip_instances} and \ref{tab:cflp_lp_instances},
we list the number of facilities and sets of customers we consider 
for the IP and LP instances, respectively. We use larger test instances for the LP instances to provide a more difficult test for that problem class.

\begin{table}[h]
\begin{minipage}[b]{0.45\linewidth}
\centering
\begin{tabular}{cc}
\hline
Facilities & Customers \\
\hline
15 & \( \{105, 125, 215\} \) \\
25 & \( \{95, 105, 185\} \) \\
35 & \( \{105, 185\} \) \\
55 & \( \{125\} \) \\
75 & \( \{105\} \) \\
\hline
\end{tabular}
\caption{CFLP instance data for IPs.}
\label{tab:cflp_ip_instances}
\end{minipage}
\hspace{0.5cm}
\begin{minipage}[b]{0.45\linewidth}
\centering
\begin{tabular}{cc}
\hline
Facilities & Customers \\
\hline
25 & \( \{305, 355, 405, 455, 495\} \) \\
55 & \( \{305, 355, 405, 455, 495\} \) \\
85 & \( \{305\} \) \\
\hline
\end{tabular}
\caption{CFLP instance data for LPs.}
\label{tab:cflp_lp_instances}
\end{minipage}
\end{table}
\paragraph{Multi Commodity Network Design Problem (CMND)}

The CMND problem is defined on a directed network comprising of nodes ($N$), arcs ($A$), and commodities ($\commodities$).
Each commodity must be routed from an origin node to a destination node in the network.
The arcs are characterized by installation costs and capacity. 
The objective is to determine a subset of arcs for installation 
with the goal of minimizing the expected total cost.
In the first-stage,
binary decisions are made for each arc to decide if it will be installed or not.
In the second-stage,
after the demand for each commodity is revealed, 
routing decisions are made for how to route the realized commodity amounts in the network.

We use the test instances in \cite{crainic2011progressive}.
The instances were originally proposed for the deterministic
fixed charge capacitated multi-commodity network design problem \cite{ghamlouche2003cycle}.
To generate stochastic programming instances, we adopt the approach outlined in \cite{jia2021benders}.
They use the techniques described in \cite{song2014chance} to create random samples for the demands of various commodities.
In each scenario, the demand of a commodity 
follows a normal distribution with the mean set to
the demand in the deterministic instances and standard deviation of 0.1 times the mean. To generate instances, we start with base instances given in Tables 
\ref{tab:ip_results_cflp_part2} and \ref{tab:ip_results_cmnd_part2}.
These tables give
the number of nodes, arcs and commodities 
for each base instance.
For each base instance, 
we run experiments on three versions of these instances that differ in the ratio of 
fixed costs to variable costs.
These are given by r02.1, r02.2, and r02.3 in the actual dataset documentation for the II base instance.

\begin{table}[h]
\begin{minipage}[b]{0.5\linewidth}
\centering
\begin{tabular}{cccc}
\hline
Problem Set & \(|N|\) & \(|A|\) & \(|\commodities|\) \\ \hline
II           & 10      & 35      & 25      \\
III          & 10      & 35      & 50      \\
IV           & 10      & 60      & 10      \\ \hline
\end{tabular}
\label{tab:cmnd_ip_instances}
\caption{CMND instance data for IPs.}
\end{minipage}
\hspace{0.5cm} 
\begin{minipage}[b]{0.5\linewidth}
\centering
\begin{tabular}{cccc}
\hline
Problem Set & \(|N|\) & \(|A|\) & \(|\commodities|\) \\ \hline
VI          & 10      & 60      & 50      \\
IX          & 10      & 83      & 50      \\
X           & 20      & 120     & 40      \\ \hline
\end{tabular}
\label{tab:cmnd_lp_instances}
\caption{CMND instance data for LPs.}
\end{minipage}
\end{table}

\paragraph{Stochastic Unit Commitment Problem (UC)}

The UC problem schedules a fleet of thermal generators over a time horizon to meet uncertain electricity demand while minimizing expected total cost.
In the first-stage, binary decisions determine the on/off commitment status for each generator in each time period, subject to minimum up-time and down-time constraints that couple decisions across periods.
In the second-stage, once demand is realized, continuous dispatch variables determine the power output levels for all committed units, subject to generator capacity limits and the system-wide demand balance constraint.
To guarantee relatively complete recourse, unmet demand is permitted at each time period but penalized at a sufficiently high cost.
For a comprehensive survey on mixed-integer programming formulations and solution methods for UC problems, we refer the reader to \cite{knueven2020mixed}.

For computational experiments, we adopt the classical thermal UC instance generation procedure
of Borghetti et al.~\cite{borghetti2003lagrangian}, distributed through Beasley's OR-Library~\cite{beasley1990or}.
It produces single-bus UC instances with a 2-day horizon (48 hourly periods) and fleets whose size we vary between 10, 20, 30, and 40 thermal units.
We use difficulty levels 2 and 3 in the instance generator,
which differ primarily in the minimum up- and down-time requirements.
Unlike CFLP and CMND, we use the same problem sizes for both LP and IP variants of UC.
From preliminary experiments, we noticed that we can solve instances of the same size for both LP and IP, as the LP relaxation is close to optimal for the IP and generates strong initial cuts.

To construct stochastic instances, we generate $K$ demand scenarios by perturbing the base load $\bar B_t$ at each time period $t$ with independent, zero-mean Gaussian noise.

\subsection{Results}
\label{sub:Results}

\begin{table}[ht]
\centering
\begin{tabular}{ll}
\hline
\textbf{Metric} & \textbf{Description} \\
\hline
\someformat{Total T} & Time taken to optimize a SAA replication \\
\someformat{LP T} & Time taken to solve the LPs to optimality \\
\someformat{Init T} & Time taken to initialize the problem. For IPs, this includes selecting \\
&active cuts from the LP and finding other initial cuts \\
\someformat{IP T} & Time taken to solve the IPs to optimality, excluding \someformat{LP T} and \someformat{Init T} \\
\someformat{Iterations} & (LP only) \# of iterations needed to solve the problem \\
\someformat{SP count} & \# of times subproblems are solved to generate cuts \\
\someformat{DSP T} & Time taken to search for a violated cut in the DSP \\
\someformat{SP T} & Time taken to solve subproblems and find a violated cut \\
\someformat{Cut T} & Time taken to find violated cuts to add to the main problem \\
\someformat{Nodes} & (IP only) \# of branch-and-bound nodes \\
\someformat{Root gap} & (IP only) Gap closed at the root node of the IP from initialization cuts \\
\someformat{Callback calls} & (IP only) \# of calls to the callback to check an integer feasible solution \\
\hline
\end{tabular}
\caption{Benders decomposition metrics.}
\label{tab:optimization_metrics}
\end{table}

Table \ref{tab:optimization_metrics}
summarizes the metrics used to evaluate different methods. 
Each metric represents the arithmetic mean of the quantity 
calculated over the 25 SAA replications after the first replication.
For results that are aggregated over multiple instances,
like in Tables \ref{tab:lp_results_cflp} - 
\ref{tab:ip_results_cmnd_part2}, 
the quantity presented is a shifted geometric mean 
over all instances of that test problem, with a shift of 1 being applied.
Geometric mean is chosen as the relevant mean when summarizing results over different base instances
because there might be a lot of variation in the values for 
different instances. Arithmetic mean is used for summarizing results across the 25 replications of an individual instance because we expect the values to remain 
more consistent over the 25 replications.
All the time-related measurements are done in seconds.

For IPs,
note that \someformat{LP T} represents the time spent 
solving the initial LP relaxation before proceeding with the IP.
\someformat{IP T} notes the time taken to solve the IP after initializing, 
and \someformat{Total T} includes the time taken for this entire process. 
For each IP instance, \someformat{Total T} = \someformat{Init T} + \someformat{LP T} + \someformat{IP T}\footnote{This equality holds on a per-instance basis, but does not hold for the summary statistics because we use geometric mean to summarize across instances.}.

\someformat{DSP T} and \someformat{SP T} describe 
how much time is being spent to generate Benders cuts 
via the DSP and by solving subproblems.
\someformat{Cut T} tells us the total time taken to find and add Benders cuts: for each instance,
\someformat{Cut T} = \someformat{DSP T} + \someformat{SP T}.


\subsubsection{LP Results}
\label{sub:LP results}

\begin{table}[h]
\begin{center}
\begin{tabular}{lrrrrrrr}
\hline
Method       & \multicolumn{1}{c}{\someformat{Total T}} & \multicolumn{1}{c}{\someformat{Iterations}} & \multicolumn{1}{c}{\someformat{SP count}} & \multicolumn{1}{c}{\someformat{Cut T}} & \multicolumn{1}{c}{\someformat{DSP T}} & \multicolumn{1}{c}{\someformat{SP T}} & \multicolumn{1}{c}{\someformat{Init T}} \\
\hline
Baseline     & 213.9                       & 81.2                           & 81.2                          & 187.4                        & -                       & 187.4                    & -                       \\
DSP          & 62.1                        & 66.4                           & {8.0}                           & 50.8                         & 36.0                      & {14.3}                     & -                        \\
Curated DSP  & 46.3                        & 58.6                           & 8.4                           & 37.8                         & 22.3                      & 15.2                     & -                        \\
Static init  & 23.3                        & 25.1                           & 8.4                           & 20.7                         & 5.4                       & 15.2                     & {0.1}                        \\
Adaptive init & {22.2}                        & {20.7}                           & 8.5                           & {18.7}                         & {3.2}                       & 15.4                     & 1.3                        \\
\hline
\end{tabular}
\end{center}
\caption{LP results: CFLP.}
\label{tab:lp_results_cflp}
\end{table}

\begin{table}[h]
\begin{center}
\begin{tabular}{lrrrrrrr}
\hline
Method       & \multicolumn{1}{c}{\someformat{Total T}} & \multicolumn{1}{c}{\someformat{Iterations}} & \multicolumn{1}{c}{\someformat{SP count}} & \multicolumn{1}{c}{\someformat{Cut T}} & \multicolumn{1}{c}{\someformat{DSP T}} & \multicolumn{1}{c}{\someformat{SP T}} & \multicolumn{1}{c}{\someformat{Init T}} \\
\hline
Baseline     & 616.9                                   & 131.4                                     & 131.4                         & 531.2                                   & -                                                 & 531.2                                             & -                                                \\
DSP          & 89.7                                    & 75.2                                      & {10.0}                          & 70.2                                    & 26.3                                                & {40.3}                                              & -                                                \\
Curated DSP  & 63.4                                    & 66.2                                      & 10.2                          & 47.3                                    & 4.9                                                 & 41.4                                              & -                                                \\
Static init  & {49.9}                                    & 31.5                                      & 10.4                          & 44.4                                    & 2.0                                                 & 42.2                                              & {0.1}                                                \\
Adaptive init & 51.2                                    & {21.9}                                      & 10.5                          & {44.0}                                    & {1.2}                                                 & 42.7                                              & 3.6                                                \\
\hline
\end{tabular}
\end{center}
\caption{LP results: CMND.}
\label{tab:lp_results_cmnd}
\end{table}

\begin{table}[h]
\begin{center}
\begin{tabular}{lrrrrrrr}
\hline
Method       & \multicolumn{1}{c}{\someformat{Total T}} & \multicolumn{1}{c}{\someformat{Iterations}} & \multicolumn{1}{c}{\someformat{SP count}} & \multicolumn{1}{c}{\someformat{Cut T}} & \multicolumn{1}{c}{\someformat{DSP T}} & \multicolumn{1}{c}{\someformat{SP T}} & \multicolumn{1}{c}{\someformat{Init T}} \\
\hline
Baseline     & 265.0                       & 44.8                           & 44.8                          & 71.7                         & -                       & 71.7                     & -                       \\
DSP          & 159.0                       & 47.5                           & 6.2                           & 36.4                         & 26.5                    & 10.3                      & -                       \\
Curated DSP  & 135.9                        & 46.9                           & 6.9                           & 22.6                         & 11.1                     & 11.4                      & -                       \\
Static init  & 56.7                        & 33.6                           & 7.2                           & 16.4                         & 4.6                     & 12.0                      & 0.8                     \\
Adaptive init & {54.0}                        & {31.5}                           & {7.4}                           & {15.6}                         & {3.6}                     & {12.3}                      & {4.0}                     \\
\hline
\end{tabular}
\end{center}
\caption{LP results: UC.}
\label{tab:lp_results_uc}
\end{table}
\begin{figure}[h]
  \centering
  \begin{subfigure}{0.48\textwidth}
    \centering
    \includegraphics[width=\linewidth]{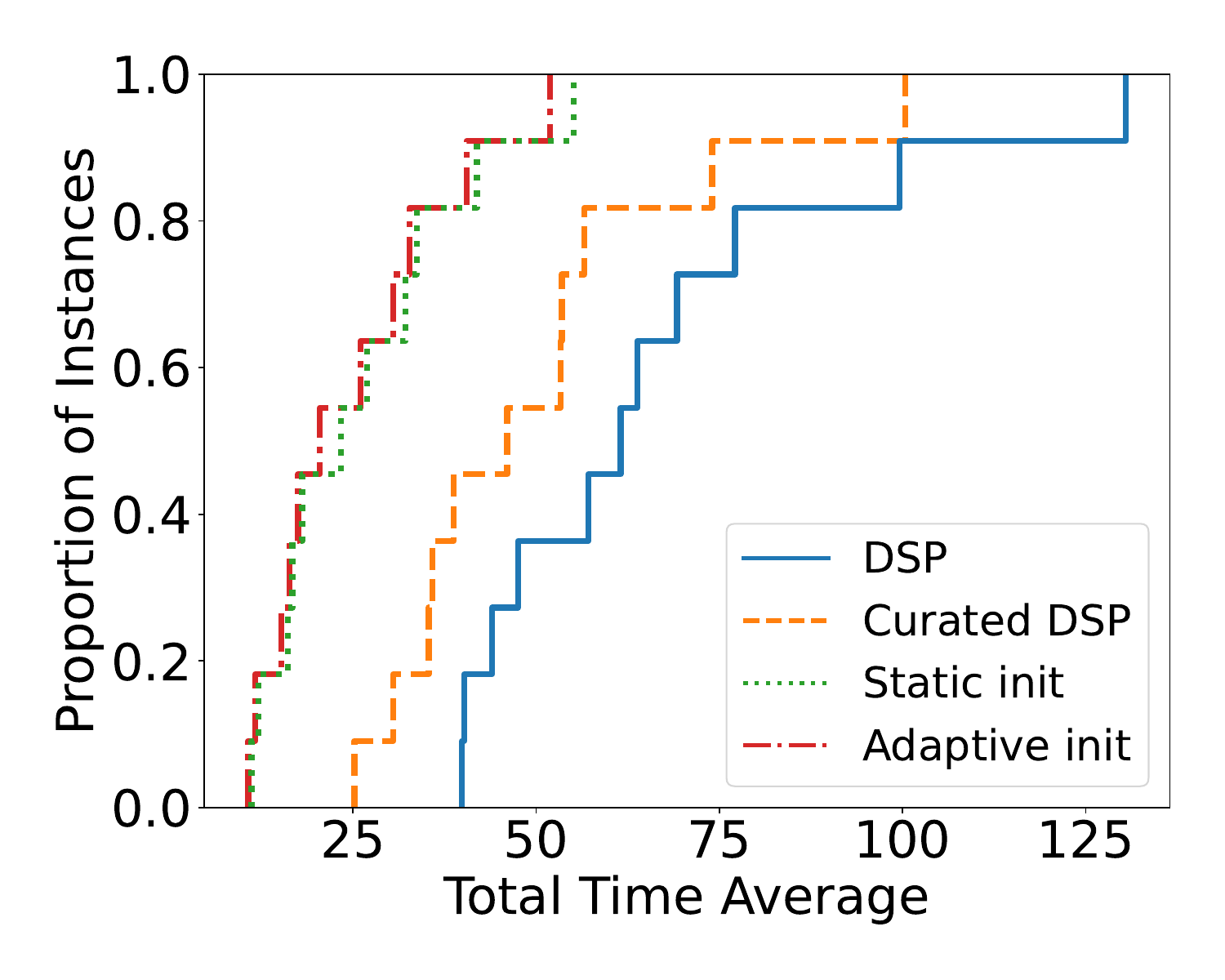}
    \caption{CFLP}
  \end{subfigure}
  \hfill
  \begin{subfigure}{0.48\textwidth}
    \centering
    \includegraphics[width=\linewidth]{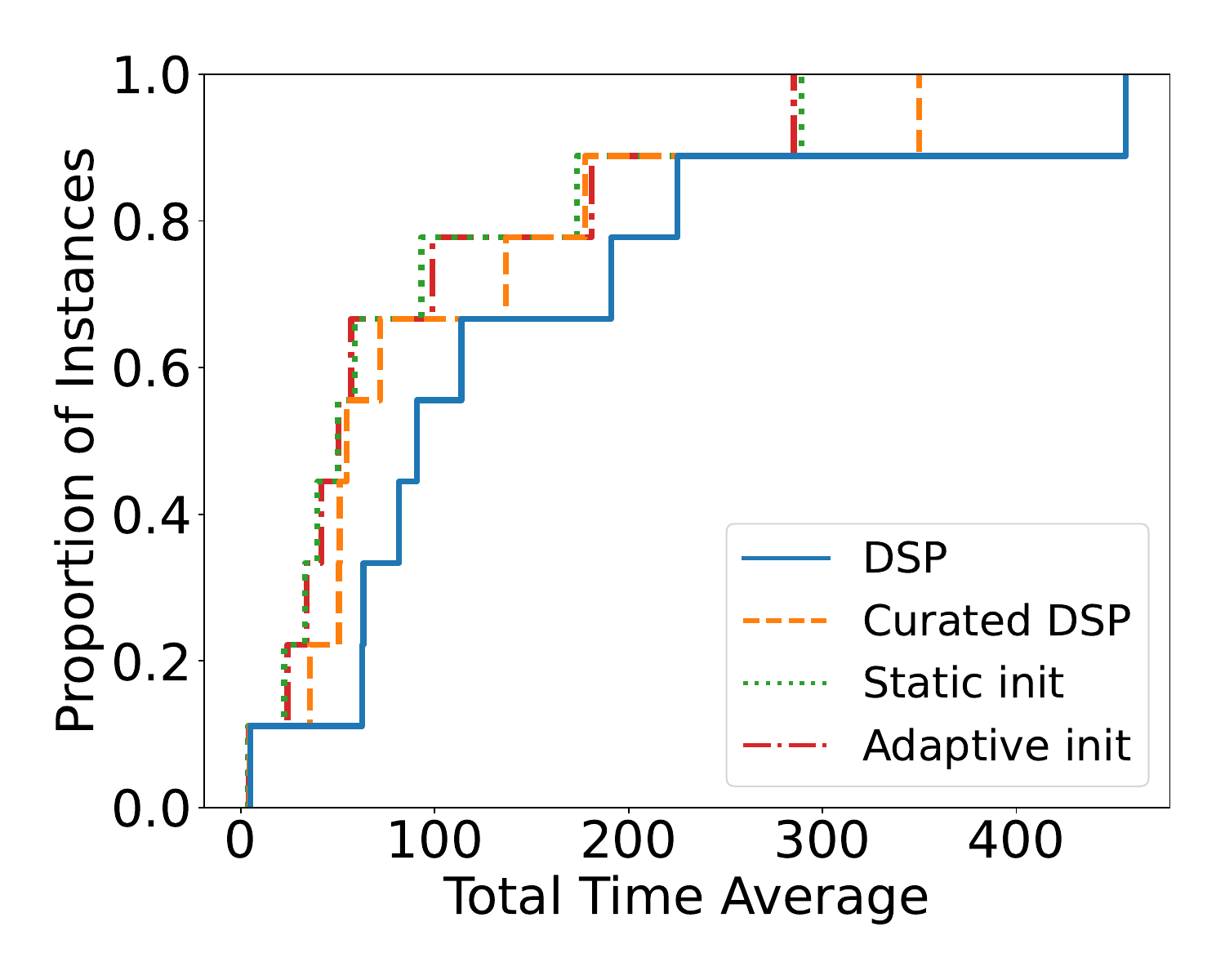}
    \caption{CMND}
  \end{subfigure}

  \vspace{0.5cm}

  \begin{subfigure}{0.48\textwidth}
    \centering
    \includegraphics[width=\linewidth]{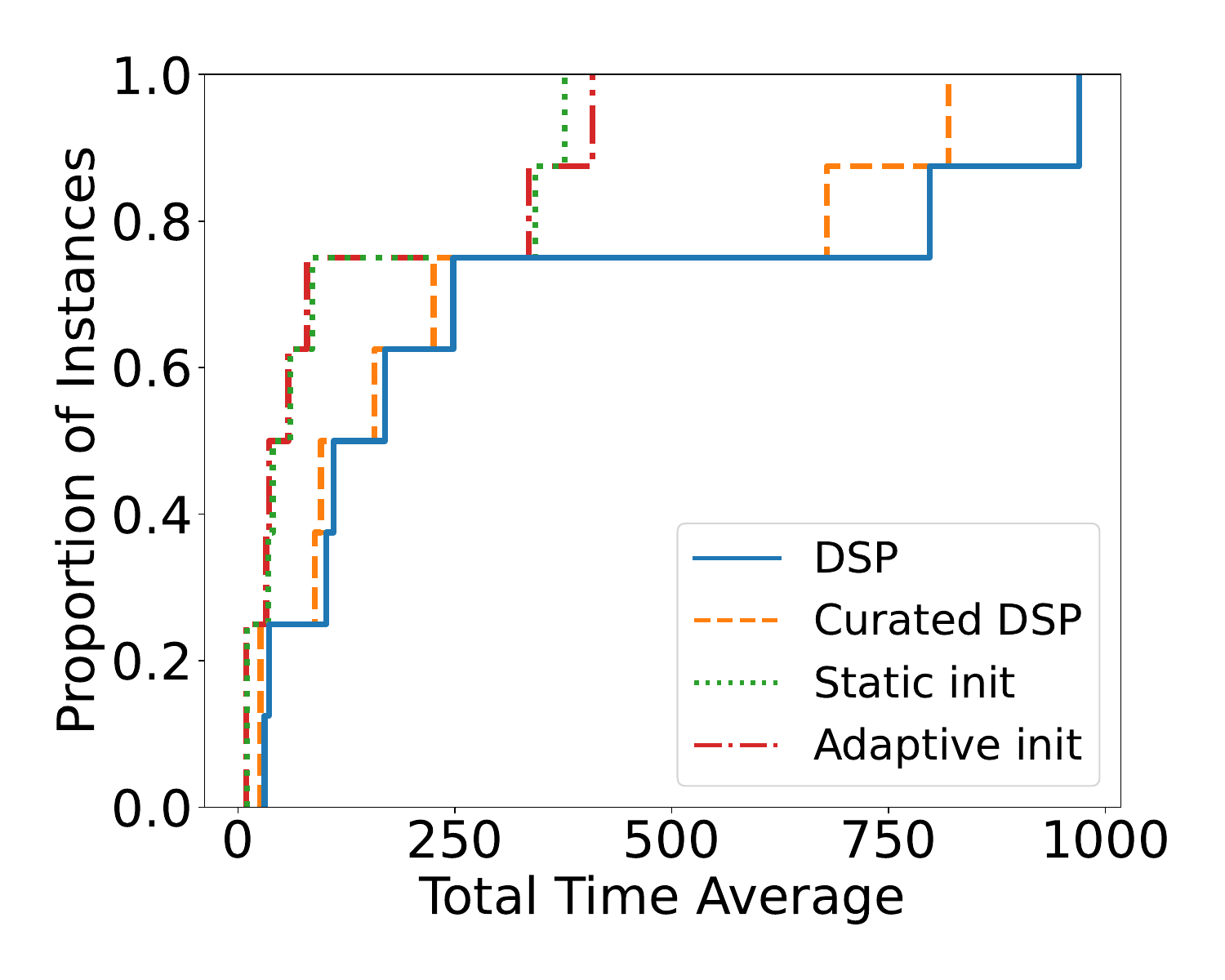}
    \caption{UC}
  \end{subfigure}
  \caption{Plots showing the fraction 
    of solved
LP instances over time for CFLP, CMND, and UC problems.}
  \label{fig:CDF_LP}
\end{figure}
Tables \ref{tab:lp_results_cflp} - \ref{tab:lp_results_uc}
display the summary results of the different methods for solving the
LP test instances for CFLP, CMND, and UC problems, respectively.
The tables
demonstrate the
significant benefits of information reuse techniques across all three problem classes.
The introduction of DSP drastically reduces \someformat{SP count},
suggesting that we are usually able to find violated cuts in the DSP, 
and only occasionally need to solve subproblems to generate cuts.
This validates the presence of valuable dual solutions within the DSP,
and their effectiveness in generating violated cuts. 
The reduction in \someformat{SP count} directly contributes to savings in \someformat{SP T}, 
leading to the observed savings in \someformat{Total T}.
A somewhat surprising result is the decrease in \someformat{iterations} of the 
algorithm from baseline to DSP for CFLP and CMND.

As hoped, we observe curated DSP  reduces \someformat{DSP T} and hence leads to a reduction in \someformat{Cut T} and ultimately \someformat{Total T}.
Intuitively, one might expect \someformat{SP T} to increase from DSP to curated DSP, 
as we have fewer dual solutions in the curated pool. 
However, the curation does not lead to a substantial increase in \someformat{SP T}
suggesting that curated DSP obtains a good trade-off in the time saved from searching the DSP against the small extra time spent solving the subproblems

Both static and adaptive initialization methods 
consistently outperform baseline and DSP methods, 
needing fewer \someformat{iterations} to converge.
This shows the value of initializing the algorithm with Benders cuts. 
This reduction in \someformat{iterations} directly contributes to
these methods having the shortest overall \someformat{Total T}.
Interestingly, adaptive initialization needs the fewest 
\someformat{iterations} to converge.
This suggests that it is able to identify useful cuts.
However, we do not see proportional decrease in \someformat{Total T} because 
adaptive initialization needs more time to find these initial cuts. 
Also, adaptive initialization usually adds more cuts in the main problem, 
leading to longer time to solve the main problem.
Overall,
using the combination of information reuse methods, we are able to solve CFLP and CMND problems approximately 10 times faster
compared to baseline on average, while for UC we observe approximately 5 times speedup.
The more modest improvement for UC is because the subproblems are easier to solve.
As a result, \someformat{Cut T} makes up a smaller share of \someformat{Total T}.
Therefore, DSP and curated DSP, which primarily reduce \someformat{Cut T}, have limited impact on \someformat{Total T}.
Hence, most of the improvement for UC comes from the initialization methods,
rather than from DSP or curated DSP.

We complement the summary results presented in Tables \ref{tab:lp_results_cflp} - \ref{tab:lp_results_uc}
with the total time cumulative distribution function (CDF) plots shown in Figure \ref{fig:CDF_LP}. In these figures, the Y-axis shows the proportion of solved instances
and the X-axis represents time. To focus on the relative improvements beyond just using DSP, these figures display only the four methods that reuse information and exclude the baseline method.
Broadly, we observe that curated DSP improves significantly over DSP and that
both initialization methods improve significantly over curated DSP, while they are comparable to each other.

\subsubsection{IP Results}
\label{sub:IP results}
\begin{table}[h]
\begin{center}
\begin{tabular}{lrrrrrrr}
\hline
Method       & \multicolumn{1}{c}{\someformat{Total T}} & \multicolumn{1}{c}{\someformat{IP T}} & \multicolumn{1}{c}{\someformat{LP T}} & \multicolumn{1}{c}{\someformat{Cut T}} & \multicolumn{1}{c}{\someformat{DSP T}} & \multicolumn{1}{c}{\someformat{SP T}} & \multicolumn{1}{c}{\someformat{Init T}} \\
\hline
Baseline     & 387.5                       & 327.1                    & 44.9                     & 163.9                    & -                         & 163.9                     & 0.6                        \\
DSP          & 244.0                       & 212.6                    & 20.0                     & 61.0                     & 56.9                      & 3.0                       & {0.5}                        \\
Curated DSP  & 216.9                       & 191.4                    & 14.7                     & 46.2                     & 41.4                      & 3.8                       & {0.5}                        \\
Static init  & 194.1                       & 182.4                    & 6.1                      & 43.9                     & 39.2                      & 3.8                       & 0.6                        \\
Adaptive init& {102.2}                       & {77.7}                     & {5.7}                      & {5.7}                      & {4.5}                       & {1.2}                       & 13.7                       \\
\hline
\end{tabular}
\end{center}
\caption{IP results: CFLP - Part 1.}
\label{tab:ip_results_cflp_part1}
\end{table}
\begin{table}[h]
\begin{center}
\begin{tabular}{lrrrrrrr}
\hline
Method       & \multicolumn{1}{c}{\someformat{Total T}} & \multicolumn{1}{c}{\someformat{IP T}} & \multicolumn{1}{c}{\someformat{LP T}} & \multicolumn{1}{c}{\someformat{Cut T}} & \multicolumn{1}{c}{\someformat{DSP T}} & \multicolumn{1}{c}{\someformat{SP T}} & \multicolumn{1}{c}{\someformat{Init T}} \\
\hline
Baseline     & 667.2                       & 537.8                    & 89.2                     & 259.2                     & -                         & 259.2                    & 1.5                       \\
DSP          & 100.9                       & 75.2                     & 18.2                     & 26.1                      & 23.2                      & 2.9                       & {1.0}                       \\
Curated DSP  & 61.8                        & 45.8                     & 9.8                      & 10.5                      & 7.0                       & 3.4                       & {1.0}                       \\
Static init  & 58.5                        & 47.2                     & 6.5                      & 9.4                       & 6.3                       & 2.9                       & 1.1                       \\
Adaptive init & {40.5}                       & {21.2}                     & {6.1}                      & {1.6}                       & {1.2}                       & {0.5}                       & 8.4                       \\
\hline
\end{tabular}
\end{center}
\caption{IP Results: CMND - Part 1.}
\label{tab:ip_results_cmnd_part1}
\end{table}

\begin{table}[h]
\begin{center}
\begin{tabular}{lrrrrrrr}
\hline
Method       & \multicolumn{1}{c}{\someformat{Total T}} & \multicolumn{1}{c}{\someformat{IP T}} & \multicolumn{1}{c}{\someformat{LP T}} & \multicolumn{1}{c}{\someformat{Cut T}} & \multicolumn{1}{c}{\someformat{DSP T}} & \multicolumn{1}{c}{\someformat{SP T}} & \multicolumn{1}{c}{\someformat{Init T}} \\
\hline
Baseline     & 313.6                       & 6.7                      & 300.0                    & 2.6                      & -                         & 2.6                       & 1.6                        \\
DSP          & 220.4                       & 11.9                     & 180.0                    & 3.2                      & 1.9                       & 1.3                       & 1.3                        \\
Curated DSP  & 194.5                       & 11.1                     & 145.2                    & 1.9                      & 0.6                       & 1.2                       & 1.3                        \\
Static init  & 101.6                       & 10.8                     & 65.5                     & 1.8                      & 0.5                       & 1.3                       & 1.9                        \\
Adaptive init& {80.3}                        & {6.1}                      & {58.3}                     & {1.3}                      & {0.2}                       & {1.1}                       & 4.9                        \\
\hline
\end{tabular}
\end{center}
\caption{IP results: UC - Part 1.}
\label{tab:ip_results_uc_part1}
\end{table}
Tables \ref{tab:ip_results_cflp_part1} - \ref{tab:ip_results_uc_part1}
display the summary results of the different methods for solving the
IP test instances for CFLP, CMND, and UC problems, respectively.
For IPs, we solve the LP relaxation first to obtain a good initialization for branch-and-bound.
Table \ref{tab:ip_results_uc_part1} reveals that for UC, 
\someformat{LP T} dominates \someformat{Total T}.
This computational profile differs from CFLP and CMND,
where IP solving consumes the majority of time.

For CFLP and CMND, we find that DSP reduces \someformat{SP T}
which translates into savings in \someformat{Total T}. 
Curated DSP reduces the \someformat{DSP T} because of a smaller pool,
in turn also helping to decrease \someformat{Total T}.
Similar to the LP results,
curated DSP does not lead to a significant increase in \someformat{SP T},
suggesting that curated DSP provides a good trade-off in the time checking the DSP against the time spent solving subproblems.
For UC, \someformat{Cut T} remains similar across most methods because the tight LP relaxation requires few additional cuts during the IP solving phase.

For all the test problems, 
static and adaptive initialization methods consistently outperform baseline and DSP techniques.
While the two initialization methods performed comparably for LPs,
we observe a clear distinction for IPs.
For CFLP, static initialization offers only a marginal improvement over curated DSP,
whereas adaptive initialization demonstrates a substantial two-fold improvement in total time 
and a ten-fold reduction in cut generation time.
This highlights the effectiveness of adaptive initialization in identifying strong initial cuts.
Data for CMND shows a similar trend,
with adaptive initialization again proving superior and static initialization providing 
only a marginal benefit over curated DSP.
For UC, both initialization schemes significantly outperform the DSP methods,
  with adaptive initialization maintaining a slight edge over static initialization.
For all the problem classes, although adaptive initialization has a higher initial computational cost (\someformat{Init T}),
it delivers superior overall performance (\someformat{Total T}),
making it the best choice for IPs.

\begin{table}[h]
\begin{center}
\begin{tabular}{lrrrr}
\hline
Method       & \multicolumn{1}{c}{\someformat{Nodes}} & \multicolumn{1}{c}{\someformat{Root gap (\%)}} & \multicolumn{1}{c}{\someformat{Callback calls}} & \multicolumn{1}{c}{\someformat{SP count}} \\
\hline
Baseline     & 2814.5                    & 2.9                          & 115.7                        & 115.7                        \\
DSP          & 2819.9                    & 3.1                          & 116.6                        & 6.1                          \\
Curated DSP  & 2824.7                    & 3.1                          & 114.8                        & 6.6                          \\
Static init  & 2645.7                    & 3.1                          & 111.0                        & 6.5                          \\
Adaptive init& {2129.8}                    & {2.8}                          & {13.4}                         & {1.9}                          \\
\hline
\end{tabular}
\end{center}
\caption{IP results: CFLP - Part 2.}
\label{tab:ip_results_cflp_part2}
\end{table}

\begin{table}[h]
\begin{center}
\begin{tabular}{lrrrr}
\hline
Method       & \multicolumn{1}{c}{\someformat{Nodes}} & \multicolumn{1}{c}{\someformat{Root gap (\%)}} & \multicolumn{1}{c}{\someformat{Callback calls}} & \multicolumn{1}{c}{\someformat{SP count}} \\
\hline
Baseline     & 2404.9                    & 12.7                         & 140.1                         & 140.1                         \\
DSP          & 1446.1                    & 11.6                         & 71.2                          & 5.0                           \\
Curated DSP  & {1357.7}                    & 11.2                         & 64.5                          & 5.5                           \\
Static init  & 1420.7                    & 9.3                          & 62.5                          & 5.2                           \\
Adaptive init & 1379.0                   & {3.9}                          & {10.8}                          & {1.2}                           \\
\hline
\end{tabular}
\end{center}
\caption{IP Results: CMND - Part 2.}
\label{tab:ip_results_cmnd_part2}
\end{table}

\begin{table}[h]
\begin{center}
\begin{tabular}{lrrrr}
\hline
Method       & \multicolumn{1}{c}{\someformat{Nodes}} & \multicolumn{1}{c}{\someformat{Root gap (\%)}} & \multicolumn{1}{c}{\someformat{Callback calls}} & \multicolumn{1}{c}{\someformat{SP count}} \\
\hline
Baseline     & 2.5                     & 0.08                         & 3.2                          & 3.2                          \\
DSP          & 9.3                     & 0.08                         & 3.3                          & 2.2                          \\
Curated DSP  & 8.3                     & 0.08                         & 3.6                          & 2.1                          \\
Static init  & 8.6                     & 0.08                         & 3.2                          & 2.2                          \\
Adaptive init& {2.0}                     & 0.08                         & {2.3}                          & {2.0}                          \\
\hline
\end{tabular}
\end{center}
\caption{IP Results: UC - Part 2.}
\label{tab:ip_results_uc_part2}
\end{table}

\begin{figure}[h]
  \centering
  \begin{subfigure}{0.48\textwidth}
    \centering
    \includegraphics[width=\linewidth]{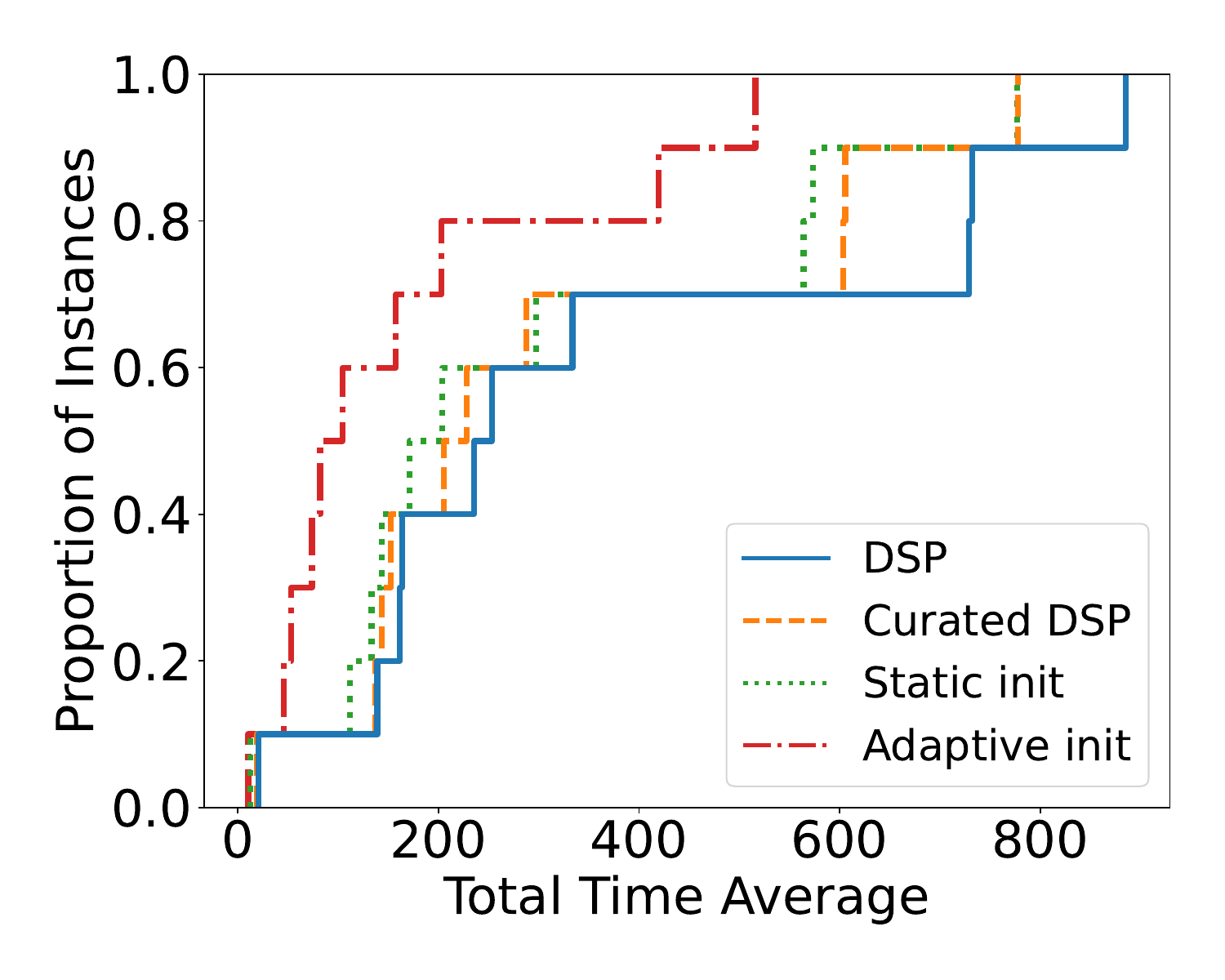}
    \caption{CFLP}
  \end{subfigure}
  \hfill
  \begin{subfigure}{0.48\textwidth}
    \centering
    \includegraphics[width=\linewidth]{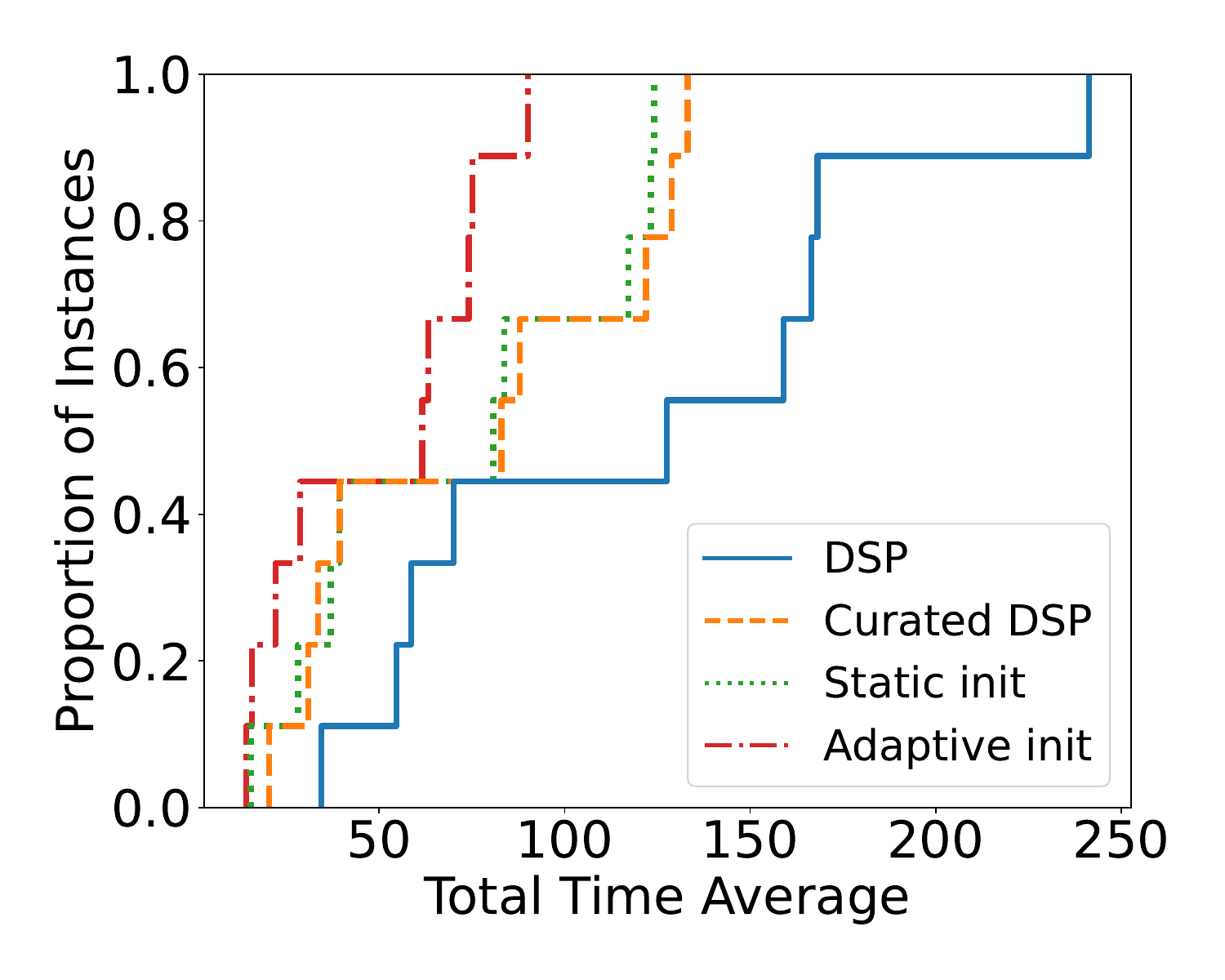}
    \caption{CMND}
  \end{subfigure}

  \vspace{0.5cm}

  \begin{subfigure}{0.48\textwidth}
    \centering
    \includegraphics[width=\linewidth]{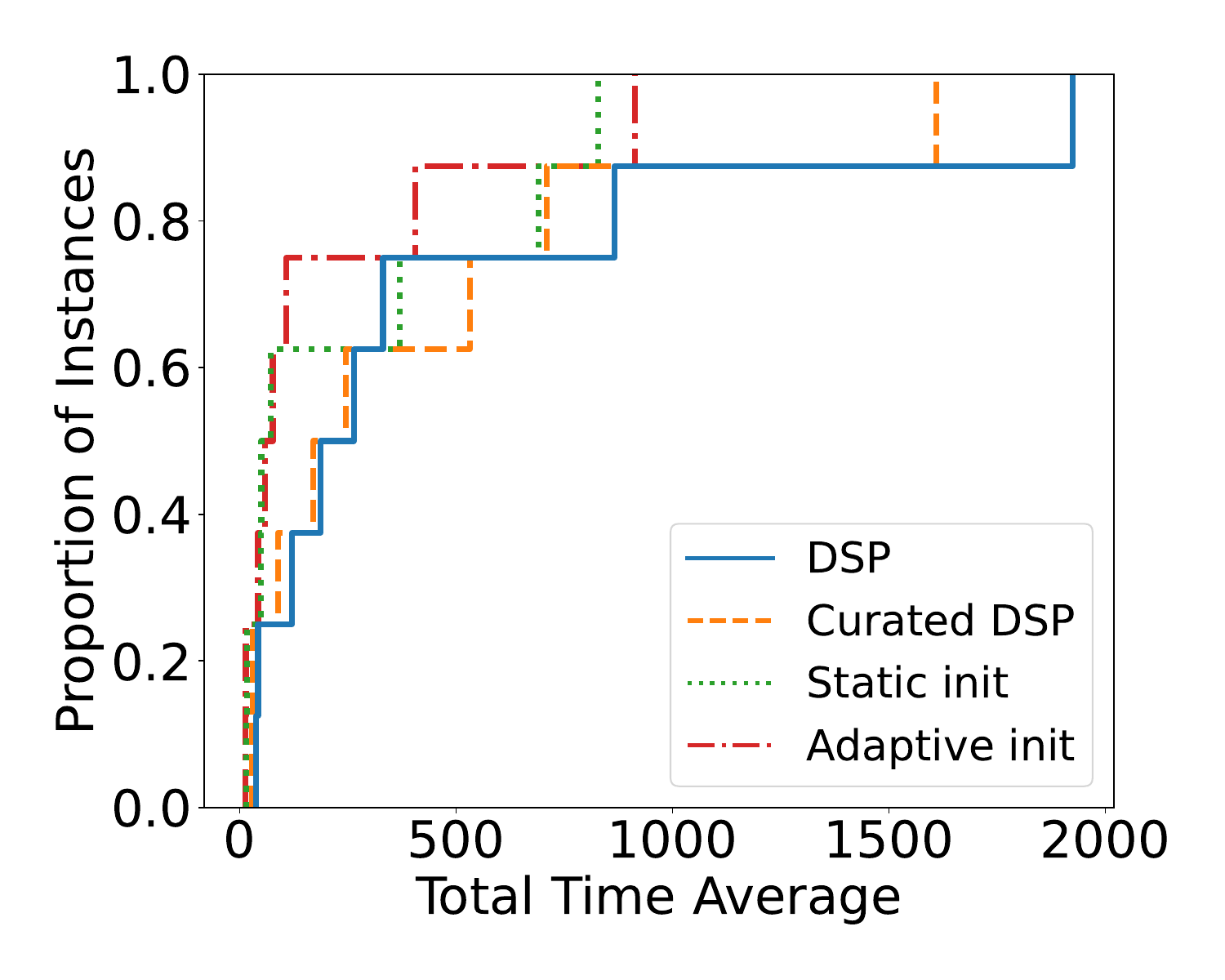}
    \caption{UC}
  \end{subfigure}
  \caption{Plots showing the fraction 
    of solved
LP instances over time for CFLP, CMND, and UC problems.}
\label{fig:CDF_IP}
\end{figure}

Tables \ref{tab:ip_results_cflp_part2} - \ref{tab:ip_results_uc_part2}
present additional branch-and-bound metrics for all three problem classes.
For CFLP and CMND, we observe significant reductions in \someformat{SP count}
after employing DSP,
similar to what we saw for LPs.
Interestingly, deploying adaptive initialization leads to a significant reduction 
in both the \someformat{callback calls} and \someformat{SP count}
compared to other methods.
This suggests that the initial cuts generated by adaptive initialization 
provide a strong relaxation 
of the original problem,
minimizing the need for looking for 
additional cuts during the optimization process.
Furthermore, adaptive initialization improves root gaps,
particularly for CMND.
Since \someformat{root gaps} help isolate the impact of initialization,
this demonstrates that adaptive initialization closes the most gap among all methods.
The resulting reduction in \someformat{callback calls} 
confirms that the initial cuts are highly effective.
For UC, Table \ref{tab:ip_results_uc_part2} shows extremely low root gaps (0.07\%) and very few branch-and-bound node counts,
confirming the near-integrality of the LP relaxation.
This indicates that once the LP is solved,
minimal additional work is required to obtain the IP solution.
This behavior reinforces our observation that UC IP performance is dominated by LP solving.

Figure \ref{fig:CDF_IP} presents the cumulative distribution function (CDF) plots of solution times
for IP instances, again displaying the results only for the methods that reuse information.
These plots further illustrate that the adaptive initialization method has the best performance.
We also see that curated DSP generally leads to an improvement 
in total time compared to regular DSP.


\subsubsection{Solution Quality with Equal Time Budgets}
\label{sub:Solution Quality}


The results presented above demonstrate significant speedups in solution times when using our information reuse methods, particularly adaptive initialization.
As these results focus on time to solve to optimality, they do not give a picture of of how much improvement in solution
quality and gap is achieved earlier in the solution process.
To investigate this, we conduct additional experiments comparing solution quality obtained by the baseline method when
limited to a time comparable to the time it takes the adaptive method to solve to optimality.

For each test instance, we first ran the adaptive initialization method on the 25 replications after the first one and recorded its average solution time $T$.
We then ran the baseline method on the first replication of the same instance with two time limits: $T$ and $2T$,
recording whether it proved optimality and, if not, the final optimality gap at termination.

Tables~\ref{tab:lp_quality_timelimit} and \ref{tab:ip_quality_timelimit} present the average optimality gaps across instances.
The \someformat{Gap - T (\%)} column shows the optimality gap when the baseline is limited to the time limit $T$,
while \someformat{Gap - 2T (\%)} shows the gap when the baseline is given twice that time limit.
All runs hit the solver timeout without proving optimality.

\begin{table}[h]
\centering
\begin{tabular}{lrr}
\hline
Problem & \multicolumn{1}{c}{\someformat{Gap - T (\%)}} & \multicolumn{1}{c}{\someformat{Gap - 2T (\%)}} \\
\hline
CMND & 12474.8 & 4061.8 \\
CFLP & 15.0 & 6.5 \\
UC & 9.4 & 1.0 \\
\hline
\end{tabular}
\caption{LP relaxation: Average optimality gaps with time limits.}
\label{tab:lp_quality_timelimit}
\end{table}

\begin{table}[h]
\centering
\begin{tabular}{lrr}
\hline
Problem & \multicolumn{1}{c}{\someformat{Gap - T (\%)}} & \multicolumn{1}{c}{\someformat{Gap - 2T (\%)}} \\
\hline
CMND & 11.1 & 10.7 \\
CFLP & 3.5 & 2.4 \\
UC & - & - \\
\hline
\end{tabular}
\caption{IP: Average optimality gaps with time limits.}
\label{tab:ip_quality_timelimit}
\end{table}

These results confirm that the baseline's longer solution times are not simply due to closing a small amount of
remaining optimality gap.
Even when run for twice the time it takes the adaptive method to solve the instances to optimality,
the baseline fails to achieve the solution quality that the adaptive initialization obtains.
For IP instances, the gaps are smaller but still significant,
demonstrating that our methods provide fundamental algorithmic improvements across both LP and IP formulations.

Note that for some IP instances, the baseline did not complete the initial LP relaxation within the time limit.
For UC-IP, all instances timeout on the LP relaxation for both time limits $T$
  and $2T$, hence no gap values are reported in
  Table~\ref{tab:ip_quality_timelimit}.
For CMND-IP, four instances did not finish the LP relaxation within time $T$ and 2 instances did not finish within $2T$.
For CFLP-IP, one instance did not finish the LP relaxation within $T$, but all instances completed it within $2T$.
The gap averages reported in Table~\ref{tab:ip_quality_timelimit} are calculated only for instances that completed the LP relaxation, as a gap cannot be computed otherwise. 

\subsubsection{Cut Distribution}
\label{sub:Cut Distribution}

\begin{table}[h]
  \begin{center}
\begin{tabular}{lrrrr}
  \hline
Method & \multicolumn{1}{c}{\someformat{Initial cuts}} & \multicolumn{1}{c}{\someformat{SP cuts}} & \multicolumn{1}{c}{\someformat{DSP cuts}} & \multicolumn{1}{c}{\someformat{Total cuts}} \\
  \hline
Baseline                   & -                                & 39192                       & -                            & 39192                          \\
DSP                        & -                                & 448                        & 33455                        & 33903                          \\
Curated DSP                & -                                & 589                        & 34139                        & 34728                          \\
Static Init                & 784                              & 608                        & 33327                        & 34719                          \\
Adaptive Init               & 13105                             & 128                         & {3610}                         & 16843                         \\
  \hline
\end{tabular}
  \end{center}
\caption{Cut distribution for the CFLP instance with 35 facilities and 105 customers.}
\label{tab:cflp_cut_distribution}
\end{table}

\begin{table}[h]
  \begin{center}
\begin{tabular}{lrrrr}
  \hline
Method & \multicolumn{1}{c}{\someformat{Initial cuts}} & \multicolumn{1}{c}{\someformat{SP cuts}} & \multicolumn{1}{c}{\someformat{DSP cuts}} & \multicolumn{1}{c}{\someformat{Total cuts}} \\
  \hline
Baseline     & -                                & 55721                       & -                            & 55721                          \\
DSP          & -                                & 85                           & 29851                        & 29936                          \\
Curated DSP  & -                                & 128                          & 27604                        & 27732                          \\
Static Init  & 782                              & 50                          & 26455                        & 27287                          \\
Adaptive Init & {8723}                            & 24                          & {3629}                         & 12376                         \\
  \hline
\end{tabular}
  \end{center}
\caption{Cut distribution for the CMND instance r03.3.}
\label{tab:cmnd_cut_distribution}
\end{table}
\begin{table}[h]
  \begin{center}
\begin{tabular}{lrrrr}
  \hline
Method & \multicolumn{1}{c}{\someformat{Initial cuts}} & \multicolumn{1}{c}{\someformat{SP cuts}} & \multicolumn{1}{c}{\someformat{DSP cuts}} & \multicolumn{1}{c}{\someformat{Total cuts}} \\
  \hline
Baseline                   & -                                & 667                         & -                            & 667                            \\
DSP                        & -                                & 116                         & 1223                         & 1339                           \\
Curated DSP                & -                                & 16                          & 1011                         & 1027                           \\
Static Init                & 726                              & 48                          & 161                          & 935                            \\
Adaptive Init              & 1520                             & 0                           & 176                          & 1696                           \\
  \hline
\end{tabular}
  \end{center}
\caption{Cut distribution for the UC instance with 20 generators and difficulty 3.}
\label{tab:uc_cut_distribution}
\end{table}
Our algorithm generates Benders cuts from three sources:
cuts provided during initialization (\someformat{Initial cuts}),
cuts derived by solving subproblems during the algorithm (\someformat{SP cuts}),
and cuts obtained by searching a pool of dual solutions (\someformat{DSP cuts}).
To provide more insight about our methods, we present in Tables \ref{tab:cflp_cut_distribution} - \ref{tab:uc_cut_distribution} the distribution of cuts of each type for sample IP instances of the CFLP, CMND, and UC problems, respectively.
To focus on differences between the methods, the \someformat{Initial cuts} and \someformat{Total cuts} exclude active LP cuts, which are nearly identical for all methods.

We find that the inclusion of DSP in the Benders decomposition framework significantly reduces the
number of subproblems solved to generate cuts, leading to a dramatic decrease in \someformat{SP cuts} across all problem types.
This is expected and confirms the presence of useful dual solutions within the pool,
capable of generating violated cuts.
For CFLP and CMND, we observe that curated DSP requires slightly more \someformat{SP cuts} than DSP,
which can be attributed to the reduced size of the
curated DSP compared to the DSP.
However, for the UC instance, curated DSP reduces \someformat{SP cuts} compared to DSP.
Using either of the initialization methods leads to a further reduction
in the number of subproblem solves required to generate cuts relative to the curated DSP.
For CFLP and CMND, we observe a drastic reduction in the number of total cuts needed
to reach the optimal solution when using adaptive initialization, which is explained by the significant decrease in the number of cuts added from the DSP.
The adaptive initialization method requires very few additional cuts beyond the initial
set introduced during problem initialization,
which translates to significant time savings in solving the SAA replication.
In contrast, for UC, the number of cuts added beyond the LP relaxation is relatively small across all methods, whether through initialization or by solving subproblems.
This suggests that for UC instances, the LP relaxation cuts provide most of the constraints needed to get to the optimal IP solution,
with initialization and DSP methods offering more modest incremental improvements.

\subsubsection{Impact of Adaptive Initialization}
Analysis of the cut distribution in
Tables \ref{tab:cflp_cut_distribution} - \ref{tab:uc_cut_distribution}
reveals that adaptive initialization
consistently introduces a higher number of initial cuts compared to static initialization.
However, for the CFLP and CMND instances, the difference in the number of initial cuts between static and adaptive initialization
is more drastic compared to the UC instance.
Therefore, to isolate the impact of cut quality from quantity on the observed performance improvements,
we focus our analysis on CFLP and CMND and conduct an experiment in which we modify the static initialization method to add the same number
of cuts as adaptive initialization for these two problem types. 
In this version of static initialization, which we refer to as boosted static initialization, we solve the LP relaxation first as usual, and
then use static initialization to add \( \left\lceil \frac{n}{K} \right\rceil \)
cuts per scenario to the main problem, where $n$ is the number of cuts that was added by the adaptive initialization method for the same instance.
For each scenario within the new SAA replication,
we identify the top \( \left\lceil \frac{n}{K} \right\rceil \) dual solutions from the DSP that lead to the highest value of 
the subproblem objective value function, $\qapp{\overline{x}, \pilist}$,
evaluated at the first two optimal solutions in \( \xlistopt \).
The cuts generated from these are then used to initialize the problem.
This ensures both initialization approaches use the same number of cuts.

In Fig. \ref{fig:static_boost}, we track the root node gap closed 
and also the average time taken to solve 
the \(i^{th}\) SAA replication over 26 replications using boosted static initialization and adaptive initialization.
The results demonstrate that adaptive initialization remains superior even when the static initialization adds the same number of cuts.
Thus, we conclude that the adaptive nature of adaptive initialization is important, and in particular
it seems to benefit from allowing the number of cuts added for each scenario to vary.


\begin{figure}[h]
  \centering
  \captionsetup{justification=centering}

  \begin{subfigure}{0.48\textwidth}
    \centering
    \includegraphics[width=\linewidth]{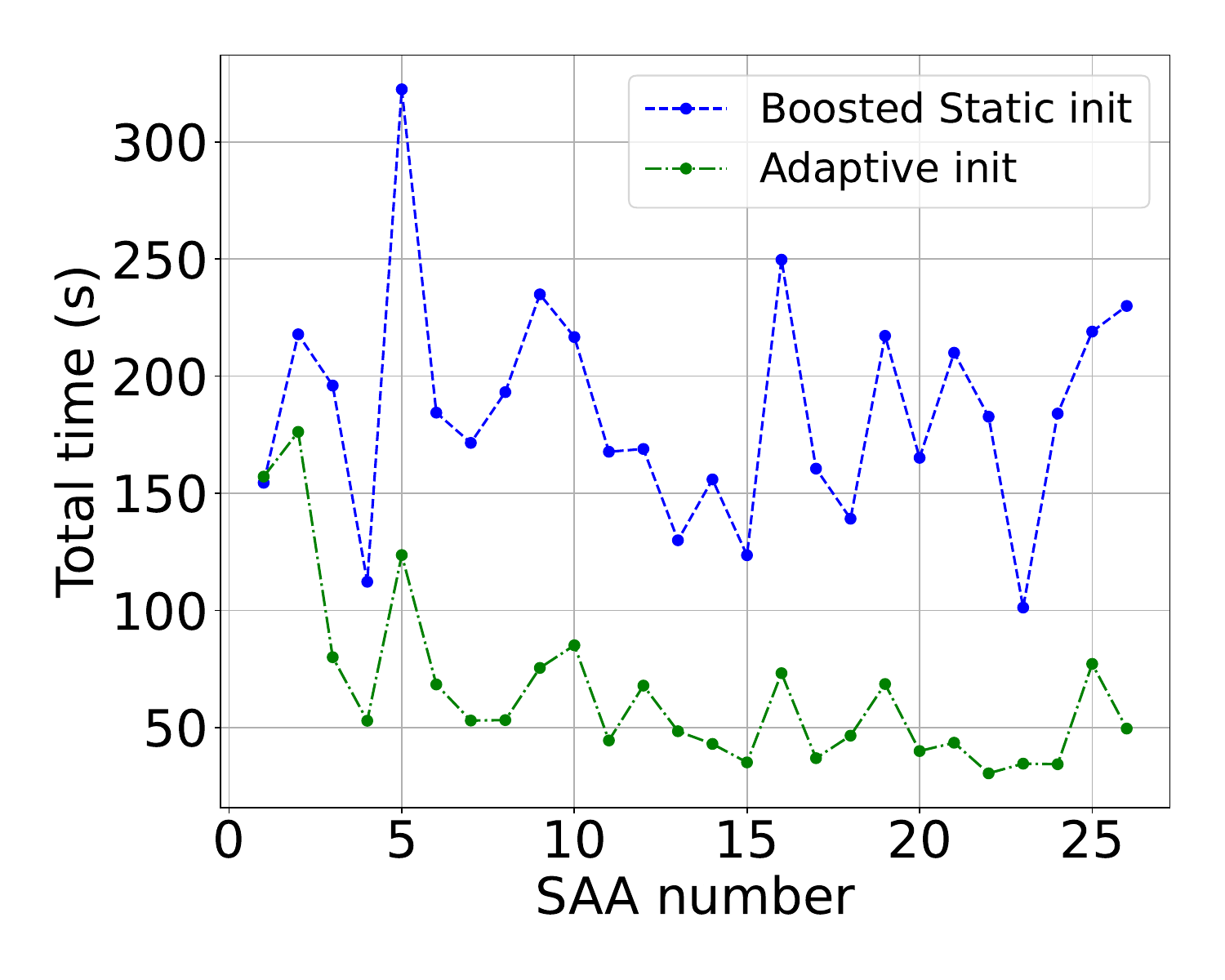}
    \caption{Total time - CFLP.}
  \end{subfigure}
  \hspace{0.02\textwidth} 
  \begin{subfigure}{0.48\textwidth}
    \centering
    \includegraphics[width=\linewidth]{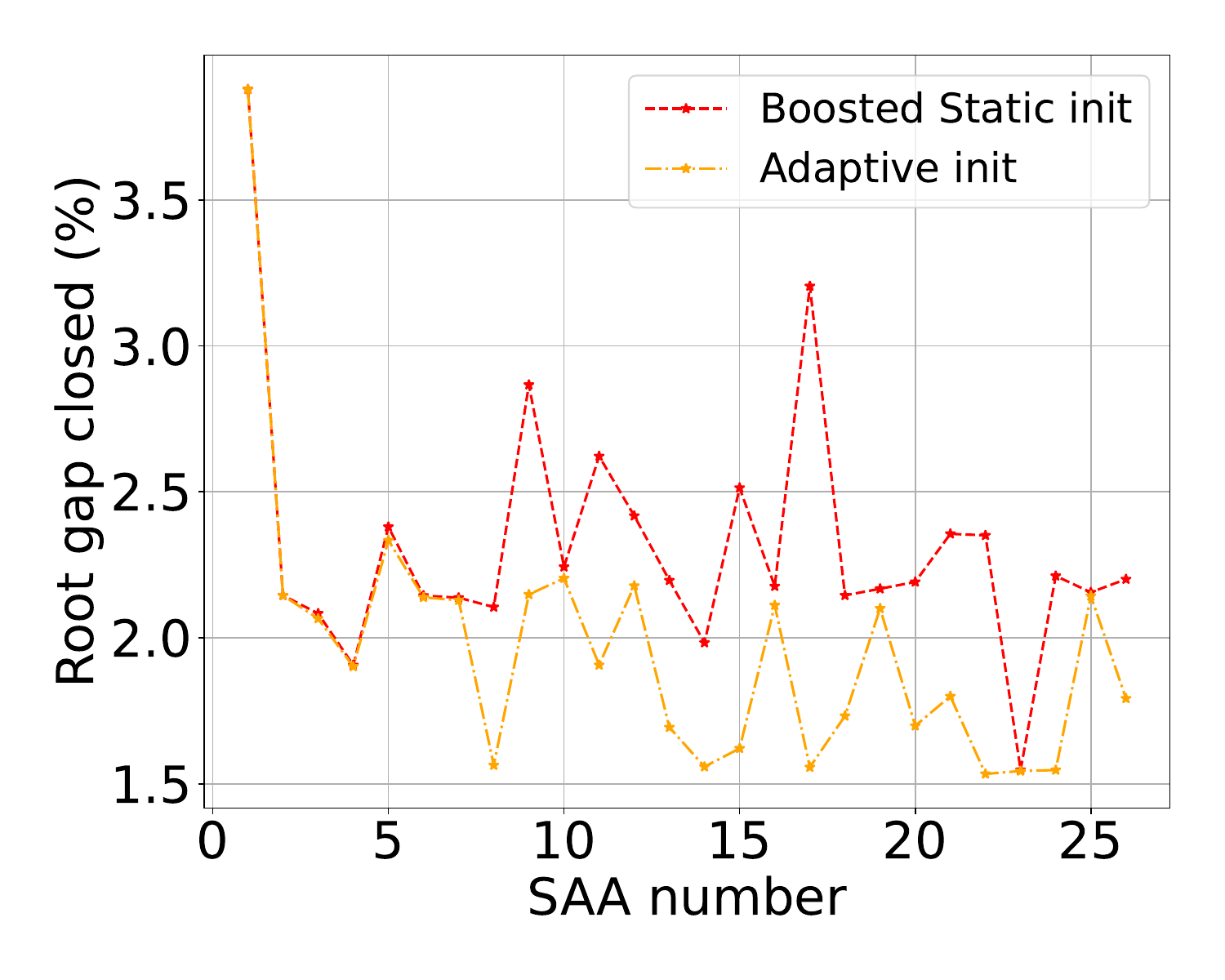}
    \caption{Root gap - CFLP.}
  \end{subfigure}
  \vspace{0.3cm} 

  \begin{subfigure}{0.48\textwidth}
    \centering
    \includegraphics[width=\linewidth]{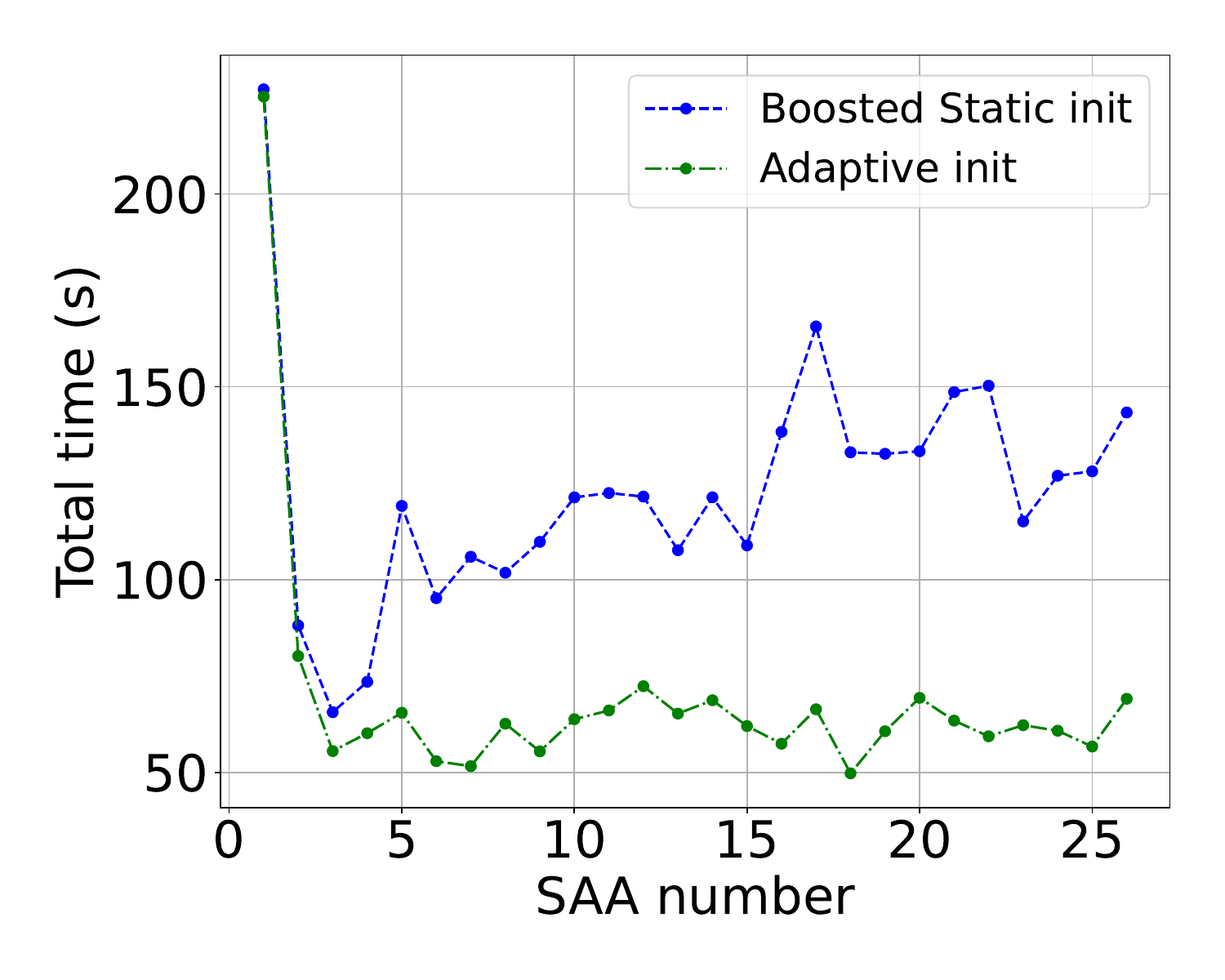}
    \caption{Total time - CMND.}
  \end{subfigure}
  \hspace{0.02\textwidth} 
  \begin{subfigure}{0.48\textwidth}
    \centering
    \includegraphics[width=\linewidth]{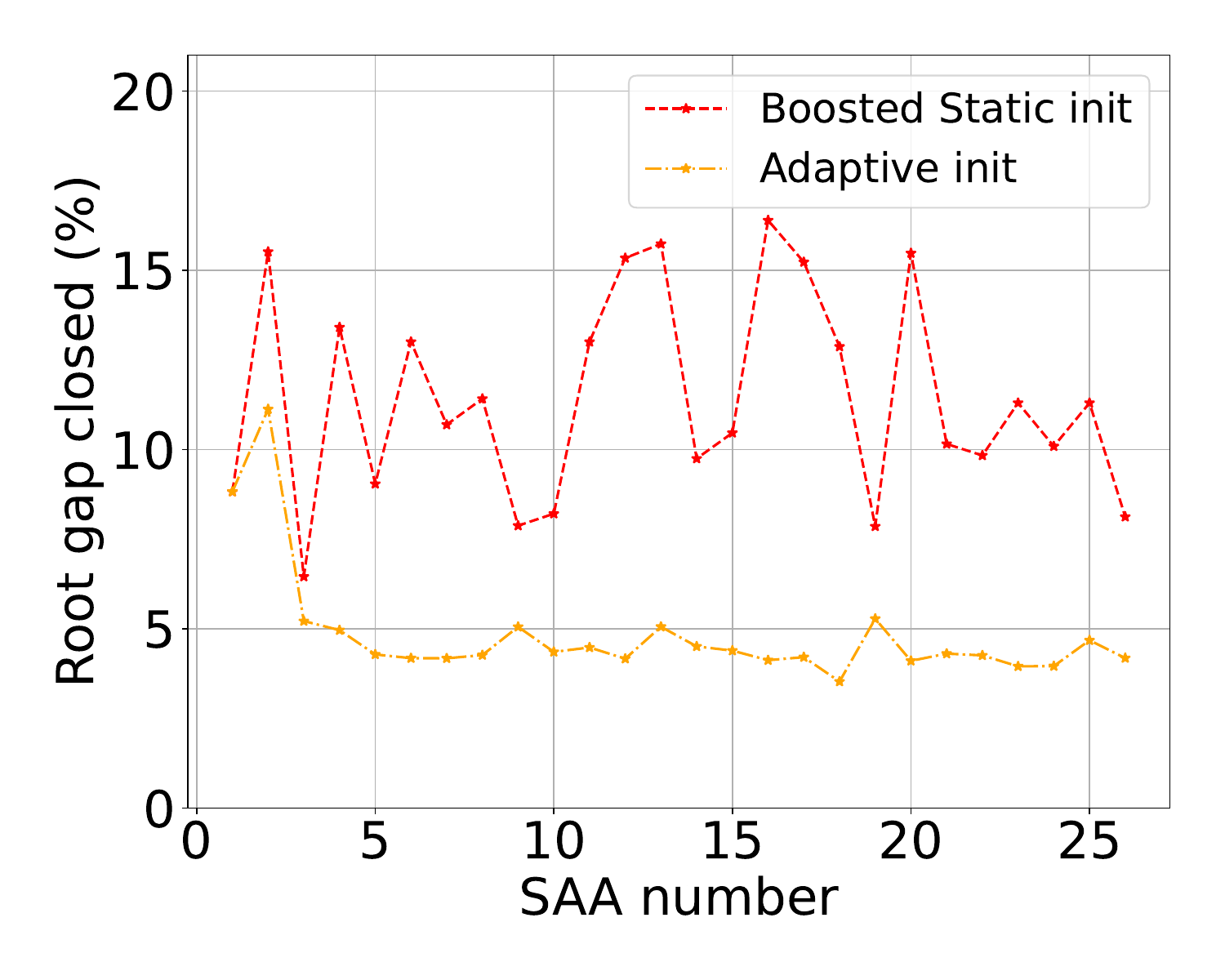}
    \caption{Root gap - CMND.}
  \end{subfigure}
  \caption{Total time and root gap trends for boosted static initialization and adaptive initialization per replication. The CFLP results (top plots) are for the instance with 35 facilities and 105 customers. The CMND results (bottom plots) are for the r03.3 instance.} 
  \label{fig:static_boost}
\end{figure}
%
%

\subsubsection{Impact of Curated DSP}

In Figure \ref{fig:cardinality_of_dsp}, we plot the size of the DSP and curated DSP
as we solve more SAA replications. For CMND and UC, we notice that, the DSP constantly
increases in size, but the curated DSP has a sharp decrease in size in the second replication.
Although it begins to grow again after that, the increase is not as significant compared to the DSP.
This indicates that we have successfully achieved our goal for the curated DSP, as we are able to maintain a controlled size of the pool.
In the case of CFLP, we notice similar trends, although the reduction in size from the first to the second replication is not as pronounced as in CMND.
Additionally, by the end of the 25 replications, the size of the pool is nearly the same as it was after the first replication.
\label{sub:impact of curated DSP}
\begin{figure}[h]
  \centering
    \begin{subfigure}{0.48\textwidth}
    \centering
    \includegraphics[width=\linewidth]{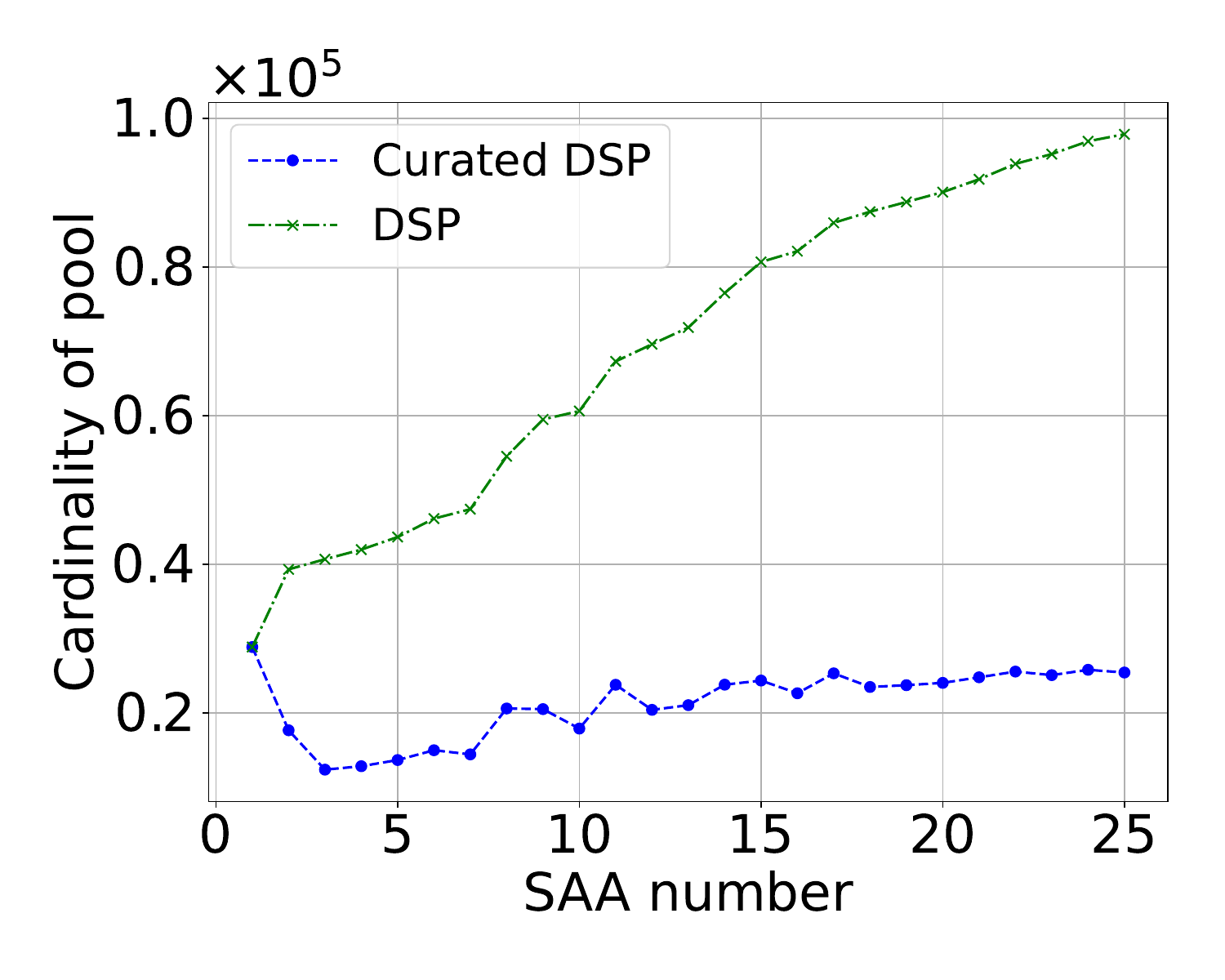}
    \caption{CFLP}
  \end{subfigure}
  \hfill
  \begin{subfigure}{0.48\textwidth}
    \centering
    \includegraphics[width=\linewidth]{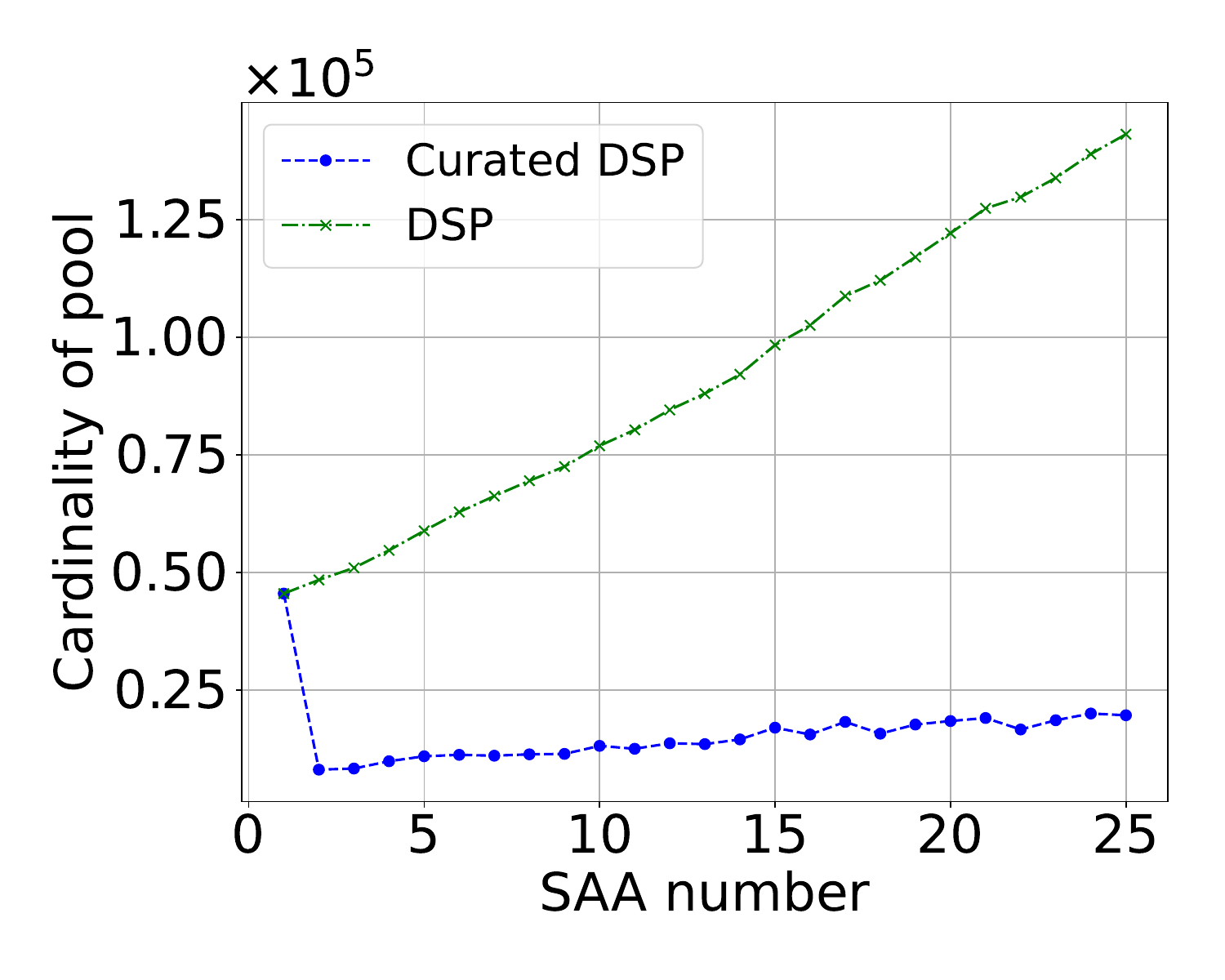}
    \caption{CMND}
  \end{subfigure}

  \vspace{0.5cm}

  \begin{subfigure}{0.48\textwidth}
    \centering
    \includegraphics[width=\linewidth]{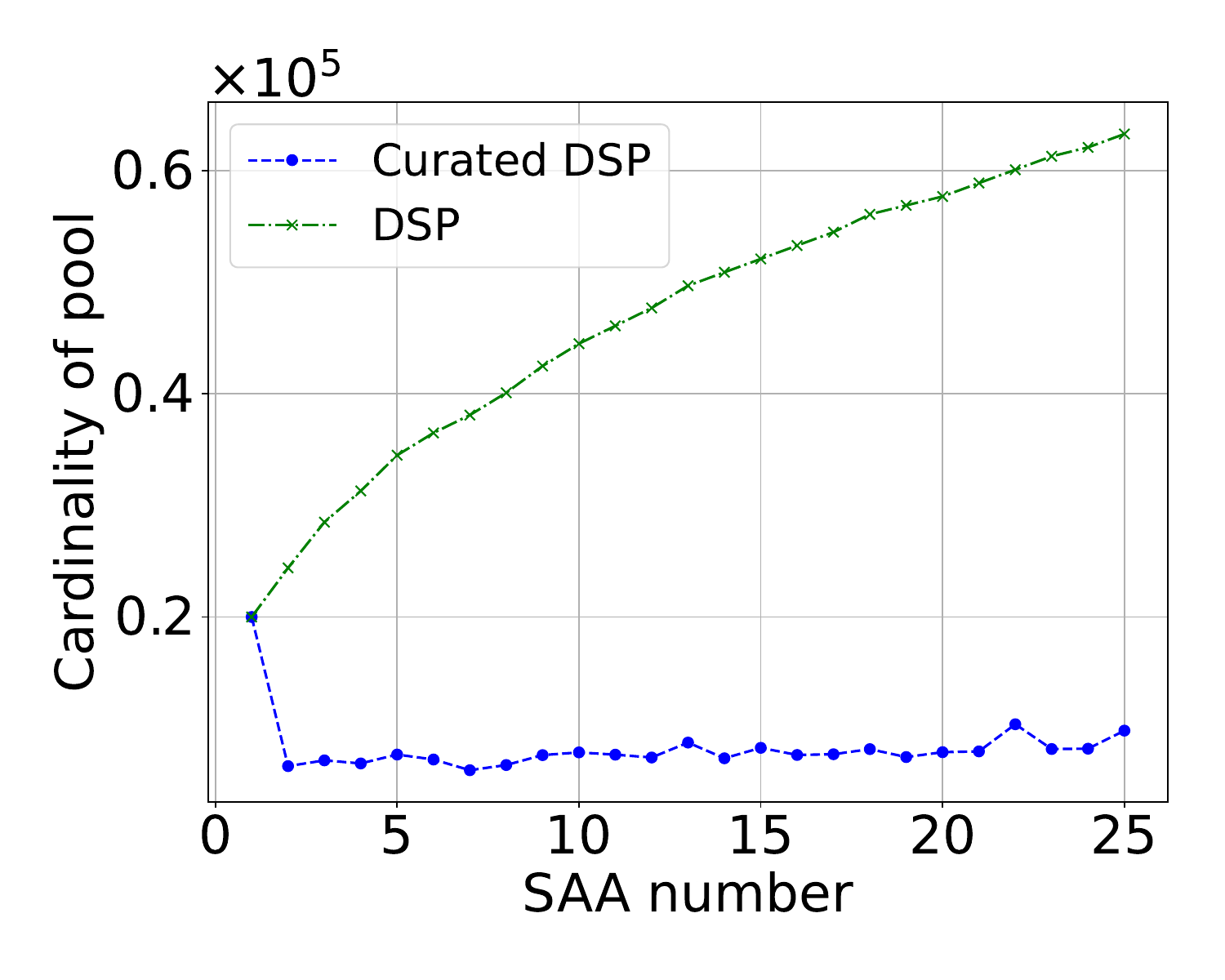}
    \caption{UC}
  \end{subfigure}
  \caption{Number of dual solutions in the pool for DSP and curated DSP.}
  \label{fig:cardinality_of_dsp}
\end{figure}

We next investigate the importance of our particular mechanism for choosing the dual solutions to keep in the curated DSP.
To do this,
we compared our method to a baseline that randomly 
selects the same number of dual solutions as were chosen in our curated DSP method.
We refer to this method as the random curated DSP.
\begin{table}[h]
  \begin{center}
\begin{tabular}{lrrr}
  \hline
Method & \multicolumn{1}{l}{CFLP} & \multicolumn{1}{l}{CMND} & \multicolumn{1}{l}{UC} \\
  \hline
Curated DSP   & 153.17                   & 89.56               & 228.42    \\
Random curated DSP    & 152.28                   & 92.19       & 286.46            \\
  \hline
\end{tabular}
\caption{Total time comparison for curated DSP and random DSP.}
\label{tab:random_dsp_time}
  \end{center}
\end{table}
In Table \ref{tab:random_dsp_time},
we display the average total time taken to solve 25 SAA replications after the first one with these two methods.
We use the same instances to test these methods as the experiments on cut distribution in Sec.
\ref{sub:Cut Distribution}.
For the UC instances, curated DSP achieves a notable improvement over random curated DSP.
This suggests that for UC, the smart selection mechanism of curated DSP successfully identifies more effective dual solutions compared to random selection.
In contrast, for CFLP and CMND instances, the performance is comparable between the two approaches, with no significant difference in solution times.
These results indicate that the primary benefit of curated DSP comes from
maintaining a smaller pool size. However, for some problem classes like UC,
the intelligent selection mechanism provides additional computational
gains by identifying more effective dual solutions.
Importantly, a key feature of our method is that it automatically determines the number of dual solutions
to retain in the curated DSP without requiring this number as input.

\subsubsection{Performance with Varying Number of Scenarios}
\label{sub:scenarios variation}

\begin{figure}[h]
  \centering
  \captionsetup{justification=centering}

  \begin{subfigure}{0.48\textwidth}
    \centering
    \includegraphics[width=\linewidth]{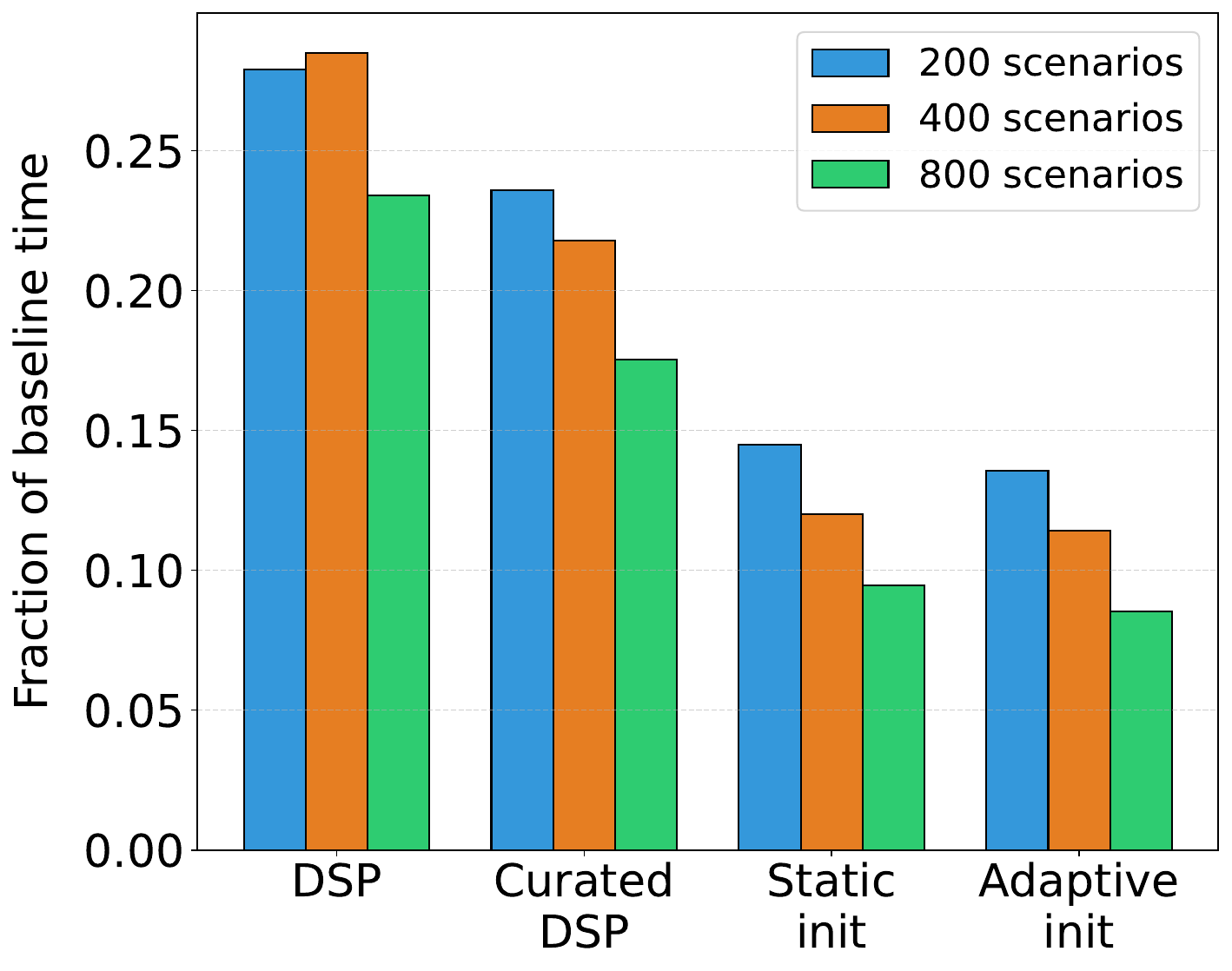}
    \caption{LP - CFLP}
  \end{subfigure}
  \hspace{0.02\textwidth} 
  \begin{subfigure}{0.48\textwidth}
    \centering
    \includegraphics[width=\linewidth]{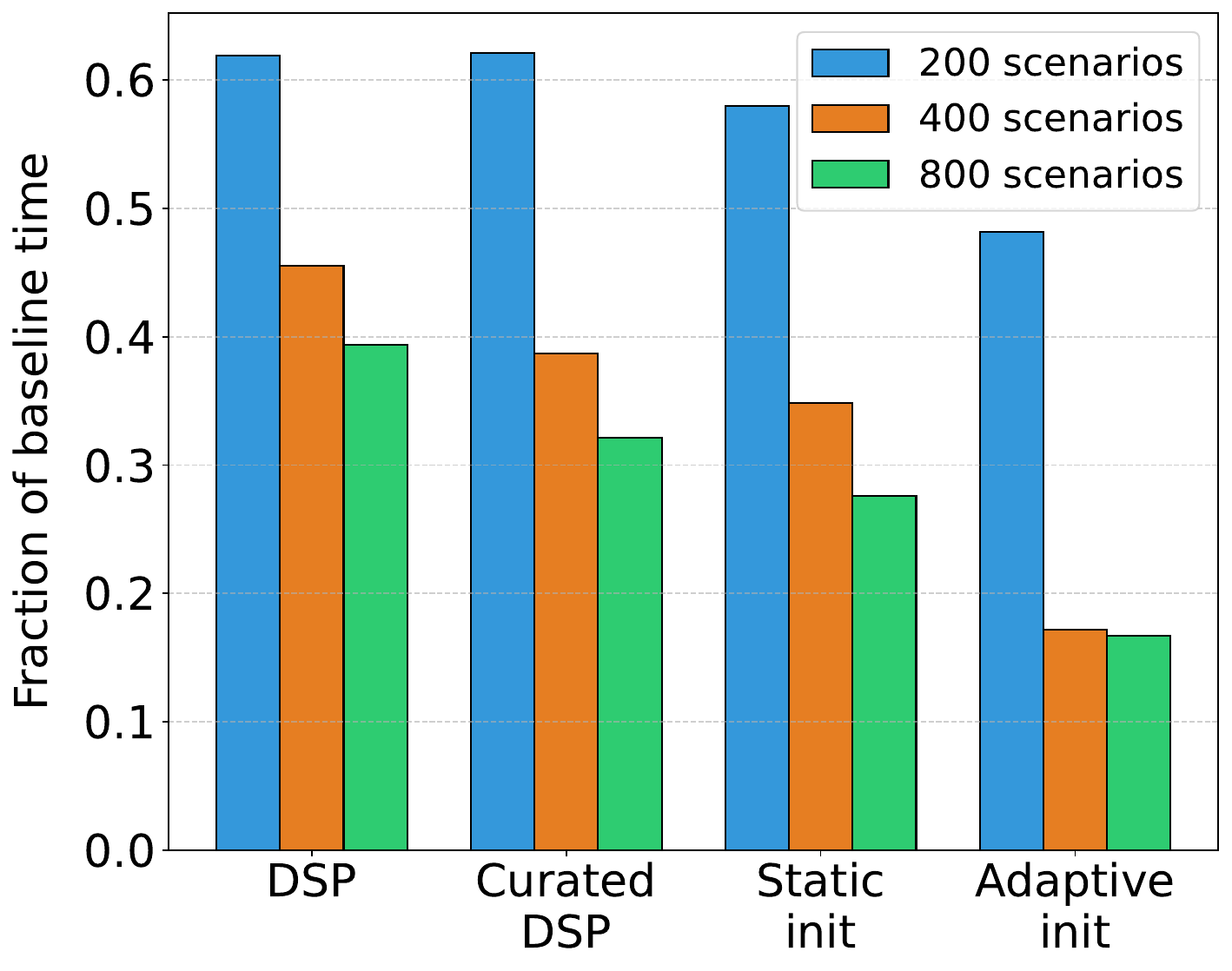}
    \caption{IP - CFLP}
  \end{subfigure}

  \vspace{0.3cm} 

  \begin{subfigure}{0.48\textwidth}
    \centering
    \includegraphics[width=\linewidth]{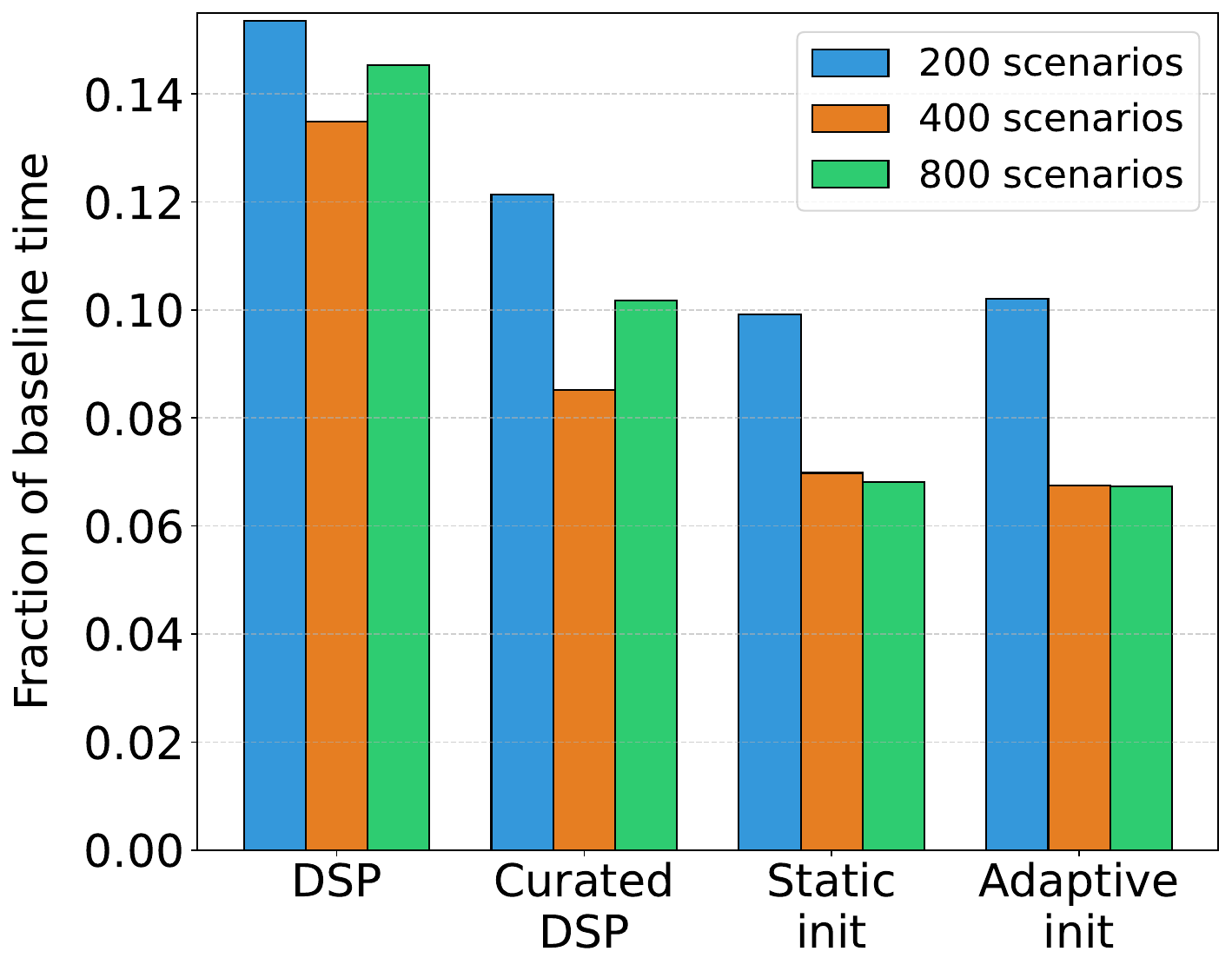}
    \caption{LP - CMND}
  \end{subfigure}
  \hspace{0.02\textwidth} 
  \begin{subfigure}{0.48\textwidth}
    \centering
    \includegraphics[width=\linewidth]{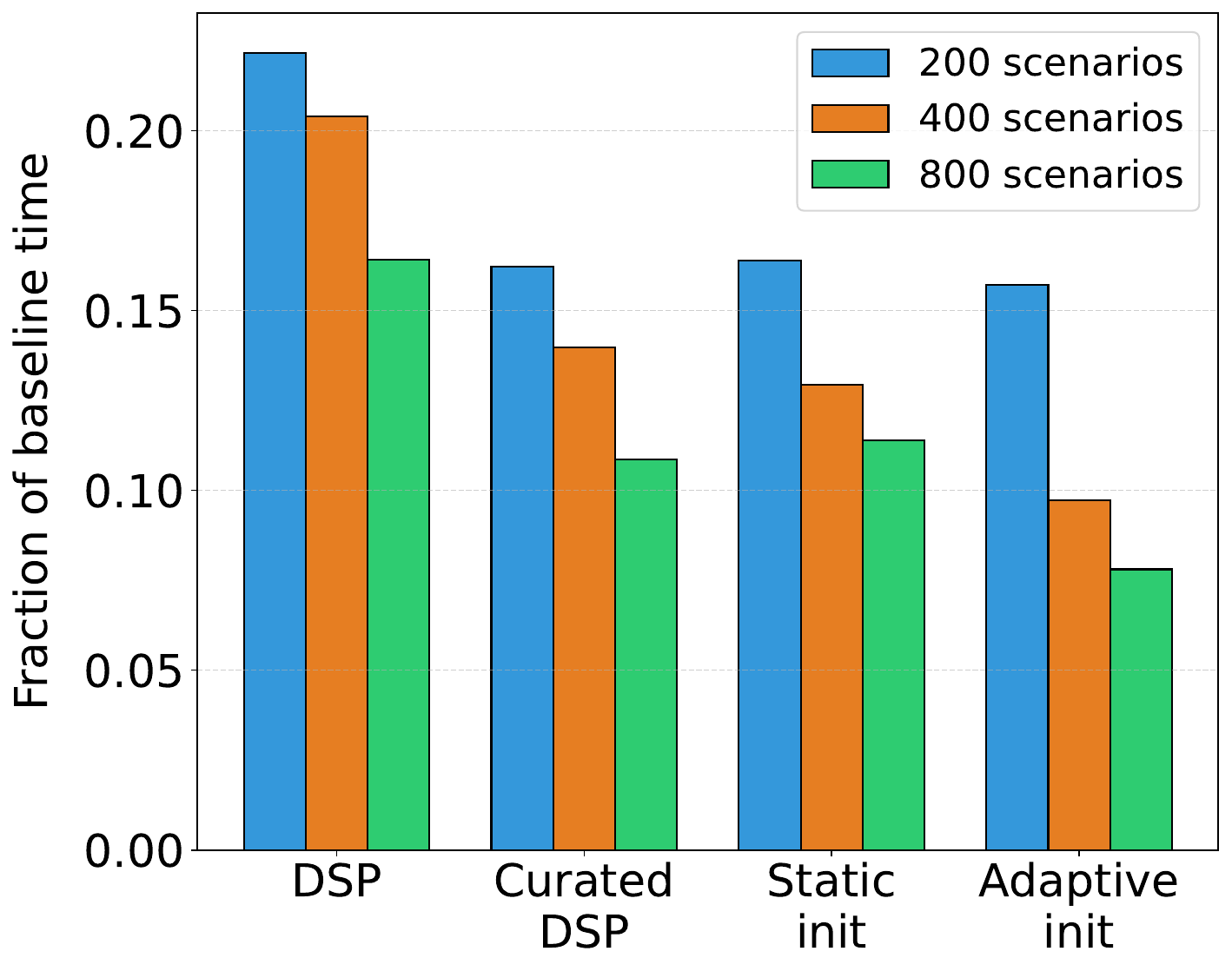}
    \caption{IP - CMND}
  \end{subfigure}

  \vspace{0.3cm} 

  \begin{subfigure}{0.48\textwidth}
    \centering
    \includegraphics[width=\linewidth]{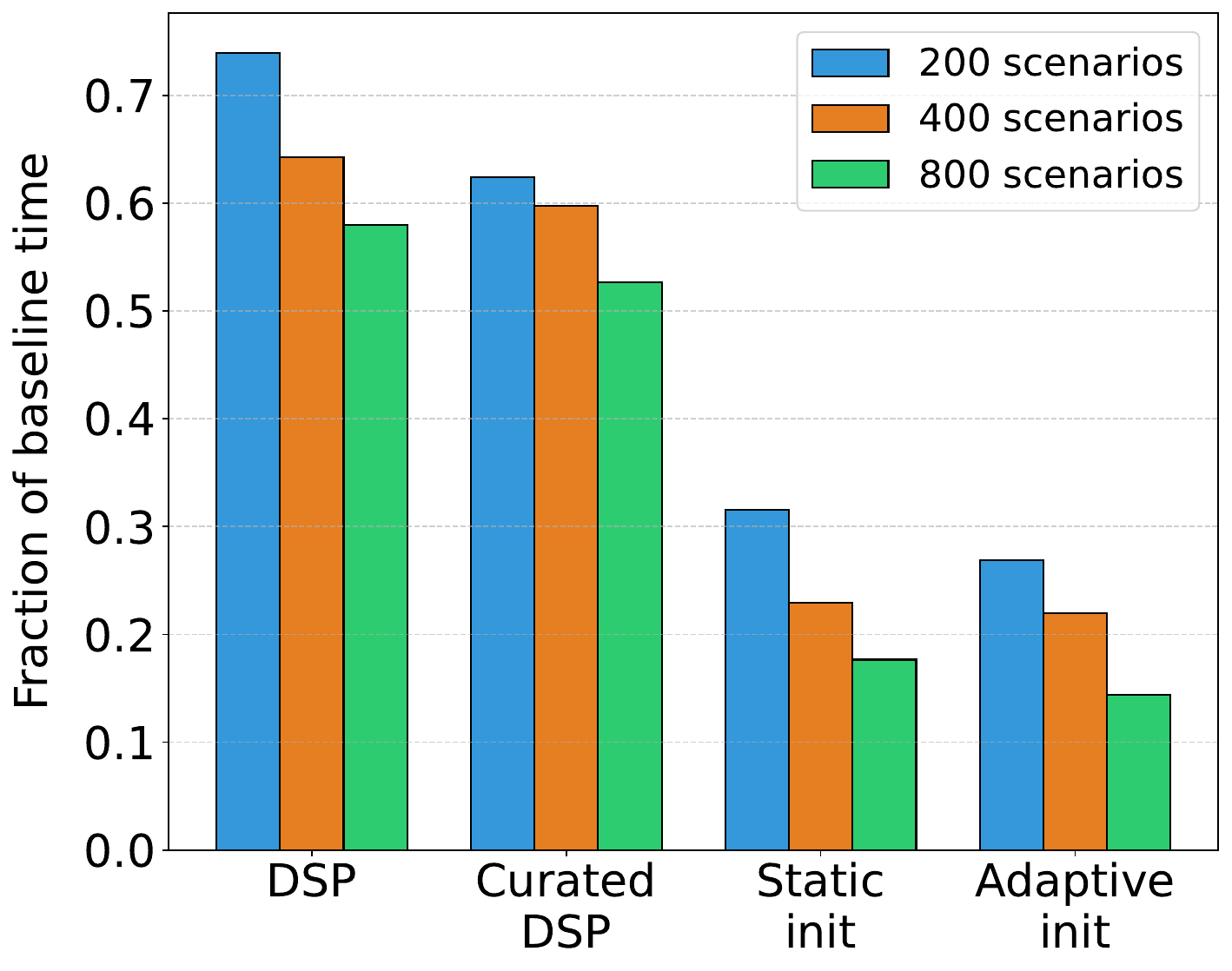}
    \caption{LP - UC}
  \end{subfigure}
  \hspace{0.02\textwidth} 
  \begin{subfigure}{0.48\textwidth}
    \centering
    \includegraphics[width=\linewidth]{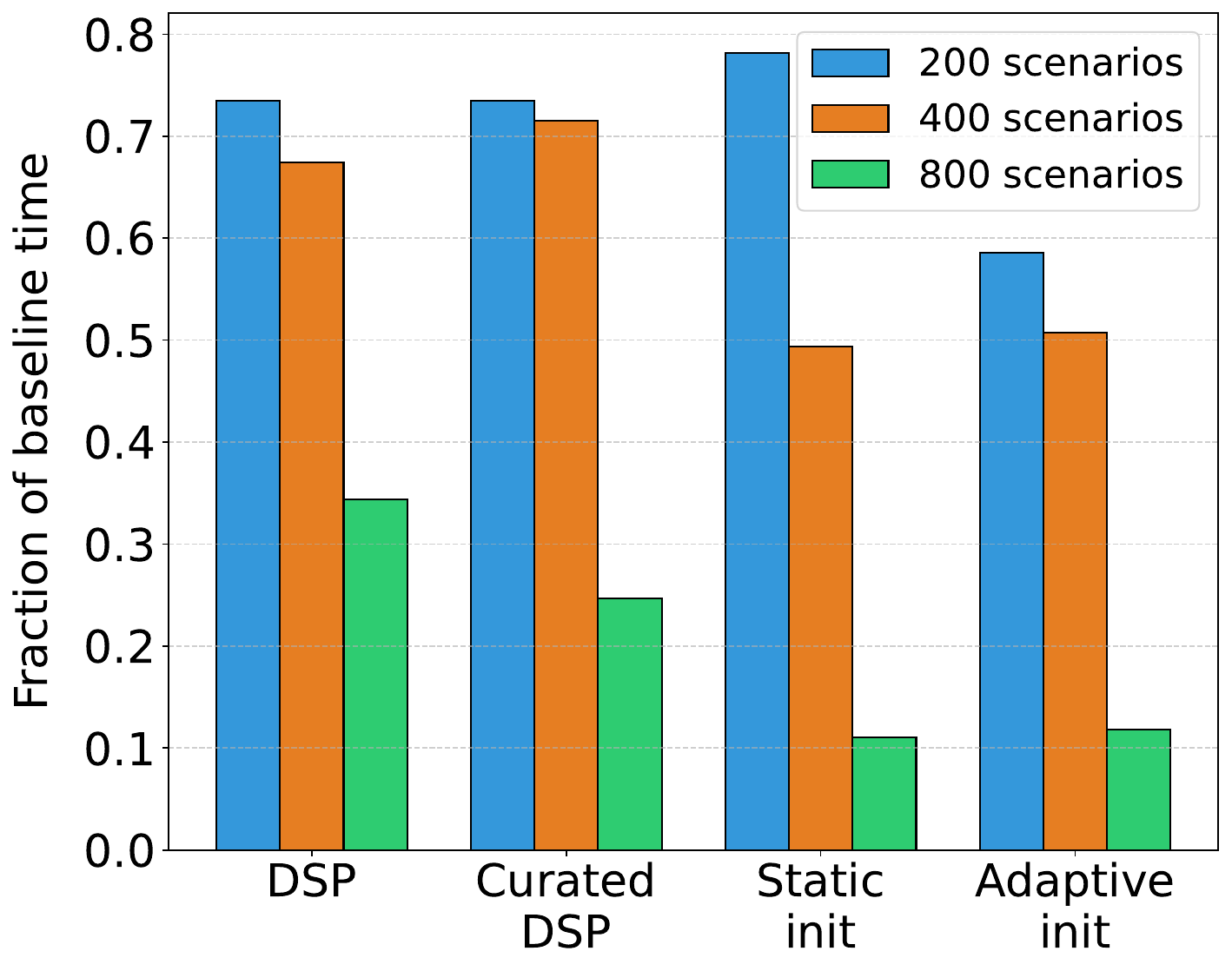}
    \caption{IP - UC}
  \end{subfigure}

  \caption{Comparison of performance of methods as number of scenarios vary.}
  \label{fig:scenario_variation}
\end{figure}

All the previous experiments have been done on instances in which the number of scenarios in each replication is 400. We next investigate the impact of varying the number of scenarios in the replications on our conclusions.

For this experiment, we select one instance for each problem class and then run our methods for that instance with the number of scenarios as 200, 400, and 800.
For network design,
we use instance r09.2 for the LP case and r03.3 for the IP case.
For facility location, we select the instance with 55 facilities
and 495 customers for LP, while the IP instance has 35 facilities and 105 customers.
For UC, we use an instance with 30 generators and difficulty 3.

Figure \ref{fig:scenario_variation}
displays the relative performance of our information reuse methods 
as a fraction of the execution time of the baseline method,
across three scenario sets.
We find that in general the relative performance of the different methods is similar across the different scenario sizes. We also find that the improvements from reusing information tend to be more significant for instances with more scenarios, which is not surprising considering that in such instances there is more work that has to be done in solving subproblems, and hence more opportunity to save time on that work with the proposed DSP and initialization techniques.

\subsection{Single-Cut Benders Decomposition}
\label{subsec:single_cut_experiments}

As described in Section~\ref{ssub:single-cut}, our proposed information reuse strategies can be adapted to the single-cut variant of Benders decomposition. To evaluate whether these adaptations remain effective in the single-cut setting, we conducted experiments on CFLP instances for both LP and IP variants using the same experimental setup as described previously, with the key difference being that we use single-cut Benders decomposition to solve each SAA replication instead of multi-cut.
In all experiments, we impose a time limit of 3600 seconds per replication. We do not report results for CMND or UC
instances with single-cut Benders decomposition, as our experiments found the single-cut method to be impractical for
these test instances.
For our computational experiments on the CFLP, we selected a subset of instances from Tables  
\ref{tab:cflp_ip_instances} and \ref{tab:cflp_lp_instances}. Specifically, for the LP variant, we tested instances with (warehouses, facilities) dimensions of (25, 305), (25, 495), and (55, 495), and for the IP variant, we tested instances with dimensions (15, 105), (15, 215), and (25, 95).


\begin{table}[]
\begin{center}
\begin{tabular}{lrrrrrrr}
\hline
Method       & \multicolumn{1}{c}{\someformat{Total T}} & \multicolumn{1}{c}{\someformat{Iterations}} & \multicolumn{1}{c}{\someformat{SP count}} & \multicolumn{1}{c}{\someformat{Cut T}} & \multicolumn{1}{c}{\someformat{DSP T}} & \multicolumn{1}{c}{\someformat{SP T}} & \multicolumn{1}{c}{\someformat{Init T}} \\
\hline
Baseline     & 621.7   & 374.1      & 374.1    & 621.1 & -     & 621.1 & -      \\
DSP          & 406.4   & 207.5      & 10.6     & 401.0 & 380.4 & 20.3  & -      \\
Curated DSP   & 138.1   & 172.9      & 12.8     & 132.8 & 106.9 & 24.4  & -      \\
Static Init   & 130.5   & 151.2      & 13.8     & 123.8 & 95.4  & 27.2  & 1.0    \\
Adaptive Init & 68.7    & 33.9       & 9.9      & 35.4  & 16.0  & 19.3  & 27.8  \\
\hline
\end{tabular}
\end{center}
\caption{Single-cut LP results: CFLP.}
\label{tab:single_lp_results_cflp}
\end{table}

Table~\ref{tab:single_lp_results_cflp} presents results for CFLP-LP instances using single-cut Benders decomposition.
The high iteration count in single-cut results in the discovery of a large number of dual solutions
throughout the solution process, causing the DSP to grow considerably. Consequently, the DSP
provides only marginal improvement over the baseline, as most of the time is spent searching through the
 large DSP.
In contrast, the curated DSP approach maintains a smaller set of high-quality dual solutions,
substantially reducing DSP search time.
This demonstrates the importance of curation when a large number of dual solutions are collected.
The initialization methods provide further improvements. Notably, adaptive initialization achieves a dramatic reduction in Benders iterations compared to the baseline, solving far fewer subproblems on average.

\begin{table}[]
\begin{center}
\begin{tabular}{lrrrrrrr}
\hline
Method       & \multicolumn{1}{c}{\someformat{Total T}} & \multicolumn{1}{c}{\someformat{IP T}} & \multicolumn{1}{c}{\someformat{LP T}} & \multicolumn{1}{c}{\someformat{Cut T}} & \multicolumn{1}{c}{\someformat{DSP T}} & \multicolumn{1}{c}{\someformat{SP T}} & \multicolumn{1}{c}{\someformat{Init T}} \\
\hline
Baseline     & 263.9   & 176.3 & 86.8 & 176.1 & -   & 176.1 & 0.4    \\
DSP          & 230.2   & 167.1 & 60.5 & 166.9 & 164.3 & 2.2   & 0.6    \\
Curated DSP   & 111.6   & 81.0  & 28.7 & 80.8  & 78.1  & 2.4   & 0.6    \\
Static Init   & 99.2    & 74.4  & 22.0 & 74.1  & 71.4  & 2.4   & 1.0    \\
Adaptive Init & 76.5    & 9.5   & 5.4  & 9.1   & 7.2   & 1.6   & 51.6  \\
\hline
\end{tabular}
\end{center}
\caption{Single-cut IP results: CFLP - Part 1.}
\label{tab:single_ip_results_cflp}
\end{table}
\begin{table}[h]
\begin{center}
\begin{tabular}{lrrrr}
\hline
Method       & \multicolumn{1}{c}{\someformat{Nodes}} & \multicolumn{1}{c}{\someformat{Root gap (\%)}} & \multicolumn{1}{c}{\someformat{Callback calls}} & \multicolumn{1}{c}{\someformat{SP count}} \\
\hline
Baseline      & 4233.2 & 22.9          & 276.1            & 276.1    \\
DSP          & 4020.5 & 16.8          & 272.6          & 3.2      \\
Curated DSP   & 3842.9 & 14.3          & 265.0          & 3.5      \\
Static Init   & 3518.5 & 9.1           & 252.9          & 3.5\\
Adaptive Init & 1642.7 & 7.7           & 25.1           & 2.3      \\
\hline
\end{tabular}
\end{center}
\caption{Single-cut IP results: CFLP - Part 2.}
\label{tab:single_ip_results_cflp_part2}
\end{table}

Table~\ref{tab:single_ip_results_cflp} and Table~\ref{tab:single_ip_results_cflp_part2} present complementary results for CFLP-IP instances using single-cut Benders decomposition.
The pattern observed for LP instances is replicated here: the uncurated DSP provides modest improvement, while curated DSP achieves more substantial gains. Static initialization further improves performance, and adaptive initialization again delivers the best results. While adaptive initialization requires a significant initialization investment, it dramatically reduces both the LP relaxation time and the IP phase time.

Table~\ref{tab:single_ip_results_cflp_part2} provides additional perspective through branch-and-bound metrics. The information reuse methods progressively improve the root gap, with adaptive initialization achieving the lowest root gap by adding high-quality cuts that effectively strengthen the problem formulation before branch-and-bound begins. This improvement in root gap quality translates directly to reduced node counts and callback invocations, with adaptive initialization exploring substantially fewer nodes and requiring far fewer callback invocations with minimal subproblem solves per callback.

These results demonstrate that our information reuse strategies remain highly effective for single-cut Benders
decomposition across both LP and IP problem variants when applied to this problem class.

\section{Conclusion and Future Directions}
We presented methods to accelerate solving a sequence of SAA replications in two-stage stochastic programming,
assuming randomness only in the right-hand sides of the subproblems.
These methods are derived for Benders decomposition as the solution algorithm and we find that, for our test problems, it is possible to reduce the time to solve the replications after the first one by a factor of 10 for both stochastic LP and IP problems by using the information reuse techniques we have proposed.

One significant direction for future research is to consider problems in which the subproblems have uncertainty in either the objective coefficients or the recourse matrix. Such problems do not have the property that the dual feasible region is fixed across scenarios, and hence the techniques we proposed do not extend directly to such problems. 

\sloppypar{
This paper focused on the Benders decomposition algorithm, as it is a leading algorithm for both two-stage stochastic LPs and for IPs with continuous recourse. Future research could investigate techniques for accelerating different algorithms in this context of solving a sequence of SAA replications. For example, for two-stage stochastic LPs, the level method \cite{lemarechal1995new,fabian2007solving} often performs better than Benders decomposition.
We anticipate that the techniques presented here would be useful for accelerating this and other cut-based decomposition methods, but testing this hypothesis would be an interesting direction for future work. It would also be interesting to explore methods to accelerate alternative methods for solving two-stage stochastic IPs, such as dual decomposition \cite{caroe1999dual} and methods that use different types of cuts \cite{boland2018combining,chen2022generating,gade2014decomposition,lulli2004branch,rahmaniani2017benders,van2020converging}.} 


\bibliographystyle{spmpsci}
\bibliography{saa-sequence}

\begin{thebibliography}{10}
\providecommand{\url}[1]{{#1}}
\providecommand{\urlprefix}{URL }
\expandafter\ifx\csname urlstyle\endcsname\relax
  \providecommand{\doi}[1]{DOI~\discretionary{}{}{}#1}\else
  \providecommand{\doi}{DOI~\discretionary{}{}{}\begingroup
  \urlstyle{rm}\Url}\fi

\bibitem{beasley1990or}
{OR}-library: distributing test problems by electronic mail.
\newblock Journal of the operational research society \textbf{41}(11),
  1069--1072 (1990)

\bibitem{adulyasak2015benders}
Adulyasak, Y., Cordeau, J.F., Jans, R.: Benders decomposition for production
  routing under demand uncertainty.
\newblock Operations Research \textbf{63}(4), 851--867 (2015)

\bibitem{bayraksan2006assessing}
Bayraksan, G., Morton, D.P.: Assessing solution quality in stochastic programs.
\newblock Mathematical Programming \textbf{108}, 495--514 (2006)

\bibitem{Benders1962}
Benders, J.: Partitioning procedures for solving mixed-variables programming
  problems.
\newblock Numerische Mathematik \textbf{4}, 238--252 (1962).
\newblock \urlprefix\url{https://doi.org/10.1007/BF01386316}

\bibitem{bengio2020learning}
Bengio, Y., Frejinger, E., Lodi, A., Patel, R., Sankaranarayanan, S.: A
  learning-based algorithm to quickly compute good primal solutions for
  stochastic integer programs.
\newblock In: Integration of Constraint Programming, Artificial Intelligence,
  and Operations Research: 17th International Conference, CPAIOR 2020, Vienna,
  Austria, September 21--24, 2020, Proceedings 17, pp. 99--111. Springer (2020)

\bibitem{bengio2021machine}
Bengio, Y., Lodi, A., Prouvost, A.: Machine learning for combinatorial
  optimization: a methodological tour d’horizon.
\newblock European Journal of Operational Research \textbf{290}(2), 405--421
  (2021)

\bibitem{birge2011introduction}
Birge, J.R., Louveaux, F.: Introduction to stochastic programming.
\newblock Springer Science \& Business Media (2011)

\bibitem{boland2018combining}
Boland, N., Christiansen, J., Dandurand, B., Eberhard, A., Linderoth, J.,
  Luedtke, J., Oliveira, F.: Combining progressive hedging with a
  {Frank}--{Wolfe} method to compute {Lagrangian} dual bounds in stochastic
  mixed-integer programming.
\newblock SIAM Journal on Optimization \textbf{28}(2), 1312--1336 (2018)

\bibitem{borghetti2003lagrangian}
Borghetti, A., Frangioni, A., Lacalandra, F., Nucci, C.A.: Lagrangian
  heuristics based on disaggregated bundle methods for hydrothermal unit
  commitment.
\newblock IEEE Transactions on Power Systems \textbf{18}(1), 313--323 (2003)

\bibitem{BOROZAN2024110985}
Borozan, S., Giannelos, S., Falugi, P., Moreira, A., Strbac, G.: Machine
  learning-enhanced benders decomposition approach for the multi-stage
  stochastic transmission expansion planning problem.
\newblock Electric Power Systems Research \textbf{237}, 110985 (2024).
\newblock \doi{https://doi.org/10.1016/j.epsr.2024.110985}.
\newblock
  \urlprefix\url{https://www.sciencedirect.com/science/article/pii/S0378779624008708}

\bibitem{caroe1999dual}
Car{\o}e, C.C., Schultz, R.: Dual decomposition in stochastic integer
  programming.
\newblock Operations Research Letters \textbf{24}(1-2), 37--45 (1999)

\bibitem{chan2023machine}
Chan, T., Lin, B., Saxe, S.: A machine learning approach to solving large
  bilevel and stochastic programs: Application to cycling network design.
\newblock Available at SSRN 4592562  (2023)

\bibitem{chen2022generating}
Chen, R., Luedtke, J.: On generating lagrangian cuts for two-stage stochastic
  integer programs.
\newblock INFORMS Journal on Computing \textbf{34}(4), 2332--2349 (2022)

\bibitem{cornuejols1991comparison}
Cornu{\'e}jols, G., Sridharan, R., Thizy, J.M.: A comparison of heuristics and
  relaxations for the capacitated plant location problem.
\newblock European journal of operational research \textbf{50}(3), 280--297
  (1991)

\bibitem{crainic2011progressive}
Crainic, T.G., Fu, X., Gendreau, M., Rei, W., Wallace, S.W.: Progressive
  hedging-based metaheuristics for stochastic network design.
\newblock Networks \textbf{58}(2), 114--124 (2011)

\bibitem{deza2023machine}
Deza, A., Khalil, E.B.: Machine learning for cutting planes in integer
  programming: A survey.
\newblock arXiv preprint arXiv:2302.09166  (2023)

\bibitem{dumouchelle2022neur2sp}
Dumouchelle, J., Patel, R., Khalil, E.B., Bodur, M.: Neur2sp: Neural two-stage
  stochastic programming.
\newblock arXiv preprint arXiv:2205.12006  (2022)

\bibitem{dupavcova2005melt}
Dupa{\v{c}}ov{\'a}, J., Popela, P.: Melt control: Charge optimization via
  stochastic programming.
\newblock In: Applications of stochastic programming, pp. 277--297. SIAM (2005)

\bibitem{dyer2006computational}
Dyer, M., Stougie, L.: Computational complexity of stochastic programming
  problems.
\newblock mathematical programming \textbf{106}, 423--432 (2006)

\bibitem{fabian2007solving}
F{\'a}bi{\'a}n, C.I., Sz{\H{o}}ke, Z.: Solving two-stage stochastic programming
  problems with level decomposition.
\newblock Computational Management Science \textbf{4}, 313--353 (2007)

\bibitem{fortz2009improved}
Fortz, B., Poss, M.: An improved benders decomposition applied to a multi-layer
  network design problem.
\newblock Operations research letters \textbf{37}(5), 359--364 (2009)

\bibitem{gade2014decomposition}
Gade, D., K{\"u}{\c{c}}{\"u}kyavuz, S., Sen, S.: Decomposition algorithms with
  parametric {Gomory} cuts for two-stage stochastic integer programs.
\newblock Mathematical Programming \textbf{144}(1-2), 39--64 (2014)

\bibitem{ghamlouche2003cycle}
Ghamlouche, I., Crainic, T.G., Gendreau, M.: Cycle-based neighbourhoods for
  fixed-charge capacitated multicommodity network design.
\newblock Operations research \textbf{51}(4), 655--667 (2003)

\bibitem{higle1991stochastic}
Higle, J.L., Sen, S.: Stochastic decomposition: An algorithm for two-stage
  linear programs with recourse.
\newblock Mathematics of operations research \textbf{16}(3), 650--669 (1991)

\bibitem{jia2021benders}
Jia, H., Shen, S.: Benders cut classification via support vector machines for
  solving two-stage stochastic programs.
\newblock INFORMS Journal on Optimization \textbf{3}(3), 278--297 (2021)

\bibitem{kim2015guide}
Kim, S., Pasupathy, R., Henderson, S.G.: A guide to sample average
  approximation.
\newblock Handbook of simulation optimization pp. 207--243 (2015)

\bibitem{kleywegt2002sample}
Kleywegt, A.J., Shapiro, A., Homem-de Mello, T.: The sample average
  approximation method for stochastic discrete optimization.
\newblock SIAM Journal on optimization \textbf{12}(2), 479--502 (2002)

\bibitem{knueven2020mixed}
Knueven, B., Ostrowski, J., Watson, J.P.: On mixed-integer programming
  formulations for the unit commitment problem.
\newblock INFORMS Journal on Computing \textbf{32}(4), 857--876 (2020)

\bibitem{lam2018bounding}
Lam, H., Qian, H.: Bounding optimality gap in stochastic optimization via
  bagging: Statistical efficiency and stability.
\newblock arXiv preprint arXiv:1810.02905  (2018)

\bibitem{larsen2023fast}
Larsen, E., Frejinger, E., Gendron, B., Lodi, A.: Fast continuous and integer
  l-shaped heuristics through supervised learning.
\newblock INFORMS Journal on Computing  (2023)

\bibitem{lee2020accelerating}
Lee, M., Ma, N., Yu, G., Dai, H.: Accelerating generalized benders
  decomposition for wireless resource allocation.
\newblock IEEE Transactions on Wireless Communications \textbf{20}(2),
  1233--1247 (2020)

\bibitem{lemarechal1995new}
Lemar{\'e}chal, C., Nemirovskii, A., Nesterov, Y.: New variants of bundle
  methods.
\newblock Mathematical programming \textbf{69}, 111--147 (1995)

\bibitem{lulli2004branch}
Lulli, G., Sen, S.: A branch-and-price algorithm for multistage stochastic
  integer programming with application to stochastic batch-sizing problems.
\newblock Management Science \textbf{50}(6), 786--796 (2004)

\bibitem{mak1999monte}
Mak, W.K., Morton, D.P., Wood, R.K.: Monte carlo bounding techniques for
  determining solution quality in stochastic programs.
\newblock Operations research letters \textbf{24}(1-2), 47--56 (1999)

\bibitem{mitrai2023learning}
Mitrai, I., Daoutidis, P.: Learning to initialize generalized benders
  decomposition via active learning.
\newblock FOCAPO/CPC, San Antonio, Texas  (2023)

\bibitem{mitrai2024computationally}
Mitrai, I., Daoutidis, P.: Computationally efficient solution of mixed integer
  model predictive control problems via machine learning aided benders
  decomposition.
\newblock Journal of Process Control \textbf{137}, 103207 (2024)

\bibitem{nair2018learning}
Nair, V., Dvijotham, D., Dunning, I., Vinyals, O.: Learning fast optimizers for
  contextual stochastic integer programs.
\newblock In: UAI, pp. 591--600 (2018)

\bibitem{naoum2013interior}
Naoum-Sawaya, J., Elhedhli, S.: An interior-point benders based branch-and-cut
  algorithm for mixed integer programs.
\newblock Annals of Operations Research \textbf{210}, 33--55 (2013)

\bibitem{rahmaniani2017benders}
Rahmaniani, R., Crainic, T.G., Gendreau, M., Rei, W.: The benders decomposition
  algorithm: A literature review.
\newblock European Journal of Operational Research \textbf{259}(3), 801--817
  (2017)

\bibitem{robbins1951stochastic}
Robbins, H., Monro, S.: A stochastic approximation method.
\newblock The Annals of Mathematical Statistics pp. 400--407 (1951)

\bibitem{saharidis2011initialization}
Saharidis, G.K., Boile, M., Theofanis, S.: Initialization of the benders master
  problem using valid inequalities applied to fixed-charge network problems.
\newblock Expert Systems with Applications \textbf{38}(6), 6627--6636 (2011)

\bibitem{sakhavand2022subproblem}
Sakhavand, N., Gangammanavar, H.: Subproblem sampling vs. scenario reduction:
  efficacy comparison for stochastic programs in power systems applications.
\newblock Energy Systems pp. 1--29 (2022)

\bibitem{sen1994network}
Sen, S., Doverspike, R.D., Cosares, S.: Network planning with random demand.
\newblock Telecommunication systems \textbf{3}(1), 11--30 (1994)

\bibitem{sen2016mitigating}
Sen, S., Liu, Y.: Mitigating uncertainty via compromise decisions in two-stage
  stochastic linear programming: Variance reduction.
\newblock Operations Research \textbf{64}(6), 1422--1437 (2016)

\bibitem{song2014chance}
Song, Y., Luedtke, J.R., K{\"u}{\c{c}}{\"u}kyavuz, S.: Chance-constrained
  binary packing problems.
\newblock INFORMS Journal on Computing \textbf{26}(4), 735--747 (2014)

\bibitem{van2020converging}
{van der Laan}, N., Romeijnders, W.: A converging {Benders’} decomposition
  algorithm for two-stage mixed-integer recourse models.
\newblock Operations Research \textbf{72}, 2190--2214 (2023)

\bibitem{van1969shaped}
Van~Slyke, R.M., Wets, R.: L-shaped linear programs with applications to
  optimal control and stochastic programming.
\newblock SIAM journal on applied mathematics \textbf{17}(4), 638--663 (1969)

\bibitem{wang2020statistical}
Wang, S., Gangammanavar, H., Ek{\c{s}}io{\u{g}}lu, S., Mason, S.J.: Statistical
  estimation of operating reserve requirements using rolling horizon stochastic
  optimization.
\newblock Annals of Operations Research \textbf{292}, 371--397 (2020)

\end{thebibliography}

\newpage
\begin{appendices}
\section{}
\label{sec:Appendix}
\begin{sloppypar}
We provide detailed descriptions of the problems considered 
in our  computational study.
We adopt the formulations
and problem descriptions from 
\cite{jia2021benders} for both test problems.
\end{sloppypar}
\subsection{Capacitated Facility Location Problem}
Consider a set $F$ of facilities,
where a facility $i \in F$ has a setup cost $f_i$ and a
production capacity limit $u_i$.
Additionally consider a set of customers $C$,
where each customer $j$ has an uncertain 
demand denoted by $\tilde{d}_j$. The vector of realizations of this uncertain demand in a scenario $k \in [K]$ is denoted by $d_k = [d_{kj}: j \in J]$.
This demand can be met by shipping from any open facility $i$
to customer $j$
at a unit transportation cost $c_{i j}$.
Any unmet demand for customer $j$ incurs a lost-sale penalty,
with a unit cost $\rho_j$.
The objective is to select a subset of facilities to open in 
order to minimize the total expected cost.

This problem is modeled as a two-stage stochastic programming problem.
The first-stage binary decision variables $x_i$
indicate whether the facility $i$ is opened or not.
In the second-stage, after revelation of the demands, 
we introduce continuous decision variables $y_{i j} \geq 0, \forall i \in F, j \in C$,
which represent goods transported from facility $i$ to customer $j$.
The model aims to find the best decisions to minimize the 
sum of facility setup cost,
expected transportation cost, and expected lost-sale cost.
The first-stage formulation is given by:
$$
\begin{aligned}
\min _x & \sum_{i \in F} f_i x_i+\sum_{\kink} p_k Q(x,d_k) \\
\text { s.t. } & x_i \in\{0,1\} \quad i \in F .
\end{aligned}
$$
The second-stage problem for each scenario $k$, $Q(x,d_k)$ is defined using 
transportation variables $y_{i j}$ from a facility $i$ to a 
customer $j$ and auxiliary variables $\alpha_j$
that denote the amount of unmet demand of customer $j$.
We have:
\begin{align*}
Q(x,d_k) = \min_{y, \alpha} \quad & \sum_{i \in F} \sum_{j \in C} c_{i j} y_{i j} + \sum_{j \in F} \rho_j \alpha_j \\
\text{subject to} \quad & \sum_{j \in C} y_{i j} \leq u_i x_i & \forall i \in F, \\
& d_{kj} - \sum_{i \in F} y_{i j} \leq \alpha_j & \forall j \in C, \\
& y_{i j} \geq 0 & \forall i \in F, j \in C, \\
& \alpha_j \geq 0 & \forall j \in C.
\end{align*}
By allowing unmet demand, the problem always has a feasible solution 
and Benders decomposition only requires optimality cuts.

\subsection{Multi Commodity Network Design Problem} 

Consider a directed network with node set $N$, arc set $A$, and commodity set $\commodities$.
Each commodity $\commodity$
has an uncertain demand $\tilde{v}_\commodity$ that 
must be routed from an origin node,
$o_\commodity \in N$, to its destination node, $d_\commodity \in N$.
The vector of demands in scenario $k \in [K]$ is denoted by
$v_{k} = [v_{k}^\commodity: \commodity \in \commodities]$.
For each arc
$(i, j) \in A$, there is an installation cost $f_{i j}$ 
and an arc capacity $u_{i j}$.
The cost for transporting one unit of commodity $\commodity$ on installed arc $(i, j)$ is $c_{i j}^\commodity$. 
Any demand that is not met is penalized at a rate of $B > 0$ per unit.

The objective in the first-stage is to decide
which subset of arcs to install to 
minimize the sum of arc installation cost and 
expected total transportation cost and penalty for unmet demand. 
In the second-stage, after demand is realized,
the goal is to determine the optimal flow of commodities 
through the installed arcs to minimize the sum of transportation and unmet demand costs.

In the first-stage, we define binary decisions $x_{i j}$ for all arcs
$(i, j) \in A$
such that $x_{i j}=1$ if we install arc $(i, j)$.
The first-stage formulation is given by:
$$
\begin{aligned}
 \min _x \sum_{(i, j) \in A} f_{i j} x_{i j}+\sum_{k \in [K]} p_k Q(x,v_k) \\
\text { s.t. } x_{i j} \in\{0,1\} \quad(i, j) \in A .
\end{aligned}
$$
In the second-stage,
we define non-negative continuous decisions 
$y_{i j}^\commodity$
to represent transportation units of commodity $\commodity$ 
on arc $(i, j)$.
Additionally, we introduce
auxiliary variables $\alpha^\commodity$ to denote the unmet demand 
for commodity $\commodity$. 
For scenario $k$, the formulation is given by:

\begin{align*}
Q(x,v_k) = \min_{y, \alpha} \quad & \sum_{\substack{(i, j) \in A}} \left[ \sum_{\substack{\commodity \in \commodities}} c_{i j}^\commodity y_{i j}^\commodity + B \alpha_i^\commodity \right] \\
\text{subject to} \quad & \sum_{\substack{j:(i, j) \in A}} y_{i j}^\commodity - \sum_{\substack{j:(j, i) \in A}} y_{j i}^\commodity  = g_{i}^{\commodity}(v_{k}^\commodity - \alpha^\commodity) \quad &\forall i \in N, \commodity \in \commodities, \\
                        & \sum_{\substack{\commodity \in \commodities}} y_{i j}^\commodity \leq u_{i j} x_{i j} \quad &\forall (i, j) \in A, \\
                        & y_{i j}^\commodity \geq 0 \quad &\forall (i, j) \in A, \commodity \in \commodities, \\
                        & \alpha_i^\commodity \geq 0 \quad &\forall i \in N, \commodity \in \commodities.
\end{align*}
The parameter $g_{i}^\commodity$ is set to $1$ if node $i$ is 
the origin of the commodity $\commodity,-1$ if node $i$ 
is the destination of the commodity $\commodity$, or 0 otherwise.
By allowing unmet demand, the problem always has a feasible solution and
Benders decomposition only requires optimality cuts.

\subsection{Stochastic Unit Commitment Problem}

Consider a set $\mathcal{G}$ of thermal generators operating over a discrete time horizon
$\mathcal{T} = \{1, \ldots, T\}$ with time periods indexed by $t$.
Each generator $g \in \mathcal{G}$ has power output bounds $\underline{P}_g, \overline{P}_g$ (minimum and maximum power output in MW when the generator is on),
temporal constraints $U_g$ (minimum up-time) and $D_g$ (minimum down-time) measured in number of periods,
and marginal production cost $C_g$ (\$/MWh).
The electricity demand at time $t$ in scenario $k \in [K]$ is denoted by $B_{tk}$ (MW).
To ensure feasibility of the subproblems,
we allow unmet demand
with penalty cost $M$ per unit.
The objective is to determine the commitment schedule (which generators are on or off at each time period) to minimize expected total cost.

The first-stage binary decision variables $x_{gt} \in \{0,1\}$
indicate the commitment status of generator $g$ at time $t$ (1 if on, 0 if off), for $g \in \mathcal{G}, t \in \{0, 1, \ldots, T\}$, where $x_{g0}$ represents the initial state at time 0.
Additionally, binary variables $w_{gt}$ and $v_{gt}$ for $g \in \mathcal{G}, t \in \mathcal{T}$
 track startup and shutdown events.
In the second-stage for scenario $k$, after revelation of demand,
continuous decision variables $p_{gtk} \geq 0, \forall g \in \mathcal{G}, t \in \mathcal{T}$,
represent power output from generator $g$ at time $t$,
and auxiliary variables $\sigma_{tk} \geq 0, \forall t \in \mathcal{T}$
to denote the unmet demand at time $t$.
The first-stage formulation is given by:
$$
\begin{aligned}
\min _x & \sum_{k \in [K]} p_k Q(x,B_k) \\
\text { s.t. } & x_{g0} = 0 \quad \forall g \in \mathcal{G}, \\
& x_{gt} = x_{g,t-1} + w_{gt} - v_{gt} \quad \forall g \in \mathcal{G}, t \in \mathcal{T}, \\
& \sum_{s=\max\{t-U_g+1, 1\}}^{t} w_{gs} \leq x_{gt} \quad \forall g \in \mathcal{G}, t \in \mathcal{T}, \\
& \sum_{s=\max\{t-D_g+1, 1\}}^{t} v_{gs} \leq 1 - x_{gt} \quad \forall g \in \mathcal{G}, t \in \mathcal{T}, \\
& x_{gt}, w_{gt}, v_{gt} \in \{0, 1\} \quad \forall g \in \mathcal{G}, t \in \mathcal{T}.
\end{aligned}
$$
The second-stage problem for each scenario $k$, $Q(x,B_k)$ is defined as:
\begin{align*}
Q(x,B_k) = \min_{p, \sigma} \quad & \sum_{g \in \mathcal{G}} \sum_{t \in \mathcal{T}} C_g p_{gtk} + \sum_{t \in \mathcal{T}} M \sigma_{tk} \\
\text{subject to} \quad & \underline{P}_g x_{gt} \leq p_{gtk} \leq \overline{P}_g x_{gt} & \forall g \in \mathcal{G}, t \in \mathcal{T}, \\
& \sum_{g \in \mathcal{G}} p_{gtk} + \sigma_{tk} \geq B_{tk} & \forall t \in \mathcal{T}, \\
& p_{gtk} \geq 0 & \forall g \in \mathcal{G}, t \in \mathcal{T}, \\
& \sigma_{tk} \geq 0 & \forall t \in \mathcal{T}.
\end{align*}
By allowing unmet demand, the problem has relatively complete recourse.

\end{appendices}

\exclude{
\section{}
\begin{align*}
\left[ 2 \exp\left( -\frac{\epsilon^2}{8\sigma^2D^2} \right) \right]^n \leq \rho \\
n\ln (2 \exp\left( -\frac{\epsilon^2}{8\sigma^2D^2} \right))  \geq \ln\rho \\
n\left(\ln 2  -\frac{\epsilon^2}{8\sigma^2D^2} \right)  \geq \ln\rho  \\
n \geq \frac{\ln \rho}{\left(\ln 2  -\frac{\epsilon^2}{8\sigma^2D^2} \right) } \text{if denominator is positive}
n \geq \frac{\ln \rho}{\left(\ln 2  -\frac{\epsilon^2}{8\sigma^2D^2} \right) } \text{if denominator is positive}
\end{align*}}

\end{document}